\documentclass[12pt]{amsart}
\usepackage{enumerate}
\usepackage{amssymb, amsthm, mathrsfs}
\usepackage{graphicx}
\usepackage{tikz-cd}
\usepackage{biblatex}
\usepackage{hyperref}

\newtheorem{thm}{Theorem}[section]
\newtheorem{lemma}[thm]{Lemma}
\newtheorem{prop}[thm]{Proposition}
\newtheorem{cor}[thm]{Corollary}

\theoremstyle{definition}
\newtheorem{definition}[thm]{Definition}
\newtheorem{notation}[thm]{Notation}

\theoremstyle{remark}
\newtheorem{remark}[thm]{Remark}
\newtheorem{example}[thm]{Example}

\newcommand{\bb}[1]{\mathbb{#1}}
\newcommand{\bs}[1]{\mathbf{#1}}
\newcommand{\mc}[1]{\mathcal{#1}}
\newcommand{\epslon}{\varepsilon}
\newcommand{\paren}[1]{\left( #1 \right)}
\newcommand{\gen}[1]{\left< #1 \right>}
\newcommand{\abs}[1]{\lvert #1 \rvert}

\newcommand{\clos}[1]{\overline{#1}}
\newcommand{\rcf}{R}

\newcommand{\orth}[1]{\mathcal{O} (#1)}
\newcommand{\s}[1]{\mathfrak{S}_{#1}}
\newcommand{\w}[1]{\mathcal{W}^{(#1)}}
\newcommand{\wo}[1]{\mathcal{W}^{(#1),\mathrm{o}}}

\newcommand{\fpos}[1]{\mathscr{F} (#1)}
\DeclareMathOperator{\Parr}{Par}
\newcommand{\Par}[2]{\Parr_{#1}({#2})}
\DeclareMathOperator{\Dim}{dim}

\DeclareMathOperator{\id}{id}
\DeclareMathOperator{\rank}{rank}

\DeclareMathOperator{\length}{length}
\DeclareMathOperator{\card}{card}
\DeclareMathOperator{\br}{br}

\begin{document}

\title{Equivariance in Approximation by Compact Sets}
\author{Saugata Basu}
\author{Alison Rosenblum}

\begin{abstract}
We adapt a construction of Gabrielov and Vorobjov for use in the symmetric case. Gabrielov and Vorobjov had developed a means by which one may replace an arbitrary set $S$ definable in some o-minimal expansion of $\bb{R}$ with a compact set $T$. $T$ is constructed in such a way that for a given $m>0$ we have epimorphisms from the first $m$ homotopy and homology groups of $T$ to those of $S$. If $S$ is defined by a boolean combination of statements $h(x)=0$ and $h(x)>0$ for various $h$ in some finite collection of definable continuous functions, one may choose $T$ so that these maps are isomorphisms for $0\leq k\leq m-1$. In this case, $T$ is also defined by functions closely related to those defining $S$.

In this paper we study sets $S$ symmetric under the action of some finite reflection group $G$. One may see that in the original construction, if $S$ is defined by functions symmetric relative to the action of $G$, then $T$ will be as well. We show that there is an equivariant map $T\rightarrow S$ inducing the aforementioned epimorphisms and isomorphisms of homotopy and homology groups. We use this result to strengthen theorems of Basu and Riener concerning the multiplicities of Specht modules in the isotypic decomposition of the cohomology spaces of sets defined by polynomials symmetric relative to $\s{n}$.
\end{abstract}

\maketitle

\section{Introduction}\label{sect:Introduction}
Fix an o-minimal structure expanding a real closed field $\rcf$ (on occasion, we will specify that this real closed field is $\bb{R}$; one may if desired assume as much throughout). We will take all sets, maps, etc. to be definable in this structure.

Gabrielov and Vorobjov in \cite{gabrielov2009approximation} develop a construction by which one may replace an arbitrary definable set $S$ by a closed and bounded set $T$, such that we have epimorphisms (and in a special case isomorphisms) from the homotopy and homology groups of the approximating set to those of the original set. In this special case, termed the \emph{constructible case}, we take $S$ to be defined by a quantifier-free formula with atoms $h>0$ or $h=0$ for some finite number of continuous definable functions $h$. In this case, the procedure described in \cite{gabrielov2009approximation} gives an explicit description of, and control over the number of, the functions defining $T$. Gabrielov and Vorobjov's paper describes the ramifications of this in calculating upper bounds on Betti numbers for sets definable in certain structures.

In \cite{basu2021vandermonde}, Basu and Riener leverage properties of symmetry (particularly, symmetry relative to the standard action of $\s{n}$ on $\rcf^n$) to study the Betti numbers of semialgebraic sets defined by symmetric polynomials of bounded degree, developing algorithms with favorable complexity bounds. The paper makes use of the construction of Gabrielov and Vorobjov presented in \cite{gabrielov2009approximation}, replacing $S$ with a set $T$ defined by closed conditions and then applying various results requiring closedness in order to determine the structure of the cohomology spaces of $T$. From Gabrielov and Vorobjob's results, we know that the cohomology spaces of $T$ are isomorphic to those of $S$. If we assume $S$ is defined by symmetric polynomials of degree bounded by some $d$ at least 2, then the set produced by the Gabrielov-Vorobjov construction is also defined by symmetric polynomials of degree no more than $d$. However, it is not immediately clear whether the maps on the level of homotopy and homology are equivariant. Without equivariance, we cannot conclude that the cohomology spaces of $S$ and $T$ have the same structures as $\s{n}$-modules, meaning that results for $T$ only translate back to $S$ in part.

This paper establishes an equivariant version of the construction of Gabrielov and Vorobjov in \cite{gabrielov2009approximation}. We introduce a few pieces of notation in order to present our main theorem.

\begin{notation}\label{def:P-set}
Let $S=\{\bs{x}\in \rcf^n\mid \mc{F}(\bs{x})\}$ be a set defined by a formula $\mc{F}$. Say $\mc{P}=\{h_1,\ldots, h_s\}$ is a finite collection of continuous definable functions $\rcf^n\rightarrow \rcf$. If $\mc{F}$ is a boolean combination of statements of the form $h=0$ and $h>0$ for $h\in \mc{P}$, we call $\mc{F}$ a \emph{$\mc{P}$-formula} and $S$ a \emph{$\mc{P}$-set}. If $\mathcal{F}$ is a monotone boolean combination (i.e. without negations) of statements of the form $h\geq 0$ and $h\leq 0$ for $h\in \mc{P}$, then we say that $\mc{F}$ is a \emph{$\mc{P}$-closed formula} and $S$ is a \emph{$\mc{P}$-closed set}.
\end{notation}

\begin{notation}[\cite{gabrielov2009approximation} Definition 1.7]\label{def:Gabrielov1.7}
When we say a statement holds for
\[
0<\epslon_0\ll \epslon_1\ll\cdots\ll \epslon_m \ll 1
\]
we mean that for each $0\leq i\leq m$ there is a definable function $f_i:(0,1)^{m-i}\rightarrow (0,1)$ (i.e. $f_m$ is a constant in $(0,1)$) such that the statement holds for all $\epslon_0,\ldots, \epslon_m$ which satisfy, for $0\leq i\leq m$, the condition that $0<\epslon_i<f_i(\epslon_{i+1},\ldots,\epslon_m)$.
\end{notation}

Our equivariant adaptation of the main theorem of Gabrielov and Vorobjov in \cite{gabrielov2009approximation} is then as follows.

\begin{thm}\label{thm:MainTheorem}
Let $G$ be a finite reflection group acting on $\bb{R}^n$, and let $S\subset \bb{R}^n$ be a definable set symmetric under the action of $G$. Choose an integer $m>0$.
\begin{enumerate}[(a)]
    \item (Definable Case) Say that $S$ is represented in a compact symmetric set $A$ by families $\{S_\delta\}_{\delta>0}$ and $\{S_{\delta, \epslon}\}_{\delta,\epslon>0}$ of compact sets (see Definition \ref{def:Gabrielov1.1}), and that each $S_\delta$ and $S_{\delta,\epslon}$ is also symmetric under the action of $G$. Then for parameters $0<\epslon_0,\delta_0,\ldots,\epslon_m,\delta_m<1$, there exists a compact symmetric set $T=T(\epslon_0,\delta_0,\ldots,\epslon_m,\delta_m)$ such that, for $0<\delta_0\ll\epslon_0\ll\ldots\ll \epslon_m\ll\delta_m\ll 1$, we have an equivariant map $\psi:T\rightarrow S$ inducing equivariant epimorphisms
    \begin{align*}
        \psi_{\#,k}&:\pi_k(T,*)\rightarrow \pi_k(S,*')\\
        \psi_{*,k}&:H_k(T)\rightarrow H_k(S)
    \end{align*}
    for each $1\leq k\leq m$.
    \item (Constructible Case) Say that $S$ is a $\mc{P}$-set for $\mc{P}=\{h_1,\ldots,h_s\}$ a collection of continuous definable functions having the property that $h\circ g\in \mc{P}$ for all $g\in G$ and $h\in \mc{P}$. For parameters $r>0$ and $0<\epslon_0,\delta_0,\ldots,\epslon_m,\delta_m<1$, let
    \[
    \mc{P}'=\bigcup_{h\in \mc{P}\cup \{r^2-(X_1^2+\cdots+ X_n^2)\}} \bigcup_{j=0}^m\{h\pm \epslon_j, h\pm \delta_j\}
    \]
    Then there exists a $\mc{P}'$-closed and bounded set $T$ such that for sufficiently large $r$ and $0<\delta_0\ll\epslon_0\ll\ldots\ll \epslon_m\ll\delta_m$, there exists an equivariant map $\psi:T\rightarrow S$ inducing equivariant homomorphisms
    \begin{align*}
        \psi_{\#,k}&:\pi_k(T,*)\rightarrow \pi_k(S,*')\\
        \psi_{*,k}&:H_k(T)\rightarrow H_k(S)
    \end{align*}
    which are isomorphisms for $1\leq k\leq m-1$ and epimorphisms for $k=m$. If $m\geq\dim(S)$, then $\psi$ induces a homotopy equivalence $T\simeq S$. Note that in particular if all functions in $\mc{P}$ are symmetric relative to $G$, then the same is true for $\mc{P}'$.
\end{enumerate}
\end{thm}

Parts (a) and (b) of this theorem are proved in Theorem \ref{thm:EquivariantGabrielov1.10i} and Corollary \ref{thm:EquivariantGabrielov1.10ii} respectively. Our arguments follow the strategy of \cite{gabrielov2009approximation}, with adjustments to ensure equivariance. The maps composed to obtain $\psi$ in the definable and constructible cases are displayed in the summary in Subsection \ref{sect:Summary}.

To construct the maps of homotopy and homology groups in \cite{gabrielov2009approximation}, Gabrielov and Vorobjov first consider a triangulation adapted to $S$ of some larger compact set $A$. Thus, a major step in our paper concerns proving the existence of a triangulation with the needed symmetry and equivariance properties for a given symmetric, definable set. This is done in Theorem \ref{thm:SymmetricTriangulation}. We also establish equivariant versions of a few other theorems as needed.

Section \ref{sect:Background} contains background and definitions. Section \ref{sect:GVConstruction} describes the construction of the approximating set $T$ in the definable and constructible cases. The results relating to symmetric triangulation and equivariant versions of the other background theorems can be found in Sections \ref{sect:Triangulation} and \ref{sect:EquivarianceTopology} respectively. Finally, in Section \ref{sect:ProofsOfEquivariance}, we assemble our results to prove the existence of an equivariant map $T\rightarrow S$ inducing the promised epimorphisms and isomorphisms of homotopy and homology groups. Section \ref{sect:Applications} discusses the ramifications for the results of Basu and Riener in \cite{basu2021vandermonde}.

The authors wish to express heartfelt appreciation for the advice of Dr. Gabrielov throughout, and particularly concerning Theorem \ref{thm:EquivariantNerveTheoremOpen}.

\section{Background}\label{sect:Background}
Much of the information here is relatively standard. We include these definitions for completeness, and to establish a few conventions.

\subsection{Symmetry}\label{sect:SymmetryBackground}

\begin{definition}
Let $G$ be a group acting on a set $X$. A subset $Y\subset X$ is said to be \emph{symmetric} with respect to the action of $G$ on $X$ if for each $g\in G$ and $y\in Y$, we have $g(y)\in Y$.
\end{definition}

We extend the definition of a symmetric function, standard for the usual action of $\s{n}$ on $\bb{R}^n$, to more general group actions.

\begin{definition}\label{def:Symmetry}
Let $f:X\rightarrow Y$ for sets $X$ and $Y$, and let $G$ be a group acting on $X$. We will say that $f$ is \emph{symmetric} with respect to the action of $G$ if we have that $f(g(x))=f(x)$ for all $x\in X$ and $g\in G$.
\end{definition}

\begin{definition}\label{def:Equivariance}
Let $f:X\rightarrow Y$ for some sets $X$ and $Y$, and let $G$ be a group which acts on both $X$ and $Y$. We say $f$ is \emph{equivariant} with respect to the action of $G$ if for all $g\in G$, the diagram
\[
\begin{tikzcd}
X \arrow[d, "g"] \arrow[r, "f"] & Y \arrow[d, "g"]\\
X \arrow[r, "f"] & Y
\end{tikzcd}
\]
commutes.
\end{definition}

We do need to give consideration to symmetry and equivariance for homotopy groups. Let $X$ be a topological space, with $G$ a group acting on $X$. Then for $g\in G$ and basepoint $x_0\in X$, the map $g:X\rightarrow X$ induces maps $g_{\#,k}:\pi_k(X,x_0)\rightarrow \pi_k(X,g(x_0))$ of homotopy groups for each $k\geq 0$. Thus, for a pointed space $(X,*)$, we understand that the induced action of $g$ on any $\pi_k(X,*)$ will also change the basepoint. For this reason, we will be careful to maintain basepoint notation in reference to homotopy groups. In particular, we must give attention to our basepoint when considering equivariance of maps on the level of homotopy groups.

\begin{definition}\label{def:HomotopyEquivariance}
Let $f:\pi_k(X,*)\rightarrow \pi_k(Y,*')$ for topological spaces $(X,*)$ and $(Y,*')$, and let $G$ be a group acting on $X$ and $Y$. We will say that $f$ is equivariant if the diagram
\[
\begin{tikzcd}
\pi_k(X,*) \arrow[d, "g"] \arrow[r, "f"] &  \pi_k(Y,*')\arrow[d, "g"]\\
\pi_k(X,g(*)) \arrow[r, "f"] & \pi_k(Y, g(*'))
\end{tikzcd}
\]
commutes.
\end{definition}

\subsection{Reflection Groups}\label{sect:ReflectionBackground}

Let $V$ be a finite dimensional real vector space equipped with an inner product, and let $\orth{V}$ denote the group of all orthogonal linear transformations of $V$. We will for the sake of clarity distinguish a \emph{linear hyperplane} (a codimension $1$ vector subspace, which must therefore contain $\bs{0}$) from an \emph{affine hyperplane} (a codimension $1$ affine subspace, which need not pass through the origin). If $P$ is a hyperplane (in either sense), then the complement $V\setminus P$ has two connected components, which are called (linear or affine, resp.) half-spaces.

\begin{definition}[see \cite{grove1996finite} Chapter 3]
Let $G$ be a finite subgroup of $\orth{V}$. A subset $H\subset V$ is a \emph{fundamental region} of $V$ relative to $G$ provided
\begin{enumerate}
    \item $H$ is an open subset of $V$
    \item $H\cap g(H)=\emptyset$ for all $e\neq g\in G$
    \item $V=\bigcup_{g\in G} \clos{g(H)}$
\end{enumerate}
\end{definition}

\begin{thm}[\cite{grove1996finite} Theorem 3.1.2]
Let $G$ be a finite subgroup of $\orth{V}$. Then there exists a finite collection $\{H_1,\ldots, H_k\}$ of linear half-spaces such that the region
\[
H=\bigcap_{i=1}^k H_i
\]
is a fundamental region of $V$ relative to $G$.
\end{thm}

We are interested in the case where $G$ is a finite reflection group acting on $\bb{R}^n$ (with the standard inner product), i.e. $G$ is generated by elements $g$ whose action on $\bb{R}^n$ are given by $g(\bs{x})$ reflecting $\bs{x}$ across some linear hyperplane $P_g$. We will frequently use the fact that any affine hyperplane $P$ in $\rcf^n$ coincides with some $\{\bs{x}\in \rcf^n\mid L(\bs{x})=0\}$, where $L:\rcf^n\rightarrow\rcf$ is of the form $L(x_1,\ldots,x_n)=a_0+a_1x_1+\cdots+a_nx_n$ for some $a_0,\ldots, a_n\in \rcf$. $P$ is a linear hyperplane iff $a_0=0$.

Say that $G$ is a finite reflection group acting on $\bb{R}^n$. Then, as detailed in \cite{grove1996finite} Chapter 4, we can select a subset $\{g_1,\ldots, g_k\}$ of non-identity elements of $G$ such that, if $g_i$ acts by reflecting through the hyperplane $P_i$, we can choose functions $L_i$ with $P_i$ given by $L_i(\bs{x})=0$ so that
\[
H=\bigcap_{i=1}^k \{\bs{x}\in \bb{R}^n\mid L_i(\bs{x})>0\}
\]
is a fundamental region of $\bb{R}^n$ relative to $G$. The sets of the form $(L_1\geq 0)\cap \ldots \cap (L_{i-1}\geq 0)\cap (L_i=0)\cap (L_{i+i}\geq 0)\cap\ldots\cap (L_k\geq 0)$ are called the $\emph{walls}$ of this fundamental region.

In this description, for $g\in G$ we have that there exists a subset $\lambda_g$ of $\{1,\ldots,k\}$ such that

\[
g(\clos{H})=\paren{\bigcap_{i\in \lambda_g}L_i(\bs{x})\leq 0}\cap\paren{\bigcap_{i\in \{1,\ldots,k\}\setminus \lambda_g} L_i(\bs{x})\geq 0}
\]

Then $\clos{H}\cap g(\clos{H})$ is given by
\[
H_{\lambda_g}=\paren{\bigcap_{i\in \lambda_g}L_i(\bs{x})=0}\cap\paren{\bigcap_{i\in \{1,\ldots,k\}\setminus \lambda_g} L_i(\bs{x})\geq 0}
\]
which is an intersection of walls of $H$. $H_{\lambda_g}$ is also the set of points of $H$ fixed by the action of $g$.

\begin{example}
In our primary example, $G=\s{n}$ acts on $\bb{R}^n$ via the standard action given by the permutation of coordinates, i.e. for $\sigma\in\s{n}$ and $\bs{x}=(x_1,\ldots,x_n)\in\bb{R}^n$ we have $\sigma(\bs{x})=(x_{\sigma(1)},\ldots, x_{\sigma(n)})$.
Here, we may take as a fundamental region the interior of the Weyl chamber
\[
\wo{n}=\{(x_1,\ldots, x_n)\in \bb{R}^n\mid x_1<\ldots<x_n\}
\]
(see \cite{basu2021vandermonde} Notation 9). The walls of the Weyl chamber correspond to the adjacent transpositions $s_i=(i\; i+1)\in \s{n}$ (i.e. the standard Coxeter generators of $\s{n}$). Each $s_i$ acts by reflecting through the hyperplane given by $x_i=x_{i+1}$, so we have that the walls of the Weyl chamber are the sets
\[
\w{n}_{s_i}=\{(x_1,\ldots, x_n)\in \bb{R}^n\mid x_1\leq\ldots\leq x_i=x_{i+1}\leq \ldots\leq x_n\}
\]
for $1\leq i\leq n-1$ (see \cite{basu2021vandermonde} Notation 11).
\end{example}

\subsection{Simplicial Complexes}\label{SimplicialComplexBackground}

Our proofs make reference to both abstract and concrete simplicial complexes. Accordingly, we present both viewpoints here, beginning with the concrete setting.

\begin{definition}[see \cite{coste2000introduction} Section 4.2]\label{def:SimplexInR}
Let the points $v_0,\ldots, v_d\in \rcf^n$ be affine independent. Then the \emph{open simplex} (of dimension $d$) with vertices $v_0,\ldots, v_d$ is the set
\begin{align*}
\Delta(v_0,\ldots, v_d)=\{x&\in \rcf^n\mid x=t_1v_1+\cdots+t_dv_d\\
&\text{ for some } t_1,\ldots,t_d\in (0,1] \text{ with } t_1+\cdots+t_d=1\}
\end{align*}
The corresponding \emph{closed simplex} is
\begin{align*}
\clos{\Delta}(v_0,\ldots, v_d)=\{x&\in \rcf^n\mid x=t_1v_1+\cdots+t_dv_d\\
&\text{ for some } t_1,\ldots,t_d\in [0,1] \text{ with } t_1+\cdots+t_d=1\}
\end{align*}
A \emph{face} of a simplex $\clos{\Delta}(v_0,\ldots,v_d)$ is a simplex $\clos{\Delta}(u_0,\ldots, u_{d'})$ with vertex set $\{u_0,\ldots,u_{d'}\}\subset \{v_0,\ldots,v_d\}$.
\end{definition}

If we are drawing our vertices from an indexed set $\{v_i\}_{i\in I}$, we may simply refer to a simplex $\Delta(v_{i_0},\ldots,v_{i_d})$ as $\Delta(i_0,\ldots, i_d)$.

\begin{definition}[see \cite{coste2000introduction} Section 4.2]\label{def:SimplicialComplexInR}
A (finite) \emph{simplicial complex} in $\rcf^n$ is a finite collection $\Lambda$ of open simplices in $\rcf^n$ with the following properties
\begin{enumerate}
    \item If $\Delta\in \Lambda$, then for each face $\clos{\Delta}'$ of $\clos{\Delta}$, we have that $\Delta'\in \Lambda$
    \item If $\Delta_1,\Delta_2\in \Lambda$, then $\clos{\Delta}_1\cap \clos{\Delta}_2=\clos{\Delta}_3$ for some $\Delta_3\in \Lambda$.
\end{enumerate}
Given a simplicial complex $\Lambda$ in $\rcf^n$, we call $\bigcup_{\Delta\in \Lambda}\Delta\subset \rcf^n$ the \emph{geometric realization} of $\Lambda$, and denote it $\abs{\Lambda}$.
\end{definition}

If not clear from context, we may include a subscript to indicate where the realization is taking place. For example, say $\Lambda$ is a simplicial complex in $\rcf^n$ seen as a subset of $\rcf^n\times \rcf^m$. Then we may wish to distinguish between $\abs{\Lambda}_{\rcf^n}$ and $\abs{\Lambda}_{\rcf^{n+m}}$, the realizations of $\Lambda$ in $\rcf^n$ and $\rcf^{n+m}$ respectively.

\begin{definition}
Let $\Delta,\Delta'\in \Lambda$ for some simplicial complex $\Lambda$. Then $\Delta'$ is a \emph{subsimplex} of $\Delta$ if $\Delta'\neq \Delta$ and $\Delta'\subset \clos{\Delta}$ (in other words, if $\clos{\Delta}'$ is a proper face of $\clos{\Delta}$).
\end{definition}

\begin{definition}\label{def:flag}
A \emph{$k$-flag} of cells in a CW complex (so in particular, of simplices in a simplicial complex) is a sequence $\sigma_0,\ldots, \sigma_k$ of cells such that $\sigma_i$ is contained in the boundary of $\sigma_{i-1}$ for each $1\leq i\leq k$.
\end{definition}

When speaking of symmetry for any CW complex (and in particular the geometric realization of a simplicial complex $\Lambda$ in $\bb{R}^n$), we will in general consider the action of $G$ on our decomposition to be that induced by a given action of $G$ on our larger space:

\begin{definition}\label{def:SymmetricSimplicialComplex}
Let $G$ be a group acting on $\bb{R}^n$. A CW complex $X$ is \emph{symmetric} with respect to the action of $G$ if for each cell $\sigma\in \Lambda$ and each $g\in G$, $g(\sigma)=\{g(x)\mid x\in \sigma\}$ is again a cell of $X$.
\end{definition}

In particular, this means that $X$ is symmetric as a subset of $\bb{R}^n$.

Now, we shift our attention to abstract simplicial complexes.

\begin{definition}[from \cite{spanier1989algebraic}]
An \emph{abstract simplicial complex} is a set $V$ of vertices and a set $\Lambda$ of finite nonempty subsets of $V$ which we consider to be simplices, having the properties
\begin{enumerate}
    \item $\{v\}\in \Lambda$ for each $v\in V$ (each vertex is a simplex)
    \item any nonempty subset of a simplex is a simplex
\end{enumerate}
\end{definition}

One can and often does ignore the distinction between an abstract simplicial complex and its set of simplices. An abstract simplicial complex also comes with a geometric realization, which we will explicitly describe for the sake of later use.

\begin{definition}[see \cite{spanier1989algebraic}]
Let $\Lambda$ be a nonempty abstract simplicial complex. The \emph{geometric realization} of $\Lambda$, denoted by $\abs{\Lambda}$, is the space whose points are functions $\alpha$ from the set of vertices of $\Lambda$ to the interval $[0,1]\subset \bb{R}$ such that
\begin{enumerate}
    \item The set $\{v\in \Lambda\mid \alpha(v)\neq 0\}$ is a simplex of $\Lambda$
    \item $\sum_{v\in \Lambda}\alpha(v)=1$
\end{enumerate}
appropriately topologized (as described in \cite{spanier1989algebraic}).
\end{definition}

It will be convenient to refer to the point $\alpha$ using the notation $\sum_{v\in \Lambda} \alpha(v) v$.

\begin{definition}[from \cite{spanier1989algebraic} Chapter 3 Section 2]
Let $\Lambda$ be any simplicial complex and let $X$ be a topological space which is a subset of a real vector space, appropriately topologized (see \cite{spanier1989algebraic} for details; in particular euclidean space and geometric realizations of simplicial complexes meet the criteria). A continuous map $f:\abs{\Lambda}\rightarrow X$ is said to be \emph{linear} (on $\Lambda$) if for $\alpha\in \abs{\Lambda}$, we have
\[
f(\alpha)=\sum_{v\in \Lambda} \alpha(v) f(v)
\]
\end{definition}

Following \cite{bredon1972introduction}, we will say an abstract simplicial complex is symmetric with respect to a group $G$ acting on its set of vertices if the map given by $g$ is simplicial (i.e. $g$ carries simplices to simplices) for each $g\in G$. An action on $\Lambda$ induces a linear action on $\abs{\Lambda}$: namely, $g(\alpha)$ is the function from $\Lambda$ to $[0,1]$ given by $g(\alpha)(v)=\alpha(g(v))$.

Any simplicial complex $\Lambda$ comes with a \emph{face poset} $\fpos{\Lambda}$, where the simplicies of $\Lambda$ are ordered by $\sigma_1\leq \sigma_2$ if $\clos{\sigma_1}$ is a face of $\clos{\sigma_2}$. For $X$ a CW complex, we will still use the notation $\fpos{X}$ for the cell poset of $X$ (which has as points the cells $\sigma$ of $X$ and order given by $\sigma_1\leq \sigma_2$ if $\sigma_1\subset \clos{\sigma_2}$).

\subsection{Miscellaneous Background}\label{MiscellaneousBackground}

The following notions will appear later in the paper.

\begin{definition}[from \cite{van1998tame} Chapter 9 Section 2]
Let $A\subset \rcf^n$ and $x\in \rcf^n\setminus A$. Then the \emph{cone} with vertex $x$ and base $A$ is the set
\[
\{ta+(1-t)x\mid a\in A, t\in [0,1]\}
\]
(the union of all line segments from $x$ to a point in $A$).
\end{definition}

\begin{definition}[from \cite{coste2000introduction} Section 3.2]
We say a family $\{S_x\}_{x\in\rcf^m}$ of subsets of some real closed field $\rcf^n$ is a \emph{definable family} (with respect to some o-minimal structure on $\rcf$) if there is a definable set $S'\subset \rcf^{m+n}$ such that $S_x$ is equal to $\{y\in \rcf^n\mid (x,y)\in S'\}$ for each $x\in \rcf^m$. It follows that each $S_x$ is a definable set.
\end{definition}

The reader unfamiliar with o-minimal geometry is also encouraged to review the cell decomposition theorem (\cite{van1998tame} Chapter 3 Section 2 or \cite{coste2000introduction} Section 2.2) before approaching our equivariant triangulation proofs in Section \ref{sect:Triangulation}.

For a connected topological space $X$, we let $\pi_k(X)$ denote the $k$th homotopy group. Let $H_k(X)$ be the $k$th singular homology group with coefficients in some fixed Abelian group, and denote by $b_k(X)=\rank(H_k(X))$ the $k$th Betti number of $X$. Throughout, we will use $\simeq$ to denote homotopy equivalence and $\cong$ to denote group isomorphism.

The following two theorems are used heavily in Gabrielov and Vorobjov's proofs in \cite{gabrielov2009approximation} and referenced also in our own arguments. These do not require equivariant versions; if a map $f:X\rightarrow Y$ is equivariant, the induced homomorphisms of homotopy and homology groups will be as well.

\begin{thm}[Whitehead Theorem on Weak Homotopy Equivalence, \cite{spanier1989algebraic} 7.6.24]\label{thm:Spanier7.6.24}
A map $f:X\rightarrow Y$ between connected CW complexes is a weak homotopy equivalence (i.e. the induced homomorphism of homotopy groups $f_{\#,k}:\pi_k(X)\rightarrow \pi_k(Y)$ is an isomorphism for each $k>0$) iff $f$ is a homotopy equivalence.
\end{thm}

\begin{thm}[Whitehead Theorem on Homotopy and Homology, \cite{spanier1989algebraic} 7.5.9]\label{thm:Spanier7.5.9}
Let $f:X\rightarrow Y$ be a continuous map between path connected topological spaces. If there is a $k>0$ such that the induced homomorphism of homotopy groups $f_{\#,j}:\pi_j(X)\rightarrow \pi_j(Y)$ is an isomorphism for $j<k$ and an epimorphism for $j=k$, then the induced homomorphism of homology groups $f_{*,j}:H_j(X)\rightarrow H_j(Y)$ is an isomorphism for $j<k$ and an epimorphism for $j=k$.
\end{thm}

\section{The Gabrielov-Vorobjov Construction}\label{sect:GVConstruction}

In this section, we describe the construction of the approximating set $T$, as given by Gabrielov and Vorobjov in \cite{gabrielov2009approximation}. We also discuss the implications when symmetry is introduced to the construction.

In order to construct $T$, we begin with families of compact sets $\{S_\delta\}$ and $\{S_{\delta,\epslon}\}$ which represent $S$ in some larger compact set $A$.

\begin{definition}[\cite{gabrielov2009approximation} Definition 1.1]\label{def:Gabrielov1.1}
Let $S\subset \bb{R}^n$ be definable, and let $A\subset \bb{R}^n$ be a compact definable set with $S\subset A$. Let $\{S_\delta\}_{\delta>0}$ and $\{S_{\delta,\epslon}\}_{\delta,\epslon>0}$ be definable families of compact subsets of $A$. We say $S$ is \emph{represented} by $\{S_{\delta}\}_{\delta>0}$ and $\{S_{\delta,\epslon}\}_{\delta,\epslon>0}$ in $A$ if we have that
\begin{enumerate}[(a)]
    \item for all $\delta',\delta\in (0,1)$, if $\delta'>\delta$, then $S_{\delta'}\subset S_\delta$
    \item $S=\bigcup_{\delta>0} S_\delta$
\end{enumerate}
and furthermore for each $\delta>0$
\begin{enumerate}[(i)]
    \item for all $\epslon',\epslon\in (0,1)$, if $\epslon'>\epslon$, then $S_{\delta,\epslon}\subset S_{\delta,\epslon'}$
    \item $S_\delta=\bigcap_{\epslon>0} S_{\delta,\epslon}$
    \item for all $\delta'$ sufficiently smaller than $\delta$ and for all $\epslon'>0$ there exists a set $U\subset A$ with $U$ open in $A$ and $S_\delta\subset U\subset S_{\delta',\epslon'}$
\end{enumerate}
\end{definition}

We then take $T$ to be the union of a selection of finitely many of the sets $S_{\delta, \epslon}$.

\begin{definition}[\cite{gabrielov2009approximation} Definition 1.8]\label{def:Gabrielov1.8}
For any nonnegative integer $m$ and parameters $\epslon_0,\delta_0,\epslon_1,\delta_1,\ldots,\epslon_m,\delta_m$, we denote
\[
T=T(\epslon_0,\delta_0,\epslon_1,\delta_1,\ldots,\epslon_m,\delta_m):=S_{\delta_0,\epslon_0}\cup S_{\delta_1,\epslon_1}\cup\cdots\cup S_{\delta_m,\epslon_m}
\]
\end{definition}

Gabrielov and Vorobjov refer to the general case, in which $S$ is an arbitrary definable set together with any representing families $\{S_\delta\}$ and $\{S_{\delta,\epslon}\}$, as the \emph{definable case}. If $S$ is a $\mc{P}$-set for some collection $\mc{P}=\{h_1,\ldots,h_s\}$ of definable continuous functions $h_i:\bb{R}^n\rightarrow \bb{R}$, they prescribe particular choices of $A$, $\{S_\delta\}$, and $\{S_{\delta,\epslon}\}$ (see Definition \ref{def:ConstructibleCase}), which grants an extra property (termed separability; see Definition \ref{def:Gabrielov5.7} or \cite{gabrielov2009approximation} Section 5.2) that allows for the stronger version of the results. We call this case the \emph{constructible case}. The choices in the constructible case also guarantee that our sets $S_\delta$ and $S_{\delta,\epslon}$ are described by definable functions closely related to the original functions $h_i$.

In these settings, Gabrielov and Vorobjov prove the following theorem. Note that they show (\cite{gabrielov2009approximation} Lemma 1.9) that provided $m$ is at least $1$, there exists a one-to-one correspondence between the connected components of $S$ and $T$. This allows Gabrielov and Vorobjov to consider each connected component individually and so ignore basepoint considerations, a luxury we do not have.

\begin{thm}[\cite{gabrielov2009approximation} Theorem 1.10]\label{thm:Gabrielov1.10}
$\;$
\begin{enumerate}[(i)]
    \item (Definable Case) For $0<\epslon_0\ll\delta_0\ll\cdots\ll \epslon_m\ll\delta_m\ll 1$ and every $1\leq k\leq m$, there are epimorphisms
    \begin{align*}
        \psi_{\#,k}&:\pi_k(T)\rightarrow\pi_k(S)\\
        \psi_{*,k}&:H_k(T)\rightarrow H_k(S)
    \end{align*}
    and in particular, $\rank(H_k(S))\leq \rank(H_k(T))$
    \item (Constructible Case) In the constructible case, for $0<\epslon_0\ll\delta_0\ll\cdots\ll \epslon_m\ll\delta_m\ll 1$ and every $1\leq k\leq m-1$, $\psi_{\#,k}$ and $\psi_{*,k}$ are isomorphisms. In particular, $\rank(H_k(S))=\rank(H_k(T))$. Moreover, if $m\geq \dim(S)$, then $T\simeq S$.
\end{enumerate}
\end{thm}

In applications, even in the definable case one may via this construction obtain improved upper bounds on the first $m$ Betti numbers of a given set $S$. Our primary setting of interest, the one considered by Basu and Reiner in \cite{basu2021vandermonde}, utilizes the constructible case.

Our main theorem (Theorem \ref{thm:MainTheorem}) is an equivariant version of the above theorem. It is clear from the definition of $T$ that, so long as each set in the family $\{S_{\delta,\epslon}\}$ is symmetric relative to the action of some group $G$, then $T$ will be as well. In the definable case, we need only assume that we have chosen our families to consist of symmetric sets. In the constructible case, we will show that if $S$ is symmetric and the collection $\{h_1,\ldots, h_s\}$ is invariant under the action of $G$, our choices will produce symmetric sets. In subsequent sections, we show that for $G$ a finite reflection group acting on $\bb{R}^n$, we can in fact construct an equivariant map $\psi:T\rightarrow S$ which induces the desired isomorphisms and epimorphisms $\psi_{\#,k}$ and $\psi_{*,k}$ on the level of homotopy and homology. The remainder of this section is devoted to describing the sets $S_\delta$ and $S_{\delta,\epslon}$ in the constructible case.

\subsection{Constructible Case: Bounding $S$}\label{sect:Boundedness}

Let $S\subset \rcf^n$ be definable, and assume $S$ is unbounded. Via the conical structure at infinity of definable sets, there exists an $r\in \rcf$, $r>0$, such that $S$ is (definably) homotopy equivalent to $S\cap \clos{B(0,r)}$. Basu, Pollack, and Roy in \cite{BasuSaugata2006AiRA} show this for semialgebraic sets. However, their proof holds for any o-minimal structure in which addition and mulitpication are definable. The proof in \cite{BasuSaugata2006AiRA} centers on the local conical structure specifically of semialgebraic sets, but this property (a consequence of Hardt triviality) holds for any o-minimal expansion of some real closed field (see \cite{van1998tame} Chapter 9 Theorem 2.3). The remaining steps of their proof may be performed in any o-minimal structure provided it contains addition and mulitpication.

\begin{prop}[Conical Structure at Infinity, see \cite{BasuSaugata2006AiRA} Proposition 5.49]\label{thm:ConicStructureAtInfty}
Let $A\subset \rcf^n$ be definable. Then there exists an $r\in \rcf$, $r>0$, such that $A$ is definably homotopy equivalent to $A\cap \clos{B(0,r)}$. Specifically, there exists a continuous definable function $h:[0,1]\times A\rightarrow A$ with $h(0,-)=\id_A$, $h(1,-)$ having image contained in $A\cap \clos{B(0,r)}$, and with $h(t,a)=a$ for each $t\in [0,1]$ and $a\in A\cap \clos{B(0,r)}$.
\end{prop}

In our case, for any set $S\subset\bb{R}^n$ that is symmetric under the action of a finite reflection group $G$, certainly $S\cap \clos{B(0,r)}$ is as well. The inclusion $S\cap \clos{B(0,r)}\hookrightarrow S$ is also clearly equivariant, and so induces equivariant isomorphisms of homotopy and homology groups
\begin{align*}
    \pi_k(S\cap\clos{B(0,r)},*)&\rightarrow \pi_k(S, *)\\
    H_k(S\cap \clos{B(0,r)})&\rightarrow H_k(S)
\end{align*}
for each $k\geq 0$.

We now replace $S$ by $S\cap\clos{B(0,r)}$ to assume henceforth that $S$ is bounded. If $S$ was a $\mc{P}$-set for some $\mc{P}=\{h_1,\ldots,h_s\}$, we replace $\mc{P}$ with $\mc{P}\cup \{r^2-(X_1^2+\cdots X_n^2)\}$, and so increase $s$ by one. Since $r^2-(X_1^2+\cdots +X_n^2)$ is symmetric relative to the action of any finite reflection group $G$, our collection $\mc{P}$ remains invariant (and in fact if each $h_i\in \mc{P}$ was symmetric, the same remains true after including this new function in our collection). Note also that in the case that each $h_i(x)$ is a polynomial, the only instance in which we have increased the maximum degree among our collection is if each of $\{h_1,\ldots,h_s\}$ had degree one.

This is slightly different from the procedure employed by Gabrielov and Vorobjov in \cite{gabrielov2009approximation}, where the larger compact set $A$ was taken as the definable one-point compactification of $\bb{R}^n$. This choice aids certain results on Betti numbers in \cite{gabrielov2009approximation}, but we find our own method more convenient when tracking the group's action. Since Gabrielov and Vorobjov bound their sets $\{S_\delta\}$ and $\{S_{\delta,\epslon}\}$ in the constructible case by intersecting with closed balls of radius $1/\delta$, our method does not stray too far from the intuition of the original. Our assumption will increase the number of equations needed to define $T$, but only slightly.

\subsection{Constructible Case: the Sets $\{S_\delta\}$ and $\{S_{\delta,\epslon}\}$}

In the constructible case (with the assumption that $S$ is bounded), we form families $\{S_\delta\}_{\delta>0}$ and $\{S_{\delta,\epslon}\}_{\delta,\epslon>0}$ by decomposing $S$ into sign sets of the functions defining $S$, and then contracting inequalities and expanding equalities by a factor of $\delta$ or $\epslon$ respectively, in a manner we now proceed to describe.

\begin{definition}[\cite{gabrielov2009approximation} Definition 1.5]\label{def:Gabrielov1.5}
Let $\mc{P}=\{h_1,\ldots,h_s\}$ be a finite collection of functions with each $h_i:\rcf^n\rightarrow \rcf$. Let $\{I_0,I_+,I_-\}$ be a partition of $\{1,\ldots, s\}$. Then the set $B_{\mc{P},(I_0,I_+,I_-)}$ given by
\[
\{\bs{x}\in \rcf^n\mid \bigwedge_{i\in I_0} (h_i(\bs{x})=0) \wedge \bigwedge_{i\in I_+} (h_i(\bs{x})>0) \wedge \bigwedge_{i\in I_-} (h_i(\bs{x})<0) \}
\]
will be called the \emph{sign set} of $\mc{P}$ corresponding to the tuple $(I_0,I_+,I_-)$.
\end{definition}

Note that any two distinct sign sets of some $\mc{P}=\{h_1,\ldots,h_s\}$ are disjoint, and that for $S$ a $\mc{P}$-set, we may write $S$ as a union of sign sets of the functions in $\mc{P}$. Though for a given collection $\mc{P}$ of functions, some tuples $(I_0,I_+,I_-)$ will produce empty sign sets, we will exclude these tuples from the sign set decomposition of any $\mc{P}$-set.

In the constructible case, we make the following choices for $A$, $\{S_\delta\}$ and $\{S_{\delta,\epslon}\}$.

\begin{definition}[from \cite{gabrielov2009approximation}]\label{def:ConstructibleCase}
Let $S\subset\bb{R}^n$ be a bounded $\mc{P}$-set for some collection $\mc{P}=\{h_1,\ldots,h_s\}$ of continuous definable functions $h:\bb{R}^n\rightarrow \bb{R}$. Let $\mc{B}$ denote the set of tuples $(I_0,I_+,I_-)$ corresponding to elements in the sign set decomposition of $S$.

Let $r>0$ be such that $S\subset \clos{B(0,r)}$. Then take $A=\clos{B(0,r')}$ for some $r'$ sufficiently larger than $r$. For each $\delta>0$, we let $S_\delta$ be the union of sets defined by
\[
\bigwedge_{i\in I_0} (h_i=0) \wedge \bigwedge_{i\in I_+} (h_i\geq\delta) \wedge \bigwedge_{i\in I_-} (h_i\leq-\delta)
\]
over all tuples $(I_0,I_+,I_-)\in \mathcal{B}$. For $\delta,\epslon>0$, we let $S_{\delta,\epslon}$ be the union of the sets given by
\[
\bigwedge_{i\in I_0} (-\epslon\leq h_i\leq\epslon) \wedge \bigwedge_{i\in I_+} (h_i\geq\delta) \wedge \bigwedge_{i\in I_-} (h_i\leq-\delta)
\]
again over all tuples $(I_0,I_+,I_-)\in \mathcal{B}$.
\end{definition}

One may check that $S$ is indeed represented by the families $\{S_\delta\}_{\delta>0}$ and $\{S_{\delta,\epslon}\}_{\delta,\epslon>0}$ in $A$. It remains to consider the conditions needed to ensure that the sets $S_\delta$ and $S_{\delta,\epslon}$ are symmetric.

\begin{prop}
Assume we are in the constructible case. If $S$ is symmetric under the action of $G$ and furthermore we have that for each $g\in G$, $\{h_1\circ g,\ldots, h_s\circ g\}=\{h_1,\ldots, h_s\}$, then every set in the families $\{S_\delta\}_{\delta>0}$ and $\{S_{\delta,\epslon}\}_{\delta,\epslon>0}$ is symmetric.
\end{prop}
\begin{proof}
Let $\mc{P}=\{h_1,\ldots, h_s\}$. Note that for a sign set $B=B_{\mc{P}, (I_0,I_+,I_-)}$ and any $g\in G$, $g(B)=\{g(\bs{x})\mid \bs{x}\in B\}$ is given by
\begin{multline*}
    \{\bs{x}\in \bb{R}^n\mid \bigwedge_{i\in I_0} (h_i(g^{-1}(\bs{x}))=0) \\
    \wedge \bigwedge_{i\in I_+} (h_i(g^{-1}(\bs{x}))>0) \wedge \bigwedge_{i\in I_-} (h_i(g^{-1}(\bs{x}))<0) \}
\end{multline*}

For each $g\in G$, we define a function on the indices  $\{1,\ldots,s\}\rightarrow\{1,\ldots,s\}$, given by $g(i)=j$ if $h_j=h_i\circ g^{-1}$.

Let $\mathcal{B}$ be the collection of tuples $(I_0,I_+,I_-)$ corresponding to the sign sets contained in $S$. We claim $\mathcal{B}$ has the property that for each $g\in G$, if $(I_0,I_+,I_-)\in \mathcal{B}$ then $(g(I_0), g(I_+), g(I_-))\in \mathcal{B}$. Indeed, if $B$ is the sign set corresponding to $(I_0,I_+,I_-)$, we have that $B\subset S$ and hence $g(B)\subset S$. But we have seen above that $g(B)$ is the sign set corresponding to $(g(I_0), g(I_+), g(I_-))$. The symmetry of each $S_\delta$ and each $S_{\delta,\epslon}$ follows from this property and the descriptions of each $S_\delta$ and $S_{\delta,\epslon}$ given in Definition \ref{def:ConstructibleCase}.
\end{proof}

In particular, if each $h_i$ is a symmetric function, then $h_i\circ g=h_i$, and so clearly $\{h_1\circ g,\ldots, h_s\circ g\}=\{h_1,\ldots, h_s\}$ for every $g\in G$.

\subsection{Connected Components}
From the definitions (in both the definable and constructible cases), we see a clear correspondence between the connected components of $T$ and those of $S$. Let $S$ be represented in $A$ by any compact families $\{S_\delta\}$ and $\{S_{\delta,\epslon}\}$, and let $T=T(\epslon_0,\delta_0,\ldots,\epslon_m,\delta_m)$ be as described in Definition \ref{def:Gabrielov1.8}. For $0<\epslon_0\ll\delta_0\ll\cdots\ll\epslon_m\ll\delta_m\ll 1$, let $\bs{S}$ and $\bs{T}$ denote the sets of connected components of $S$ and $T$ respectively. Gabrielov and Vorobjov describe a map $C:\bs{T}\rightarrow \bs{S}$. If $S'$ is a connected component of $S$, then those elements of $\bs{T}$ which map to $S'$ under $C$ are exactly those components which would make up the approximating set defined relative to $S'$ alone. Thus, by the symmetry of $T$ and $S$, $C$ is equivariant. Gabrielov and Vorobjov demonstrate in \cite{gabrielov2009approximation} Lemma 1.9 that, provided $m>0$, we may choose the conditions upon the parameters in such a way that $C$ is bijective. Hence, though we can't simply reduce to the case of a connected set without risking a loss of symmetry, we may when needed use $C$ to pair the connected components of $S$ and $T$ and then apply the needed theorems to each pair individually.

\section{Symmetric Triangulation}\label{sect:Triangulation}

The construction of the maps in \cite{gabrielov2009approximation} relies upon triangulating the given set $S$ inside the larger compact set $A$. Definable sets are triangulable (see for example \cite{coste2000introduction} Theorem 4.4, quoted \href{thm:Coste4.4}{below}). However, in order to ensure equivariance, we will need a triangulation that respects the action of our group. That is the aim of this section.

\subsection{Triangulation of Definable Sets}\label{sect:TriangulationDefinableSets}

\begin{definition}[\cite{coste2000introduction} Section 4.3]\label{def:CosteTriangulation}
Let $A$ be a closed bounded definable subset of $\rcf^n$. Then a \emph{triangulation} $(\Lambda,\Phi)$ of $A$ is a finite simplicial complex $\Lambda$ in $\rcf$ together with a definable homeomorphism $\Phi:\abs{\Lambda}\rightarrow A$.
\end{definition}

If we wish to decompose a set which is not necessarily closed into images of simplices, we must do so within some larger compact set. For $A\subset \rcf^n$ compact and $S\subset A$, we will say a triangulation $(\Lambda, \Phi)$ of $A$ is \emph{adapted} to $S$ if $S$ is a union of images by $\Phi$ of simplices of $\Lambda$. Triangulations of compact definable sets always exist, and may be adapted to any finite collection of subsets.

\begin{thm}[Triangulation of Definable Sets (\cite{coste2000introduction} Theorem 4.4)]\label{thm:Coste4.4}
Let $A$ be a closed and bounded definable subset of $\rcf^n$, and let $S_1,\ldots, S_l$ be definable subsets of $A$. Then there exists a triangulation $(\Lambda,\Phi)$ of $A$ adapted to each of $S_1,\ldots, S_l$. Furthermore, we may choose all vertices of $\Lambda$ to be in $\bb{Q}^n$.
\end{thm}

We need to triangulate the symmetric set $S$ in a manner that retains symmetry relative to our group of interest. Specifically, we would like a triangulation with the following properties:

\begin{definition}\label{def:EquivariantTriangulation}
Let $G$ be a finite reflection group acting on $\rcf^n$ and let $A\subset \rcf^n$ be a closed bounded definable set symmetric relative to $G$. A triangulation $(\Lambda, \Phi)$ of $A$ is said to be an \emph{equivariant triangulation} if $\Lambda$ is symmetric as a simplicial complex and the map $\Phi:\abs{\Lambda}\rightarrow A$ is equivariant.
\end{definition}

Our general strategy for obtaining an equivariant triangulation of $A$ adapted to $S$ involves triangulating a fundamental region of $\bb{R}^n$, and then applying the action of $G$. This strategy has appeared in other contexts; see in particular \cite{illman1978smooth} which proves such a result for smooth manifolds. We would like to establish a proof in the definable case. Furthermore, we would like our proof to follow the spirit of \cite{coste2000introduction} Theorem 4.4, in which we create a concrete simplicial complex in $\bb{R}^n$ symmetric relative to our existing action of $G$.

In order to make this approach feasible in the concrete setting, we must triangulate the portion of $A$ lying in our fundamental region in such a way that the resulting simplicial complex lies within the same fundamental region of $\rcf^n$, with points lying in its walls remaining as such. This motivates the following definition.

\begin{definition}
Let $B\subset \rcf^n$. If a triangulation $(\Lambda,\Phi)$ of $A\subset \rcf^n$ is adapted to $A\cap B$ and furthermore we have $x\in A\cap B$ if and only if $\Phi^{-1}(x)\in \abs{\Lambda}\cap B$, then we say the triangulation \emph{respects} the set $B$.
\end{definition}

To show that we may find a triangulation respecting our fundamental region and its walls, we will in fact prove something slightly stronger: for an arrangement of affine hyperplanes given by some collection $\{L_1(\bs{x})=0, \ldots, L_k(\bs{x})=0\}$ of linear functions, we can find a triangulation of $A$ that respects any of the various half-spaces determined by our hyperplanes (for convenience, we will phrase this as respecting sign sets of $\{L_1,\ldots,L_k\}$). We lose the property that all vertices lie in $\bb{Q}^n$, but the vertices of our triangulation will be rational expressions in the coefficients of $L_1,\ldots, L_k$. Since our application does not in fact require any such condition on the vertices, this is not an issue here.

The reader may wish to study the proof of the following lemma in conjunction with Example \ref{eg:RestrictedTriangulation}, in order to see concrete usage of the notation set forth in the proof.

\begin{lemma}\label{thm:RestrictedTriangulation}
Let $A\subset \rcf^n$ be closed, bounded, and definable, and let $S_1,\ldots, S_l$ be definable subsets of $A$. Let $\{L_1,\ldots, L_k\}$ be a collection of functions with each $L_i:\rcf^n\rightarrow \rcf$ given by $L_i(\bs{x})=a_{i,0}+a_{i,1}x_1+\cdots+a_{i,n}x_n$ for some $a_{i,j}\in \rcf$. Then there exists a triangulation $(\Lambda,\Phi)$ of $A$ a which is adapted to each $S_1,\ldots, S_l$ and which respects all sign sets of $\{L_1,\ldots, L_k\}$. Furthermore, we can choose this triangulation so that the coordinates of all vertices are in $\bb{Q}(\{a_{i,j}\mid 1\leq i\leq k, 0\leq j\leq n\})^n$.
\end{lemma}
\begin{proof}
We can obtain a triangulation with the desired properties by way of a few slight modifications to the proof of Theorem \ref{thm:Coste4.4} presented in \cite{coste2000introduction}. Specifically, the proof there chooses a number of points $b_\nu$ which serve as vertices in the simplicial complex $\Lambda$. We must add the requirement that our points $b_\nu$ belong to the correct sign sets, and confirm that our additional claims hold with this adjustment.

The proof in \cite{coste2000introduction} proceeds by induction on dimension. In the case $n=1$, we have that $L_i=a_{i,0}+a_{i,1}x_1$ for each $i$. We may assume without loss of generality that $a_{i,1}\neq 0$ for every $i$. Set $c_i=-\frac{a_{i,0}}{a_{i,1}}$. Reorder and remove duplicates to assume we have points $c_1<\cdots<c_k\in \rcf$. Then the sign sets we must respect are the points and intervals $(-\infty, c_1)$, $\{c_i\}$ for each $i$, $(c_i,c_{i+1})$ for $1\leq i\leq k-1$, and $(c_k, \infty)$. As \cite{coste2000introduction} describes, we choose points $\xi_1<\cdots<\xi_p\in \rcf$ so that each of $A$, $S_1\ldots, S_l$, and each nonempty intersection of $A$ and a sign set of $\{L_1,\ldots, L_k\}$ can be written as a union of points and intervals $\{\xi_\mu\}$ and $(\xi_\mu, \xi_{\mu+1})$.

We define a map $\tau:\{\xi_1,\ldots, \xi_p\}\rightarrow \rcf$ as follows:
\begin{itemize}
    \item If $k=0$, then let $\tau(\xi_i)=i-1$ for each $i$.
\end{itemize}
Otherwise,
\begin{itemize}
    \item If $\xi_\mu=c_i$ for some $i$, let $\tau(\xi_\mu)=c_i$
    \item Let $\xi_{0,1}<\cdots<\xi_{0,p_0}$ with $\{\xi_{0,1},\ldots, \xi_{0,p_0}\}=\{\xi_1,\ldots, \xi_p\}\cap (-\infty, c_1)$. Then set $\tau(\xi_{0,\mu})=c_1-p_0+\mu-1$
    \item Let $\xi_{i,1}<\cdots<\xi_{i,p_i}$ with $\{\xi_{i,1},\ldots, \xi_{i,p_i}\}=\{\xi_1,\ldots, \xi_p\}\cap (c_i, c_{i+1})$ (where $1\leq i\leq k-1$). Then set $\tau(\xi_{i,\mu})=c_i+\mu\frac{c_{i+1}-c_i}{p_i+1}$
    \item Let $\xi_{k,1}<\cdots<\xi_{k,p_k}$ with $\{\xi_{k,1},\ldots, \xi_{k,p_k}\}=\{\xi_1,\ldots, \xi_p\}\cap (c_k, \infty)$. Then set $\tau(\xi_{k, \mu})=c_k+\mu$
\end{itemize}

In either case, we have that $\tau$ preserves order and respects containment in sign sets of $\{L_1,\ldots, L_k\}$. Each $\tau(\xi_\mu)$ is also a rational expression in the coefficients of the $L_i$'s. We extend $\tau$ piecewise linearly to obtain a definable order preserving homeomorphism $\tau': \rcf\rightarrow \rcf$, and take our simplicial complex $\Lambda$ to be $\tau'(A)$ and $\Phi=(\tau')^{-1}_{\restriction\abs{\Lambda}}:\abs{\Lambda}\rightarrow A$.

Assume now that $n>1$, and that our claim holds in dimension $n-1$. We can safely follow the procedure in \cite{coste2000introduction} to assume each $S_i$ is closed. We denote the boundary of $A$ by $F_0$, and the boundary of $S_i$ by $F_i$ for each $1\leq i\leq l$. Choose a cell decomposition of $\rcf^n$ adapted to each of $F_0,\ldots, F_l$ and also to the sign sets of $\{L_1,\ldots, L_k\}$. If we let $p:\rcf^n\rightarrow \rcf^{n-1}$ be the projection onto the first $n-1$ coordinates, we have that our cell decomposition of $\rcf^n$ gives us a decomposition of $p(A)$ into definably connected definable subsets $X_\alpha$, and for each $X_\alpha$ a finite number of continuous functions $\zeta_{\alpha,1}<\ldots<\zeta_{\alpha,m_\alpha}:X_\alpha\rightarrow \rcf$. This means each $F_i$ can be written as a union of graphs $\zeta_{\alpha,\mu}$. Note that because our cell decomposition of $\rcf^n$ is adapted to the sign sets of $\{L_1,\ldots, L_k\}$, each cell in $A$ is contained in precisely one such sign set. Furthermore for each sign set $B$ of $\{L_1,\ldots, L_k\}$, we have that $A\cap B$ is a union of cells in the decomposition.

Consider the sets given by $\bigcap_{i\in \lambda}\{L_i(\bs{x})=0\}$ for various selections of $\lambda\subset \{1,\ldots, k\}$. If for a given $\lambda$, $P'_\lambda=p\paren{\bigcap_{i\in \lambda}\{L_i(\bs{x})=0\}}$ has dimension $n-2$, then $P'_\lambda$ is an affine hyperplane of $\rcf^{n-1}$, and therefore is the zero set for some linear function $L_\lambda'(\bs{x})=a_{\lambda,0}'+a_{\lambda,1}'x_1+\cdots+a_{\lambda,n-1}'x_{n-1}:\rcf^{n-1}\rightarrow \rcf$.
Furthermore, we can choose $L_\lambda'$ such that each $a_{\lambda,j}'\in \rcf$ is contained in $\bb{Q}(\{a_{i,j}\mid i\in \lambda, 0\leq j\leq n\})$. Removing any duplicates, we obtain a collection $\{L_1',\ldots, L_{k'}'\}$ of functions $L_i'(\bs{x})=a_{i,0}'+a_{i,1}'x_1+\cdots+a_{i,n-1}'x_{n-1}:\rcf^{n-1}\rightarrow \rcf$ having the property that all coefficients $a_{i,j}'$ are rational expressions in the coefficients $\{a_{i,j}\}$ of the functions $L_i$. Furthermore, if $B$ is a sign set of $\{L_1,\ldots,L_k\}$, then $p(B)$ is a union of sign sets of $\{L_1',\ldots,L_{k'}'\}$. Since our cell decomposition was adapted to all sign sets $B$ of $\{L_1,\ldots, L_k\}$, we have that each $X_\alpha$ is contained in exactly one sign set of $\{L_1',\dots, L_{k'}'\}$.

Let $(\Lambda_{n-1}, \Psi_{n-1})$ be a triangulation of $p(A)\subset \rcf^{n-1}$ which is adapted to each $X_\alpha$ and respects each sign set of $\{L_1',\ldots, L_{k'}'\}$. Consider (as \cite{coste2000introduction} does) the set
\[
A'=\{(x',x_n)\in \abs{\Lambda_{n-1}}\times \rcf\mid (\Psi_{n-1}(x'),x_n)\in A\}
\]
Note that the map $\Psi=(\Psi_{n-1}, \id):A'\rightarrow A$ is a definable homeomorphism. Let $\{\delta_\beta\}$ denote the (finite) collection of all simplices of $\Lambda_{n-1}$. Since $\Psi_{n-1}$ is such that for each cell $X_\alpha\subset p(A)$, $\Psi_{n-1}^{-1}(X_\alpha)$ is a union of simplices of $\Lambda_{n-1}$, we have that for each $\delta_\beta$ there is precisely one $\alpha$ so that $\Psi_{n-1}(\delta_{\beta})\subset X_\alpha$. Consider the graphs $\zeta_{\alpha,1}<\cdots<\zeta_{\alpha, m_\alpha}$ defined on each $X_\alpha$ in our cell decomposition of $A$. Then for a given $\beta$ and $1\leq \mu\leq m_\beta:=m_\alpha$, $\xi_{\beta, \mu}=\zeta_{\alpha,\mu}\circ\Psi_{n-1}\restriction_{\delta_\beta}:\delta_\beta\rightarrow \rcf$ is such that $\Psi\circ\xi_{\beta, \mu}=\zeta_{\alpha,\mu}\circ\Psi_{n-1}$ on $\delta_\beta$. We also have that on each $\delta_\beta$, $\xi_{\beta, 1}<\cdots <\xi_{\beta, m_\beta}$. Let $\{C_\nu\}$ be the (finite) collection of all graphs $\xi_{\beta,\mu}:\delta_\beta\rightarrow \rcf$ and bands $(\xi_{\beta, \mu}, \xi_{\beta, \mu+1})\subset \delta_\beta\times \rcf$ which are contained in $A'$. Note that for each $\nu$, there exists one sign set $B$ of $\{L_1,\ldots, L_m\}$ such that $\Psi(C_{\nu})\subset B$.

For each $\delta_\beta$, let $b(\delta_\beta)=((b(\delta_\beta))_1,\ldots,(b(\delta_\beta))_{n-1})$ denote the barycenter of the simplex $\delta_\beta$. For each cell $C_\nu$ in $A'$, we will choose a point $b_\nu$ to serve as a vertex in our new simplicial complex $\Lambda$. Consider a simplex $\delta_\beta$ in $\Lambda_{n-1}$. Let
\begin{align*}
    I&=\{1\leq i\leq k\mid L_i(\Psi_{n-1}(b(\delta_\beta)), x_n)=0 \text{ has a unique solution}\}\\
    &= \{1\leq i\leq k\mid L_i(b(\delta_\beta), x_n)=0\text{ has a unique solution}\}
\end{align*}
For $i\in I$, let $c_{\beta, i}\in \rcf$ be such that $L_i(\Psi_{n-1}(b(\delta_\beta)), c_{\beta, i})=0$, and let $c_{\beta, i}'\in \rcf$ be such that $L_i(b(\delta_\beta), c_{\beta, i}')=0$. Because our triangulation of $A'$ respects $\{L_1',\ldots, L_{k'}'\}$, we know that the map of subsets $\{c_{\beta,i}\mid i\in I\}\rightarrow \{c'_{\beta,i}\mid i\in I\}$ of $\rcf$ given by $c_{\beta,i}\mapsto c'_{\beta,i}$ is a well-defined, order preserving bijection. So, let $\card(\{c_{\beta,i}\mid i\in I\})=k_\beta$, and reorder and remove skipped indices and duplicates, to assume $c_{\beta, 1}<\cdots <c_{\beta, k_\beta}$ (and in this same indexing $c_{\beta, 1}'<\cdots < c_{\beta, k_\beta}'$).

For each cell $C_\nu$ defined on $\delta_\beta$, let
\[
b_\nu=((b(\delta_\beta))_1,\ldots,(b(\delta_\beta))_{n-1}, (b_\nu)_n)
\]
where $(b_\nu)_n$ is assigned as follows. First, assume $C_\nu$ is the graph of the function $\xi_{\beta,\mu}$. Then we assign $(b_\nu)_n$ according to a schema similar to the one appearing in the $n=1$ case:
\begin{itemize}
    \item If $k_\beta=0$, let $(b_\nu)_n=\mu-1$
\end{itemize}
Otherwise,
\begin{itemize}
    \item If $\xi_{\beta,\mu}(b(\delta_\beta))=c_{\beta,i}$ for some $i$, let $(b_\nu)_n=c_{\beta,i}'$
    \item Let $\xi_{\beta,0,1}<\cdots<\xi_{\beta,0,p_0}$ be those graphs with $\xi_{\beta,0,\mu}(b(\delta_\beta))<c_{\beta,1}$. Then for $C_\nu$ corresponding to $\xi_{\beta,0,\mu}$, set $(b_\nu)_n=c_{\beta,1}'-p_0+\mu-1$
    \item For each $1\leq i<k_{\beta}$, let $\xi_{\beta,i,1}<\cdots<\xi_{\beta,i,p_i}$ be those graphs with $c_{\beta,i}<\xi_{\beta,0,\mu}(b(\delta_\beta))<c_{\beta,i+1}$. Then for $C_\nu$ corresponding to $\xi_{\beta,i,\mu}$, set $(b_\nu)_n=c_{\beta,i}'+\mu\frac{c_{\beta,i+1}'-c_{\beta,i}'}{p_i+1}$
    \item Let $\xi_{\beta,k_\beta,1}<\cdots<\xi_{\beta,k_\beta,p_{k_\beta}}$ be those graphs such that $c_{\beta,k_\beta}<\xi_{\beta,0,\mu}(b(\delta_\beta))$. Then for $C_\nu$ corresponding to $\xi_{\beta,k_\beta,\mu}$, set $(b_\nu)_n=c_{\beta, k_\beta}'+\mu$
\end{itemize}
Finally, if $C_\nu$ is a band $(\xi_{\beta, \mu}, \xi_{\beta, \mu+1})$, set $(b_\nu)_n=\frac{(b_{\nu_0})_n+(b_{\nu_1})_n}{2}$, where $C_{\nu_0}$ and $C_{\nu_1}$ correspond to the graphs $\xi_{\beta,\mu}$ and $\xi_{\beta,\mu+1}$ respectively (note that since $A'$ is bounded, all of our bands appearing among the sets $C_\nu$ are bounded). Then for any cell $C_\nu$, we have that the point $b_\nu$ is in the same sign set of $\{L_1,\ldots, L_k\}$ as $\Psi(C_\nu)$ is. Note also that each $c_{\beta,i}'$ is a rational expression in the coordinates of $b(\delta_\beta)$ and the coefficients of $L_i$. By our inductive hypothesis, all vertices in $\Lambda_{n-1}$ have coordinates which are rational expressions in the coefficients of $L_1',\ldots, L_{k'}'$, and hence of $L_1,\ldots, L_k$. Hence the same applies to our barycenters $b(\delta_\beta)$ for $\delta_\beta\in \Lambda_{n-1}$, and so by the definition of $(b_\nu)_n$ above, each $b_\nu$ has coefficients which are in $\bb{Q}(\{a_{i,j}\})^n$.

For each cell $C_\nu$, \cite{coste2000introduction} next builds a polyhedron $\clos{D}_\nu$ together with its subdivision into simplices. The procedure occurs inductively on the dimension of $\clos{C}_\nu$: if $\clos{C}_\nu$ is a point, take $\clos{D}_\nu$ to be $\{b_\nu\}$. Otherwise, take $\clos{D}_\nu$ as the cone from $b_\nu$ to the union of all $\clos{D}_{\nu'}$ with $\clos{C}_{\nu'}\subset \partial(\clos{C}_\nu)$. The decomposition of $\clos{D}_\nu$ into simplices comes from taking cones with vertex $b_\nu$ and base a simplex contained in $\clos{D}_{\nu'}$ for some $D_{\nu'}\subset \partial(\clos{D}_\nu)$.

Taking all these simplices together, we obtain our desired simplicial complex $\Lambda$ with $\abs{\Lambda}=\bigcup\clos{D}_\nu$. Of primary note is the fact that the vertices of all simplices come from among the points $b_\nu$. Hence because $L_1=0,\ldots, L_k=0$ define affine subspaces of $\rcf^n$, we have that for a given sign set $B$ of $\{L_1,\ldots, L_k\}$, a (relatively open) simplex $\Delta(b_{\nu_1},\ldots, b_{\nu_m})$ in $\Lambda$ is either contained in $B$ (if all $b_{\nu_i}$ are in $B$) or disjoint from $B$ (else).

For each $\clos{D}_\nu$, we define a preparatory homeomorphism $\theta_\nu: \clos{D}_\nu\rightarrow\clos{C}_\nu$ as follows: if $\clos{C}_\nu$ is a graph $\clos{\xi}_{\beta, \mu}$ (see \cite{coste2000introduction} for a justification of why we may continuously extend $\xi_{\beta, \mu}$ to the closed simplex $\clos{\delta}_\beta$), let $\theta_\nu(x',x_n)=(x',\clos{\xi}_{\beta,\mu}(x'))$. If $C_\nu$ is a closed band $[\clos{\xi}_{\beta,\mu},\clos{\xi}_{\beta, \mu+1}]$, we map each segment $(\{x'\}\times \rcf)\cap \clos{D}_\nu$ affinely to the corresponding segment $\{x'\}\times [\clos{\xi}_{\beta,\mu}(x'),\clos{\xi}_{\beta,\mu+1}(x')]$. Composing with $\Psi$ (and by convexity of our sign sets), we have that for a simplex $\Delta$ of $\Lambda$ with $\Delta\subset\clos{D}_\nu$, $\Psi\circ \theta_\nu(\Delta)\subset B\cap A$ iff $\Delta \subset B$.

Unfortunately, we cannot simply piece together our maps $\Psi\circ\theta_\nu$ to obtain the desired homeomorphism $\Phi:\abs{\Lambda}\rightarrow A$. Instead, we construct a new $\Phi':\abs{\Lambda}\rightarrow A'$, inducting on the dimension of $\clos{D}_\nu$. If $\clos{D}_\nu$ is a point, take $\Phi'_\nu:\clos{D}_\nu\rightarrow \clos{C}_\nu$ to be $\theta_\nu$. Otherwise, we can construct a homeomorphism $\rho_\nu:\partial(\clos{D}_\nu)\rightarrow \partial(\clos{D}_\nu)$ by specifying that ${\rho_\nu}\restriction_{\clos{D}_{\nu'}}=\theta_\nu^{-1}\circ\Phi'_{\nu'}$ for each $\clos{D}_{\nu'}\subset \partial(\clos{D}_\nu)$. We use the conic structure of $\clos{D}_\nu$ to extend $\rho_\nu$ to a homeomorphism $\eta_\nu:\clos{D}_\nu\rightarrow\clos{D}_\nu$, and set $\Phi'_\nu=\theta_\nu\circ\eta_\nu$. Now, $\Phi'$ given by $\Phi'_{\restriction\clos{D}_\nu}=\Phi'_\nu$ is well defined even on the boundaries of the sets $\clos{D}_\nu$. Finally, set $\Phi:\abs{\Lambda}\rightarrow A$ to be $\Psi\circ \Phi'$.

That the triangulation $(\Lambda,\Phi)$ is adapted to $S_1,\ldots,S_l$ and also $A\cap B$ for each sign set $B$ of $\{L_1,\ldots, L_m\}$ follows as from the proof in \cite{coste2000introduction}. We must still check that the triangulation respects all of our sign sets. Let $x\in \abs{\Lambda}$, and let $\Delta\in \Lambda$ be the unique simplex with $x\in\Delta$. Take $\clos{D}_\nu$ to be of minimal dimension with $\Delta\subset \clos{D}_\nu$. If $\clos{D}_\nu$ is a point, then $\clos{D}_\nu=\Delta=\{x\}=\{b_\nu\}$, and we have already established that for a given sign set $B$, $x=b_\nu\in B$ iff $\Phi(x)=\Psi\circ\theta_\nu(x)\in B$. Say that $\Dim(\clos{D}_\nu)>0$. Then $\Delta=\Delta(b_{\nu_0},\ldots, b_{\nu_{q-1}},b_{\nu_q}=b_\nu)$, and there is some $\nu'$ such that $\clos{D}_{\nu'}\subset \partial(\clos{D}_{\nu})$ and $\clos{\Delta'}=\clos{\Delta}(b_{\nu_0},\ldots,b_{\nu_{q-1}})\subset \clos{D}_{\nu'}$. Then assuming we have established that $\Phi$ respects sign sets of $\{L_1,\ldots, L_m\}$ on $\clos{D}_{\nu'}$, we know that ${\rho_\nu}\restriction_{\clos{D}_{\nu'}}=\theta_\nu^{-1}\circ \Phi'_{\nu'}$ carries $\Delta'$ to a subset of some $B\cap \abs{\Lambda}$ iff $\Delta'$ is already a subset of $B\cap \abs{\Lambda}$. Writing any $y\in \clos{D}_\nu$ (uniquely) as $ty'+(1-t)b_\nu$ for some $t\in [0,1]$ and the proper choice of $y'\in \partial(\clos{D}_\nu)$, we have that $\eta_\nu(y)=t\rho(y')+(1-t)b_\nu$. Then for our given $x\in \Delta$, $x\in B$ iff $b_{\nu_1},\ldots,b_{\nu_q}\in B$ iff $x'\in \Delta'\subset B$ and $b_\nu\in B$ iff $\rho_\nu(x')\in B$ and $b_\nu\in B$ iff $\eta_\nu(x)\in B$. Then since $\Phi'_\nu=\theta_\nu\circ\eta_\nu$, we have that $x\in B$ iff $\Phi(x)=\Psi\circ\Phi'(x)\in B$.
\end{proof}

\begin{example}\label{eg:RestrictedTriangulation}
Let $A=\clos{B(0,1)}\subset \bb{R}^2$. We will illustrate how Lemma \ref{thm:RestrictedTriangulation} may be applied to give a triangulation of $A$ respecting sign sets of $\{L_1=y, L_2=x-y\}$. Throughout, we will associate objects to their notation in the proof of Lemma \ref{thm:RestrictedTriangulation}, so that the example may aid in the parsing of the proof.

\begin{figure}
    \centering
    \includegraphics[width=.5\textwidth]{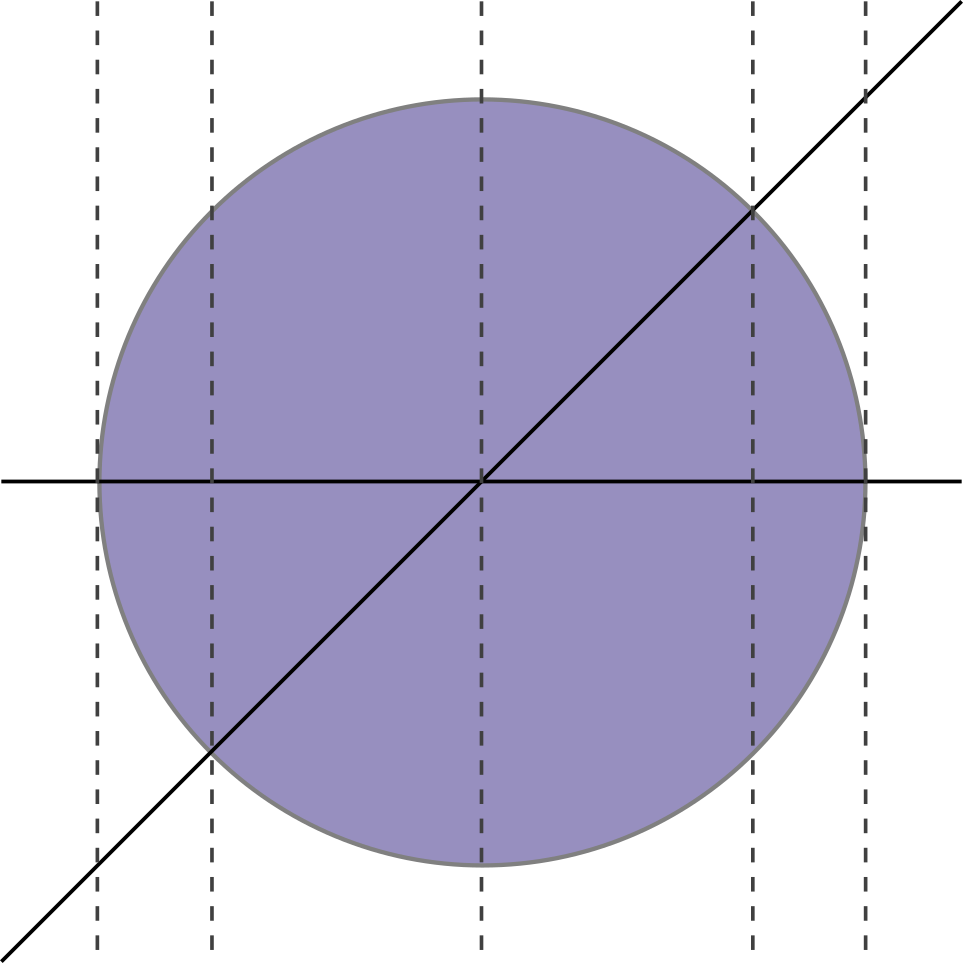}
    \caption{Cell decomposition adapted to $A$ and to sets $\{y=x\}$ and $\{y=0\}$}
    \label{fig:EgCellDecomp}
\end{figure}

We begin with a cell decomposition of $\bb{R}^2$ adapted to $A$ and to all sign sets of $\{L_1,L_2\}$. Following the most obvious choice of decomposition, we obtain a subdivision of $\bb{R}^1$ into points $\{-1\}$, $\{-\sqrt{2}/2\}$, $\{0\}$, $\{\sqrt{2}/2\}$, and $\{1\}$ and the intervals between them. The full cell decomposition is shown in Figure \ref{fig:EgCellDecomp}. Projecting to $\bb{R}^1$ for our induction, we must obtain a triangulation $(\Lambda_{n-1},\Psi_{n-1})$ of $[-1,1]$ adapted to the sets $\{-1\}$, $(-1,-\sqrt{2}/2)$, $\{-\sqrt{2}/2\}$, $(-\sqrt{2}/2,0)$, $\{0\}$, $(0,\sqrt{2}/2)$, $\{\sqrt{2}/2\}$, $(\sqrt{2}/2,1)$, and $\{1\}$ and respecting the sets $(-\infty,0)$, $\{0\}$, and $(0,\infty)$.

We apply the $n=1$ case of the algorithm outlined in the proof of Lemma \ref{thm:RestrictedTriangulation}. In that notation, we have $k=1$, $c_1=0$, and $\xi_1=-1,\xi_2=-\sqrt{2}/2,\xi_3=0,\xi_4=\sqrt{2}/2, \xi_5=1$. Then the map $\tau$ is given by $\xi_1\mapsto -2, \xi_2\mapsto -1, \xi_3\mapsto 0, \xi_4\mapsto 1, \xi_5\mapsto 2$. Hence our triangulation $(\Lambda_{n-1},\Psi_{n-1})$ of $p(A)$ is such that $\abs{\Lambda_{n-1}}=[-2,2]$ and $\Psi_{n-1}:\abs{\Lambda_{n-1}}\rightarrow p(A)$ is induced by piecewise linearly extending the pairings given by $\tau$.

To construct our triangulation of $A$ itself, we first identify our vertices $b_\nu$. For the sake of example, we will concentrate on the simplex $\delta_\beta=(0,1)$ of $\Lambda_{n-1}$. We have that $\Psi_{n-1}(\delta_\beta)\subset X_\alpha=(0,\sqrt{2}/2)$ in our cell decomposition of $p(A)$ (in fact, in this case $\Psi(\delta_\beta)=X_\alpha$), and that the barycenter $b(\delta_\beta)= \frac{1}{2}\in \bb{R}^1$. In the cell decomposition of $A$, we have four graphs defined on $(0,\sqrt{2}/2)$: let $\zeta_1$ be the lower semicircle, $\zeta_2$ be the line $y=0$, $\zeta_3$ be the line $y=x$, and $\zeta_4$ be the upper semicircle. There are a total of seven cells (graphs and bands) defined on this interval and contained in $A$. In $A$, the line $\{x=\Psi_{n-1}(b(\delta_\beta))\}$ (that is, $\{x=\sqrt{2}/4\}$) meets $\{L_1=0\}$ and $\{L_2=0\}$ at $y$-values of $c_{\beta,1}=0$ and $c_{\beta,2}= \sqrt{2}/4$ respectively. Translating to our triangulation built upon $\Lambda_{n-1}$, we have that $\{x=b(\delta_\beta)\}$ meets $\{L_1=0\}$ and $\{L_2=0\}$ at $c'_{\beta,1}=0$ and $c'_{\beta,2}= \frac{1}{2}$. In order to preserve this correspondence, we assign our vertices $b_\nu\in \bb{R}^2$ as follows.
\begin{itemize}
    \item $C_\nu$ corresponds to the graph $\zeta_1$: $b_{\nu}=(\frac{1}{2}, -1)$
    \item $C_\nu$ corresponds to the graph $\zeta_2$: $b_{\nu}=(\frac{1}{2}, 0)$
    \item $C_\nu$ corresponds to the graph $\zeta_3$: $b_{\nu}=(\frac{1}{2}, \frac{1}{2})$
    \item $C_\nu$ corresponds to the graph $\zeta_4$: $b_{\nu}=(\frac{1}{2}, \frac{3}{2})$
    \item $C_\nu$ corresponds to the band $(\zeta_1,\zeta_2)$: $b_{\nu}=(\frac{1}{2}, -\frac{1}{2})$
    \item $C_\nu$ corresponds to the band $(\zeta_2,\zeta_3)$: $b_{\nu}=(\frac{1}{2}, \frac{1}{4})$
    \item $C_\nu$ corresponds to the band $(\zeta_3,\zeta_4)$: $b_{\nu}=(\frac{1}{2}, 1)$
\end{itemize}
We use these vertices and those obtained by applying the same process to the remaining cells to build polyhedra $\clos{D}_\nu$, each of which comes with a subdivision into simplices. Taken together, these simplices give us our desired complex $\Lambda$ respecting sign sets of $\{L_1, L_2\}$, as shown in Figure \ref{fig:EgTriangulation}. There are a total of 10 2-dimensional polyhedra, with 64 2-dimensional simplices.

\begin{figure}
    \centering
    \includegraphics[width=.5\textwidth]{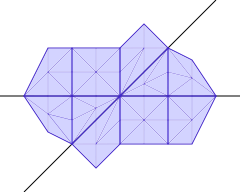}
    \caption{Triangulation of $A$ respecting sets $\{y=x\}$ and $\{y=0\}$}
    \label{fig:EgTriangulation}
\end{figure}
\end{example}

Now we prove our symmetric triangulation theorem.

\begin{thm}\label{thm:SymmetricTriangulation}
Let $A$ be a closed and bounded subset of $\bb{R}^n$ symmetric under the action of a finite reflection group $G$, and let $S_1,\ldots, S_l$ be symmetric sets which are subsets of $A$. Then there exists an equivariant triangulation $(\Lambda, \Phi)$ of $A$ adapted to $S_1,\ldots, S_l$.

Assume we have fixed a collection $\{L_1,\ldots,L_k\}$ of functions $L_i:\bb{R}^n\rightarrow \bb{R}$ given by $L_i(\bs{x})=a_{i,1}x_1+\cdots a_{i,n}x_n$, so that
\[
H=\bigcap_{i=1}^k \{L_i(\bs{x})\geq 0\}
\]
is a fundamental region of $\bb{R}^n$ with respect to $G$. Then we may choose our triangulation so that vertices of $\Lambda$ are in $\bb{Q}(\{a_{i,j}\mid 1\leq i\leq k, 1\leq j\leq n\})^n$.
\end{thm}

\begin{proof}
Let $G$ be a finite reflection group acting on $\bb{R}^n$, and if not already selected, let $\{L_1,\ldots, L_k\}$ be linear functions defining a fundamental region $H$ of $\bb{R}^n$ with respect to $G$. We apply Lemma \ref{thm:RestrictedTriangulation} to the closed, bounded, definable set $A\cap \clos{H}$, definable subsets $S_1\cap \clos{H},\ldots, S_l\cap \clos{H}$, and our collection of functions $\{L_1,\ldots, L_k\}$. The lemma gives us a finite simplicial complex $\Lambda_{\id}$ with vertices in $\bb{Q}(\{a_{i,j}\})$ and a definable homeomorphism $\Phi_{\id}:\abs{\Lambda}\rightarrow A\cap \clos{H}$. Note that since $H$ and all its walls are sign sets of $\{L_1,\ldots, L_k\}$, the triangulation $(\Lambda_{\id}, \Phi_{\id})$ respects $H$ and all its walls. In particular, $\abs{\Lambda}\subset \clos{H}$.

We now define our proposed triangulation of $S$. Let
\[
\Lambda=\bigcup_{g\in G}\{g(\Delta)\mid \Delta\in \Lambda_{\id}\}
\]
and let $\Phi:\abs{\Lambda}\rightarrow A$ be given by $\Phi(\bs{x})=g(\Phi_{\id}(\bs{x}'))$ where $\bs{x}'\in \abs{\Lambda_{\id}}$ and $g\in G$ are such that $g(\bs{x}')=\bs{x}$. We claim that (everything is well-defined and) $\Lambda$ and $\Phi$ give an equivariant triangulation of $A$ adapted to $S_1,\ldots, S_l$.

$\Lambda$ is a symmetric simplicial complex: We have that if $\Delta\in \Lambda_{\id}$ and $g\in G$, $g(\Delta)$ remains a simplex in $\bb{R}^n$ by linearity of $g$, and so $\Lambda$ is a collection of simplices. The symmetry property for $\Lambda$ then holds by construction, so it remains to show that for $\Delta_i,\Delta_j\in \Lambda$, $\clos{\Delta}_i\cap\clos{\Delta}_j$ is the closure of some simplex in $\Lambda$. Without loss of generality, assume $\Delta_1\in \Lambda_{\id}$ and $\Delta_2=g(\Delta_2')$ for $\Delta_2'\in \Lambda_{\id}$ and $g\in G$. This means $\clos{\Delta}_1\cap \clos{g(\Delta_2')}\subset \clos{H}\cap g(\clos{H})$. As described in Subsection \ref{sect:SymmetryBackground}, $\clos{H}\cap g(\clos{H})=H_{\lambda_g}$ is a set of the form
\[
H_{\lambda_g}=\paren{\bigcap_{i\in \lambda_g}\{L_i(\bs{x})=0\}}\cap\paren{\bigcap_{i\in \{1,\ldots,k\}\setminus \lambda_g} \{L_i(\bs{x})\geq 0\}}
\]
which in particular is a sign set of $\{L_1,\ldots, L_k\}$. Since $(\Lambda_{\id},\Phi_{\id})$ is hence adapted to $A\cap H_{\lambda_g}$, we have that $\clos{\Delta}_1\cap H_{\lambda_g}$ is also the closure of a simplex of $\Lambda_{\id}$. Since points of $H_{\lambda_g}$ are fixed under the action of $g$, $\clos{\Delta}_2\cap H_{\lambda_g}=\clos{\Delta}_2'\cap H_{\lambda_g}$ is the closure of a simplex of $\Lambda_{\id}$ as well. Then $\clos{\Delta}_1\cap\clos{\Delta}_2=(\clos{\Delta}_1\cap H_{\lambda_g})\cap(\clos{\Delta}_2\cap H_{\lambda_g})$ is an intersection of closures of simplices of $\Lambda_{\id}$, and hence a common face of $\clos{\Delta}_1$ and $\clos{\Delta}_2$.

Vertices of $\Lambda$ are in $\bb{Q}(\{a_{i,j}\})^n$: We will show that, for $g\in G$ and any point $\bs{x}$ in $\bb{Q}(\{a_{i,j}\})^n$, $g(\bs{x})\in \bb{Q}(\{a_{i,j}\})^n$. Since any element $g$ of $G$ can be written as a product of those elements $g_1,\ldots, g_k$, where the action of $g_i$ is reflection through the linear hyperplane $L_i=0$, we may assume $g=g_i$ for some $1\leq i\leq k$. The reflection of $\bs{x}$ through $L_i=0$ is given by
\[
g_i(\bs{x})=\bs{x}-2\frac{(\bs{x},\bs{r}_i)\bs{r}_i}{(\bs{r}_i,\bs{r}_i)}
\]
where we take $\bs{r}_i$ to be the vector $\gen{a_{i,1},\ldots,a_{i,n}}$ (which is perpendicular to the hyperplane $L_i=0$). Then since we have assumed that we are using the standard inner product on $\bb{R}^n$, the coordinates of $g_i(\bs{x})$ are all also in $\bb{Q}(\{a_{i,j}\})$, as desired.

$\Phi$ is well-defined: Let $\bs{x}\in \abs{\Lambda}$ and say that $\bs{x}=g_1(\bs{x}_1)=g_2(\bs{x}_2)$ for $\bs{x}_1,\bs{x}_2\in \abs{\Lambda_{\id}}$ and $g_1,g_2\in G$. We must show that $g_1(\Phi_{\id}(\bs{x}_1))=g_2(\Phi_{\id}(\bs{x}_2))$. Because $\bs{x}_1=g_1^{-1}(g_2(\bs{x}_2))$ with $\bs{x}_1,\bs{x}_2\in \abs{\Lambda_{\id}}$, we have that $\bs{x}_1\in \abs{\Lambda_{\id}}\cap g_1^{-1}\circ g_2(\abs{\Lambda_{\id}}) \subset H_{\lambda_{g_1^{-1}g_2}}$. Since $g_1^{-1}\circ g_2$ fixes the points of $H_{\lambda_{g_1^{-1}g_2}}$, we obtain that $\bs{x}_1=\bs{x}_2$. Since $\Phi_{\id}$ carries $H_{\lambda_{g_1^{-1}g_2}}$ to itself, we have that $\Phi_{\id}(\bs{x}_1)\in H_{\lambda_{g_1^{-1}g_2}}$, as is $g_1^{-1}\circ g_2(\Phi_{\id}(\bs{x}_1))$, and so $\Phi_{\id}(\bs{x}_1)=g_1^{-1}\circ g_2(\Phi_{\id}(\bs{x}_1))$, i.e. we have obtained that $g_1(\Phi_{\id}(\bs{x}_1))=g_2(\Phi_{\id}(\bs{x}_2))$. That the image of $\Phi$ is $A$ follows from the symmetry of $A=\{g(\bs{x})\mid \bs{x}\in A\cap \clos{H}\text{ and } g\in G\}$.

$\Phi$ is a homeomorphism: We have already established that the surjectivity of $\Phi$ follows from the surjectivity of $\Phi_{\id}$ onto $A\cap \clos{H}$. To show injectivity, say $\Phi(\bs{x}_1)=\Phi(\bs{x}_2)$ for some $\bs{x}_1,\bs{x}_2\in \abs{\Lambda}$, i.e.  $g_1(\Phi_{\id}(\bs{x}_1'))=g_2(\Phi_{\id}(\bs{x}_2'))$ for some $\bs{x}_1',\bs{x}_2'\in \abs{\Lambda_{\id}}$ and $g_1,g_2\in G$. Then since $\Phi_{\id}(\bs{x}_1')=g_1^{-1}\circ g_2(\Phi_{\id}(\bs{x}_2'))$, both are in $H_{\lambda_{g_1^{-1}g_2}}$, and so $\Phi_{\id}(\bs{x}_1')=\Phi_{\id}(\bs{x}_2')$. By the injectivity of $\Phi_{\id}$, this means $\bs{x}_1'=\bs{x}_2'$. Finally, since $\bs{x}_1'=\bs{x}_2'$ is in $H_{\lambda_{g_1^{-1}g_2}}$, $\bs{x}_1'=g_1^{-1}\circ g_2(\bs{x}_2')$, so we have $\bs{x}_1=g_1(\bs{x}_1')=g_2(\bs{x}_2')=\bs{x}_2$. Continuity of $\Phi$ follows from continuity on $g(\Lambda_{\id})$ for each $g\in G$ and agreement on the boundaries.

$\Phi$ is equivariant under the action of $G$ by construction. The preimage under $\Phi$ of each set $S_i$ among $S_1,\ldots, S_l$ is the union of the simplices $g(\Delta)$ for $g\in G$ and $\Delta\in \Lambda_{\id}$ such that $\Delta\subset\Phi_{\id}^{-1}(S_i\cap \clos{H})$. Hence, $(\Lambda,\Phi)$ gives our desired triangulation.
\end{proof}

\subsection{Triangulation of Definable Functions}\label{sect:TriangulationDefinableFns}

In the proofs in \cite{gabrielov2009approximation}, to ensure that the triangulation we use is properly compatible with the family of sets $\{S_\delta\}_{\delta>0}$, the authors invoke the triangulation of definable functions. The original theorem from \cite{coste2000introduction} is below. We will proceed to prove a version for functions symmetric relative to the action of some finite reflection group $G$.

\begin{thm}[Triangulation of Definable Functions, \cite{coste2000introduction} Theorem 4.5]\label{thm:Coste4.5}
Let $A$ be a closed and bounded definable subset of $\rcf^n$ and $f:A\rightarrow \rcf$ a continuous definable function. Then there exists a finite simplicial complex $\Lambda$ in $\rcf^{n+1}$ and a definable homeomorphism $\rho: \abs{\Lambda}\rightarrow A$ such that $f\circ \rho$ is an affine function on each simplex of $\Lambda$. Moreover, given $S_1,\ldots, S_l$ definable subsets of $A$, we may choose the triangulation $\rho:\abs{\Lambda}\rightarrow A$ to be adapted to the $S_i$.
\end{thm}

Following \cite{coste2000introduction}, we will prove a more general result concerning the triangulation of symmetric definable subsets of $\bb{R}\times\bb{R}^n$, which we will apply to what is essentially the graph of our function $f$. Our procedure is similar to the one used for symmetric definable sets. Let $\pi:\rcf\times \rcf^n\rightarrow \rcf$ denote projection on the first coordinate.

\begin{lemma}[ref \cite{coste2000introduction} Proposition 4.6 and Proposition 4.8]\label{thm:RestrictedTriangulationFn}
Let $A$ be a closed, bounded, definable subset of $\rcf\times \rcf^n$, and let $S_1,\ldots, S_l$ be definable subsets of $A$. Let $\{L_1,\ldots, L_k\}$ be a collection of functions with $L_i:\rcf\times \rcf^n\rightarrow\rcf$ given by $L_i:(y,x_1,\ldots, x_n)\mapsto a_{i,0}+a_{i,1}x_1+\cdots+a_{i,n}x_n$. Then there exists a triangulation $(\Lambda,\Phi)$ of $A$ adapted to $S_1,\ldots, S_l$, respecting all sign sets of $\{L_1,\ldots, L_k\}$, and having vertices of $\Lambda$ in $\bb{Q}\times \bb{Q}(\{a_{i,j}\mid 1\leq i\leq k, 0\leq j\leq n\})^n$, as well as a definable homeomorphism $\tau:\rcf\rightarrow \rcf$, such that $\tau\circ\pi\circ \Phi=\pi_{\restriction\abs{\Lambda}}$.
\end{lemma}
\begin{proof}
We induct on $n$. In the $n=0$ case, since we are assuming all functions in our collection $\{L_1,\ldots, L_k\}$ to be independent of the first coordinate, we have nothing more to prove than \cite{coste2000introduction} does, and so may as there choose a finite partition of $\rcf$ adapted to $A$ and $S_1,\ldots, S_l$. Letting $x_1<\cdots< x_p$ be the points in $\rcf$ defining the partition, we take $\Lambda$ to be the points $i\in \{1,\ldots, p\}$ and intervals $[i,i+1]$ such that ${x_i}\in A$ or $[x_i,x_{i+1}]\subset A$. We let $\tau(x_i)=i$ and extend piecewise affinely to a map $\tau: \rcf\rightarrow \rcf$.

Assume $n>0$ and that the statement holds for $n-1$. We may again assume all $S_i$ are closed (as described in the proof of \cite{coste2000introduction} Theorem 4.4). We set $F'_0$ to be the boundary of $A$, $F'_i$ to be the boundary of $S_i$ for each $i$, and $F'=F'_0\cup F'_1\cup\ldots\cup F'_l$. Since $F'$ is definable with dimension at most $n$, we have finitely many $c\in \rcf$ for which the set $\{\bs{x}\in \rcf^n\mid (c,\bs{x})\in F'\}$ has dimension $n$. Let $C$ be the set of all such $c$. For reasons of dimension, the proof in \cite{coste2000introduction} considers sets $F_i$ given by taking the union of $\clos{F'_i\setminus(C\times \rcf^n)}$ with the boundary of $F'_i\cap (C\times \rcf^n)$ in $C\times \rcf^n$. We choose a cell decomposition of $\rcf\times \rcf^n$ adapted to $F_0,\ldots, F_l$, the sets $\{c\}\times \rcf^n$ for each $c\in C$, and all sign sets of $\{L_1,\ldots, L_k\}$.

Let $p:\rcf\times \rcf^n\rightarrow \rcf\times \rcf^{n-1}$ be the projection on the first $n$ coordinates. Our cell decomposition partitions $p(A)$ into definably connected subsets $X_\alpha$. By our inductive hypothesis, we have a triangulation $(\Lambda_{n-1}, \Psi_{n-1})$ of $\rcf\times \rcf^{n-1}$ which is adapted to each $X_\alpha$, respects sign sets of $\{L_1',\ldots, L_{m'}'\}$ (where this collection is defined in a manner analogous to the proof of Lemma \ref{thm:RestrictedTriangulation}), and with vertices of $\Lambda_{n-1}$ in $\bb{Q}\times \bb{Q}(\{a_{i,j}'\mid 1\leq i\leq m', 0\leq j\leq n-1\})^{n-1}= \bb{Q}\times \bb{Q}(\{a_{i,j}\mid 1\leq i\leq m, 0\leq j\leq n\})^{n-1}$. We also have a map $\tau':\rcf\rightarrow \rcf$, having the property that $\tau'\circ\pi_{n,1}\circ \Psi_{n-1}={\pi_{n,1}}\restriction_{\abs{\Lambda_{n-1}}}$ (where $\pi_{n,1}$ is the projection $\rcf\times \rcf^{n-1}\rightarrow \rcf$ on the first coordinate). Now, if we follow the remaining steps in the proof of Lemma \ref{thm:RestrictedTriangulation}, we obtain a triangulation $(\Lambda,\Phi)$ of $A$ which is adapted to $S_1,\ldots, S_l$, respects sign sets of $\{L_1,\ldots, L_k\}$, and has vertex coordinates in $\bb{Q}\times \bb{Q}(\{a_{i,j}\})^n$. $\Lambda$ and $\Phi$ are such that $p\circ \Phi=\Psi_{n-1}\circ p_{\restriction\abs{\Lambda}}$. Taking $\tau=\tau'$, this means that our property $\tau\circ\pi\circ \Phi=\pi_{\restriction\abs{\Lambda}}$ holds.
\end{proof}

\begin{lemma}[ref \cite{coste2000introduction} Proposition 4.6 and Proposition 4.8]\label{thm:SymmetricTriangulationFnPrep}
Let $A$ be a closed, bounded, definable subset of $\bb{R}\times\bb{R}^n$ symmetric under the action of some finite reflection group $G$ on $\bb{R}^n$ extended to $\bb{R}\times \bb{R}^n$, and let $S_1,\ldots, S_l$ be definable symmetric subsets of $A$. Then there exists an equivariant triangulation $(\Lambda, \Phi)$ of $A$ adapted to $S_1,\ldots, S_l$ and a definable homeomorphism $\tau:\bb{R}\rightarrow \bb{R}$ having the property that $\tau\circ\pi\circ \Phi=\pi_{\restriction\abs{\Lambda}}$.

Assume we have fixed a collection $\{L_1,\ldots,L_k\}$ of functions $L_i:\bb{R}^n\rightarrow \bb{R}$ given by $L_i(\bs{x})=a_{i,1}x_1+\cdots a_{i,n}x_n$, so that
\[
H=\bigcap_{i=1}^k \{L_i(\bs{x})>0\}
\]
is a fundamental region of $\bb{R}^n$ with respect to $G$. Then we may choose our triangulation so that vertices of $\Lambda$ are in $\bb{Q}\times \bb{Q}(\{a_{i,j}\mid 1\leq i\leq k, 1\leq j\leq n\})^n$.
\end{lemma}
\begin{proof}
This is analogous to the proof of Theorem \ref{thm:SymmetricTriangulation}. Let $H$ be a fundamental region of $\bb{R}^n$ with respect to $G$. If not already specified, we take $\{L_1,\ldots, L_k\}$ with $L_i(\bs{x})=a_{i,1}x_1+\cdots+a_{i,n}x_n$ to be a collection functions defining $H$. For each $i$, let $\tilde{L}_i:\bb{R}\times \bb{R}^n\rightarrow \bb{R}$ with $\tilde{L}_i(y,\bs{x})=L_i(\bs{x})$. We apply Lemma \ref{thm:SymmetricTriangulationFn} to $A\cap (\bb{R}\times \clos{H})$, subsets $S_1\cap (\bb{R}\times \clos{H}),\ldots, S_l\cap(\bb{R}\times \clos{H})$, and linear functions $\{\tilde{L}_1,\ldots, \tilde{L}_k\}$. We obtain a triangulation $(\Lambda_{\id},\Phi_{\id})$ with $\abs{\Lambda_{\id}}\subset \bb{R}\times \clos{H}$ and coordinates of vertices of $\Lambda_{\id}$ in $\bb{Q}\times \bb{Q}(\{a_{i,j}\})^n$, and $\Phi_{\id}:\abs{\Lambda_{\id}}\rightarrow A\cap (\bb{R}\times \clos{H})$, and also obtain a map $\tau:\bb{R}\rightarrow \bb{R}$, having the property that $\tau\circ \pi\circ \Phi_{\id}=\pi_{\restriction\abs{\Lambda_{\id}}}$. Again, we take
\[
\Lambda=\bigcup_{g\in G}\{g(\Delta)\mid \Delta\in \Lambda_{\id}\}
\]
and $\Phi:\abs{\Lambda}\rightarrow A$ given by $\Phi(\bs{x})=g(\Phi_{\id}(\bs{x}'))$, where $g\in G$ and $\bs{x}'\in \abs{\Lambda_{\id}}$ are such that $\bs{x}=g(\bs{x}')$. As in the proof of Theorem \ref{thm:SymmetricTriangulation}, this provides a symmetric triangulation of $A$ adapted to $S_1,\ldots, S_l$, with all vertices of $\Lambda$ in $\bb{Q}\times \bb{Q}(\{a_{i,j}\})^n$.

It remains to show that $\tau\circ\pi\circ \Phi=\pi_{\restriction\abs{\Lambda}}$. Take $\bs{x}\in \abs{\Lambda}$. Then we have that $\bs{x}=g(\bs{x}')$ for some $g\in G$ and $\bs{x}'\in \abs{\Lambda_{\id}}$. Note that by the definition of our action of $G$ extended to $\bb{R}\times \bb{R}^n$, we have that $\pi$ is symmetric relative to this action. Then
\begin{align*}
    \tau\circ\pi\circ\Phi(\bs{x})&=\tau\circ\pi(g(\Phi_{\id}(\bs{x}')))\\
    &=\tau\circ\pi(\Phi_{\id}(\bs{x}'))\\
    &=\pi(\bs{x}')\\
    &=\pi(g(\bs{x}'))=\pi(\bs{x})
\end{align*}
as desired.
\end{proof}

Our theorem now follows, applying Lemma \ref{thm:SymmetricTriangulationFnPrep} to $A'=\{(f(\bs{x}),\bs{x})\mid \bs{x}\in A\}$.

\begin{thm}\label{thm:SymmetricTriangulationFn}
Let $A$ be a closed, bounded, definable subset of $\bb{R}^n$ symmetric under the action of
a finite reflection group $G$, and let $f:A\rightarrow \bb{R}$ be a continuous definable function symmetric relative to the action of $G$. Then there exists a finite symmetric simplicial complex $\Lambda$ in $\bb{R}^{n+1}$ and a definable equivariant homeomorphism $\rho: \abs{\Lambda}\rightarrow A$ such that $f\circ \rho$ is an affine function on each simplex of $\Lambda$. Moreover, given $S_1,\ldots, S_l$ definable symmetric subsets of $A$, we may choose the triangulation $\rho: \abs{\Lambda}\rightarrow A$ to be adapted to the sets $S_i$.

Assume we have fixed a collection $\{L_1,\ldots,L_k\}$ of functions $L_i:\bb{R}^n\rightarrow \bb{R}$ given by $L_i(\bs{x})=a_{i,1}x_1+\cdots a_{i,n}x_n$, so that
\[
H=\bigcap_{i=1}^k \{L_i(\bs{x})>0\}
\]
is a fundamental region of $\bb{R}^n$ with respect to $G$. Then we may choose our simplicial complex $\Lambda$ so that all vertices of $\Lambda$ are in $\bb{Q}\times \bb{Q}(\{a_{i,j}\mid 1\leq i\leq k, 1\leq j\leq n\})^n$.
\end{thm}
\begin{proof}
Consider the set $A'=\{(f(\bs{x}), \bs{x})\in \bb{R}\times \bb{R}^n\mid \bs{x}\in A\}$. Since $f$ is symmetric relative to the action of $G$ on $\bb{R}^n$, $A'$ is a symmetric set relative to the action induced by $G$ on $\bb{R}\times \bb{R}^n$, and projection on the last $n$ coordinates gives an equivariant definable homeomorphism $p:A'\rightarrow A$. By Lemma \ref{thm:SymmetricTriangulationFnPrep}, we have a symmetric triangulation $(\Lambda, \Phi)$ of $A'$ which is adapted to the sets $S_1',\ldots, S_l'$ (with $S_i'=\{(f(\bs{x}),\bs{x})\mid \bs{x}\in S_l\}$) and has vertices in $\bb{Q}\times \bb{Q}(\{a_{i,j}\})^n$, and a definable function $\tau:\bb{R}\rightarrow \bb{R}$ such that $\tau\circ\pi\circ\Phi=\pi_{\restriction\abs{\Lambda}}$ (where $\pi:\bb{R}\times \bb{R}^n\rightarrow \bb{R}$ is projection on the first coordinate). Applying $f$ to $\bs{x}\in A$ is equivalent to applying $\pi$ to $(f(\bs{x}),\bs{x})\in A'$, and so taking $\rho:\abs{\Lambda}\rightarrow A$ to be $p\circ\Phi$, we have that $f\circ\rho=\tau^{-1}\circ\pi_{\restriction\abs{\Lambda}}$, which is an affine map on each simplex of $\Lambda$ by construction.
\end{proof}

\subsection{Equivariance and Hardt Triviality}\label{sect:HardtTriviality}

We will also want an equivariant version of Hardt Triviality for o-minimal sets. Because the proof uses the same sort of argument as equivariant triangulation, we include it within this section.

\begin{definition}[\cite{van1998tame} Chapter 9 Definitions 1.1 and 1.8]\label{def:VanDenDries9.1.1}
Let $X\subset \rcf^n$ and $A\subset \rcf^m$ be definable sets and $f:X\rightarrow A$ a continuous definable function. For $A'\subset A$, we say that $f$ is \emph{definably trivial} over $A'$ if for any $a\in A'$ there is a definable homeomorphism $h:f^{-1}(A')\rightarrow f^{-1}(a)\times A'$ such that the diagram
\[
\begin{tikzcd}
f^{-1}(A') \arrow[dr, "f"] \arrow[rr, "h"] & & f^{-1}(a)\times A'\arrow[dl, "\pi"']\\
& A'
\end{tikzcd}
\]
commutes (where $\pi$ is the projection on the second coordinate).

For $X_1,\ldots, X_l$ definable subsets of $X$, we say that the definable trivialization $h$ \emph{respects} $X_1,\ldots, X_l$ if $h$ maps each $X_j\cap f^{-1}(A')$ homeomorphically to $(X_j\cap f^{-1}(a))\times A'$.
\end{definition}

\begin{thm}[\cite{van1998tame} Chapter 9 Theorem 1.7]\label{thm:VanDenDries9.1.7}
Let $X\subset \rcf^n$ and $A\subset \rcf^m$ be definable sets, $f:X\rightarrow A$ a continuous definable function, and $X_1, \ldots, X_l$ definable subsets of $X$. Then we may partition $A$ into a finite number of definable subsets $A_i$ such that $f$ is definably trivial over each $A_i$ in a manner that respects each of $X_1,\ldots, X_l$.
\end{thm}

We would like to show that, if $f$ is a symmetric function, then the homeomorphisms guaranteed by the trivialization are equivariant.

\begin{thm}\label{thm:EquivariantHardtTriviality}
Let $G$ be a finite reflection group acting on $\bb{R}^n$, and let $X\subset \bb{R}^n$ and $X_1,\ldots, X_l\subset X$ be symmetric relative to $G$. Say that $A\subset\bb{R}^m$ is definable  and $f:X\rightarrow A$ is a continuous, definable function symmetric relative to $G$. Then we may partition $A$ into a finite number of definable subsets $A_i$ such that  for each $i$ there is an equivariant homeomorphism $h_i:f^{-1}(A_i)\rightarrow f^{-1}(a_i)\times A_i$ giving a definable trivialization of $f$ over $A_i$ (where $a_i$ is any element of $A_i$, and the action on $f^{-1}(a_i)\times A$ is given by $g(x,y)=(g(x),y)$). Furthermore, each trivialization $h_i$ respects the subsets $X_1,\ldots, X_l$.
\end{thm}
\begin{proof}
Let $H$ be a fundamental region of $\bb{R}$ with respect to $G$. We apply Theorem \ref{thm:VanDenDries9.1.7} to $f_{\restriction\clos{H}}:\clos{H}\cap X\rightarrow A$, asking that our definable trivializations respect the sets $\clos{H}\cap X_1,\ldots, \clos{H}\cap X_l$ as well as each $H_\lambda \cap X$ for $H_{\lambda}$ an intersection of walls of $H$. We obtain a finite partition $\{A_i\}$ of $A$, as well as homeomorphisms $h_{\id,i}:f^{-1}_{\restriction\clos{H}}(A_i)\rightarrow f^{-1}_{\restriction\clos{H}}(a_i)\times A_i$ which respect both the portions of each $X_j$ within $\clos{H}$ and the intersections of $X$ with the walls of $H$.

Since $f$ is a symmetric function, we know that for $a_i\in A_i$,
\[
f^{-1}(a_i)=\bigcup_{g\in G} g(f^{-1}_{\restriction\clos{H}}(a_i))
\]
and that the same holds true for $f^{-1}(A_i)$. We define $h_i:f^{-1}(A_i)\rightarrow f^{-1}(a_i)\times A_i$ by $h_i(\bs{x})=g(h_{\id,i}(\bs{x}'))$, where $\bs{x}'\in f^{-1}(A_i)\cap \clos{H}$ and $g\in G$ are such that $g(\bs{x}')=\bs{x}$. We must show that $h_i$ is a definable trivialization of $f$.

We may use arguments similar to those of Theorem \ref{thm:SymmetricTriangulation} to establish that $h_i$ is a homeomorphism from $f^{-1}(A_i)$ to $f^{-1}(a_i)\times A_i$. The map $h_i$ is equivariant by construction, and $h_i$ sends each $X_j\cap f^{-1}(A_i)$ homeomorphically to $(X_j\cap f^{-1}(a_i))\times A_i$ by construction and the symmetry of each $X_j$. Say that $\bs{x}\in f^{-1}(A_i)$ with $f(\bs{x})=\bs{y}$. Then if $\bs{x}=g(\bs{x}')$ for some $g\in G$ and $\bs{x}'\in f^{-1}(A_i)\cap \clos{H}$, by symmetry we know that $f(\bs{x}')=\bs{y}$ as well, and so $h(\bs{x}')=(\bs{z},\bs{y})$ for some $\bs{z}\in f^{-1}(a_i)$. Then $h(\bs{x})=g(\bs{z},\bs{y})=(g(\bs{z}),\bs{y})$, so $h_i$ is indeed a definable trivialization of $f$ over $A_i$.
\end{proof}

\section{Equivariant Versions of some Topological Results}\label{sect:EquivarianceTopology}

We address some topological aspects of the construction in this section. Specifically, for $X$ a regular CW complex whose cells are convex polyhedra (so specifically, for $X$ a simplicial complex), we explicitly describe certain maps between $X$ and the order complex of the face poset of $X$, in aid of demonstrating equivariance. We also establish an equivariant version of a Nerve Theorem due to Bj\"{o}rner.

\begin{definition}[\cite{bjorner2005combinatorics} Appendix A2]
Let $P$ be a partially ordered set. The \emph{order complex} of $P$, $\Delta(P)$, is the abstract simplicial complex with vertex set the elements of $P$ and simplices given by finite chains $x_0<\cdots<x_d$ of elements of $P$.
\end{definition}

The order complex of the face poset of a simplicial complex coincides with the first barycentric subdivision of that simplicial complex. For a regular CW complex, then, taking the order complex of the face poset serves as a generalized barycentric subdivision:

\begin{thm}[\cite{lundell2012topology} Theorem 1.7]\label{thm:Lundell1.7}
Let $X$ be a regular CW complex. Then $\Delta(\fpos{X})$ is homeomorphic to $X$.
\end{thm}

The proof involves choosing a point in the interior of each cell $\sigma$ to act as the barycenter of the cell, and sending the vertex of $\Delta(\fpos{X})$ corresponding to $\sigma$ to that barycenter. Elsewhere, the homeomorphism arises from the fact that the characteristic maps used to assemble $X$ as a regular CW complex endow each cell with the structure of a cone with the barycenter as vertex and boundary as base. Unfortunately, this construction as it stands is not specific enough to ensure equivariance without some sort of equivariance condition on the characteristic maps. However, if we know that each cell of $X$ is, for example, a convex polyhedron, we may make our choices in a manner sufficiently canonical to ensure equivariance. Note that by polyhedron in this context we mean a bounded subset of $\rcf^n$ obtained by intersecting a finite number of affine half planes.

\begin{definition}\label{def:CentroidalHomeomorphism}
Let $X\subset \rcf^n$ be a regular CW complex in which each cell is a convex polyhedron. Then we will refer to the homeomorphism $\Psi$ described below as the \emph{centroidal} homeomorphism $\abs{\Delta(\fpos{X})}\rightarrow X$.

We define $\Psi$ inductively. Let $X^k$ denote the $k$-skeleton of $X$. In the $k=0$ case, let $\Psi^0: \abs{\Delta(\fpos{X^0})}\rightarrow X^0$ send the vertex $v_\sigma$ of $\Delta(\fpos{X^0})$ corresponding to the $0$-dimensional cell $\sigma=\{x\}$ to the point $x\in X$.

Now say $k\geq 1$. Assume we have defined $\Psi^{k-1}:\abs{\Delta(\fpos{X^{k-1}})}\rightarrow X^{k-1}$ and established that $\Psi^{k-1}$ is a homeomorphism. We may identify $\Delta(\fpos{X^{k-1}})$ with the subset of $\Delta(\fpos{X^k})$ consisting of simplices whose vertices correspond to cells of $X$ of dimension less than $k$; on this subset, let $\Psi^k$ agree with $\Psi^{k-1}$. Now, assume $\sigma$ is a cell of dimension $k$, and let $v_\sigma$ be the vertex of $\Delta(\fpos{X^k})$ corresponding to $\sigma$. Define $\Psi^k(v_\sigma)=b_\sigma$, where $b_\sigma$ is the centroid of the cell $\sigma$ (which is in the interior of $\sigma$ by convexity). Say $x\in \abs{\Delta(\fpos{X^k})}$. Then we may uniquely write $x=tv_\sigma+(1-t)x'$, where $\sigma$ is a cell of $X$ of dimension $k$, $x'$ is in the realization of some simplex of $\Delta(\fpos{X^{k-1}})$, and $t\in [0,1]$. Set $\Psi^k(x)=tb_\sigma+(1-t)\Psi^{k-1}(x')$. Since each cell $\sigma$ of $X$ can be seen as a cone with base $\partial \sigma$ and vertex $b_\sigma$, this gives a well-defined homeomorphism to $X^k$. Since $X=\bigcup_k X^k$, we inductively obtain our desired homeomorphism $\Psi:\abs{\Delta(\fpos{X})}\rightarrow X$.
\end{definition}

When we refer to the barycentric subdivision of a polyhedral CW complex $X\subset \rcf^n$, we will mean the simplicial complex in $\rcf^n$ whose geometric realization is equal as a set to $X$ and whose cell structure is inherited from this map.

\begin{remark}\label{thm:SymmetricOrderComplex}
Let $X\subset \rcf^n$ be a regular CW complex whose cells are all convex polyhedra. If $G$ is a group acting linearly on $X$ such that $X$ is symmetric as a CW complex under the action of $G$, then the action of $G$ on $X$ induces an action of $G$ on $\Delta(\fpos{X})$ under which $\Delta(\fpos{X})$ is symmetric, and the homeomorphism $\Psi:\abs{\Delta(\fpos{X})}\rightarrow X$ of Definition \ref{def:CentroidalHomeomorphism} is equivariant.
\end{remark}

We will also want to consider subsets of simplicial or polyhedral CW complexes which are not full subcomplexes. Let $X$ be polyhedral CW complex and let $Y$ be a union of cells of $X$. Gabrielov and Vorobjov in \cite{gabrielov2009approximation} Remark 2.12 describe how we may consider only those cells of the barycentric subdivision of $X$ which are contained with their closures in $Y$. This gives us a subcomplex of the barycentric subdivision of $X$ which is homotopy equivalent to $Y$, and which may replace $Y$ in our applications. Though this procedure is relatively standard, we will describe the contraction explicitly so that we may ensure equivariance.

\begin{definition}\label{def:BarycentricRetraction}
Let $X\subset \rcf^n$ be a regular CW complex whose cells are all convex polyhedra, and let $Y$ be a union of cells of $X$. Let $\widehat{X}$ be the barycentric subdivision of $X$, and let $\widehat{Y}=\{\Delta\in \widehat{X}\mid \Delta\subset Y\}$. Define the \emph{barycentric retraction} of $Y$ to be
\[
\br(Y)=\{\Delta\in \widehat{X}\mid \clos{\Delta}\subset Y\}
\]
\end{definition}

\begin{prop}\label{thm:BarycentricRetractingMap}
     Say $Y$ is a union of cells of some regular CW complex $X\subset \rcf^n$ whose cells are all convex polyhedra. Then $\br(Y)$ is homotopy equivalent to $Y$
\end{prop}
\begin{proof}
    We construct a homotopy $h_{\br}:[0,1]\times Y\rightarrow Y$, which we will term the \emph{barycentric retracting map}.

    Observe that $\br(Y)\subset Y$ is a subcomplex of $\widehat{X}$, whose vertex set is precisely those $0$-simplices of $\widehat{X}$ corresponding to cells of $Y$. In fact, the simplices of $\br(Y)$ correspond precisely to flags of cells in $Y$. This means that $Y=\emptyset$ iff $\br(Y)=\emptyset$.

    Let $\Delta=\Delta(v_0,\ldots, v_d)$ be a simplex of $\widehat{Y}$, and define the set $J_\Delta=\{v_{i_0},\ldots,v_{i_{d'}}\}=\{v_i\in\{v_0,\ldots,v_d\}\mid v_i\in \br(Y)\}$. Then if we let $\Delta'=\Delta(v_{i_0},\ldots, v_{i_{d'}})$, it is clear from the correspondence of simplices of $\br(Y)$ to flags of cells in $Y$ that $J_\Delta$ is nonempty and that $\clos{\Delta}\cap \abs{\br(Y)}=\clos{\Delta}'$. Say that $x\in Y$ with $x\in \clos{\Delta}$. Then we may write $x$ uniquely as $x=\sum_{i=0}^d t_{v_i} v_i$, where $0\leq t_{v_i}\leq 1$ for each $i$ and $\sum t_{v_i} = 1$. Then, for $t\in [0,1]$, let
    \[
    h_{\br}(t,x)=t \sum_{v_i\not\in J_\Delta} t_{v_i} v_{i} + \frac{1-t\sum_{v_i\not\in J_\Delta} t_{v_i}}{\sum_{v_i\in J_\Delta} t_{v_i}} \sum_{v_i\in J_\Delta}t_{v_i} v_i
    \]
    From this, we obtain our desired map $h_{\br}:[0,1]\times Y\rightarrow \abs{\widehat{Y}}$, noting that the definition of $h_{\br}$ agrees on the boundaries of simplices. Observe that $h_{\br}(0,-):Y\rightarrow \abs{\br(Y)}$ serves as homotopy inverse to the inclusion $\abs{\br(Y)}\hookrightarrow Y$.
\end{proof}

\begin{prop}\label{thm:BREquivariance}
Let $X\subset\rcf^n$ be a regular CW complex whose cells are all convex polyhedra, and let $Y$ be a union of cells of $X$. If $X$ is symmetric as a CW complex under the action of some group $G$ acting linearly on $X$ and $Y$ is also symmetric relative to $G$, then $h_{\br}:[0,1]\times Y\rightarrow Y$ is equivariant (where the action of $G$ on $[0,1]\times Y$ is given by $g(t,x)=(t,g(x))$ for all $g\in G$).
\end{prop}
\begin{proof}
From the symmetry of $X$ and $Y$, we have that $\widehat{X}$, $\widehat{Y}$, and $\br(Y)$ are all symmetric. Hence, for $\Delta\in \widehat{Y}$, $g(J_\Delta)=J_{g(\Delta)}$. Equivariance of $h_{\br}$ then follows from the linearity of the action of $G$.
\end{proof}

Even when $Y$ is not a full subcomplex of $X$, we may make sense of the simplicial complex $\Delta(\fpos{Y})$. As described in Definition \ref{def:BarycentricRetraction}, we may identify $\Delta(\fpos{Y})$ with $\br(Y)$. Furthermore, if $X$ and $Y$ are both symmetric, then the homeomorphism $\abs{\Delta(\fpos{Y})}\rightarrow \abs{\br(Y)}$ is equivariant.

We conclude by addressing equivariance in B\"{o}rner's Nerve Theorem, which will be used in Section \ref{sect:Tau}.

\begin{definition}[\cite{gabrielov2009approximation} Definition 2.9]\label{def:Gabrielov2.9}
Let $\{X_i\}_{i\in I}$ be a family of sets. The \emph{nerve} of $\{X_i\}_{i\in I}$ is the abstract simplicial complex $\mc{N}$ having vertex set $I$ and simplices given by the finite subsets $\sigma$ of $I$ such that
\[
\bigcap_{i\in \sigma} X_i\neq \emptyset
\]
\end{definition}

A variety of theorems exist which relate a space covered by a family of sets to the nerve of that cover. Bj\"{o}rner's Nerve Lemma has the advantage of only requiring the triviality of sufficiently many homotopy groups of intersections of sets in the cover (with a slightly weaker conclusion in exchange).

\begin{thm}[Nerve Theorem, \cite{bjorner2003nerves} Theorem 6 and Remark 7]\label{thm:NerveTheorem}
Let $X$ be a regular connected CW complex and $\{X_i\}_{i\in I}$ a family of subcomplexes with $X=\bigcup_{i\in I} X_i$. Let $\mc{N}$ be the nerve of $\{X_i\}$
\begin{enumerate}[(i)]
    \item Say that every finite nonempty intersection $X_{i_1}\cap\ldots\cap X_{i_t}$ is $(k-t+1)$-connected. Then there is a map $f:X\rightarrow \abs{\mc{N}}$ such that the induced homomorphism $f_{\#,j}:\pi_j(X)\rightarrow \pi_j(\abs{\mc{N}})$ is an isomorphism for all $j\leq k$ and an epimorphism for $j=k+1$.
    \item If every finite nonempty intersection $X_{i_1}\cap\ldots\cap X_{i_t}$ is contractible, then $f$ gives a homotopy equivalence $X\simeq \abs{\mc{N}}$.
\end{enumerate}
\end{thm}

A version of Bj\"{o}rner's nerve lemma sufficiently equivariant for our purposes reads thus:

\begin{thm}\label{thm:EquivariantNerveTheorem}
Let $X$ be a regular connected CW complex which is a subset of a real vector space, and whose cells are all convex polyhedra. Let $G$ act linearly on $X$ in such a way that $X$ is symmetric as a CW complex under the action of $G$. Let $\{X_i\}_{i\in I}$ be a family of subcomplexes with $X=\bigcup_{i\in I}X_i$, and say that for each $g\in G$ and $i\in I$, $g(X_i)\in \{X_i\}_{i\in I}$. Then if $\mc{N}$ is the nerve of $\{X_i\}_{i\in I}$, we have the following:
\begin{enumerate}[(i)]
    \item Say every finite nonempty intersection $X_{i_1}\cap\ldots\cap X_{i_t}$ is $(k-t+1)$-connected. Then there is an equivariant map $f:X\rightarrow\abs{\mc{N}}$ such that the induced (equivariant) homomorphism $f_{\#,j}:\pi_j(X,*)\rightarrow \pi_j(\abs{\mc{N}},f(*))$ is an isomorphism for all $j\leq k$ and an epimorphism for $j=k+1$.
    \item If every finite nonempty intersection $X_{i_1}\cap\ldots\cap X_{i_t}$ is contractible, then $f$ gives a homotopy equivalence $X\simeq \abs{\mc{N}}$.
\end{enumerate}
\end{thm}
\begin{proof}
(i): We claim that we may construct an equivariant map in the manner described in the proof in \cite{bjorner2003nerves}. By Proposition \ref{thm:SymmetricOrderComplex}, we have an equivariant homeomorphism $X\rightarrow \abs{\Delta(\fpos{X})}$. Let $\varphi: \fpos{X}\rightarrow \fpos{\mathcal{N}}$ be given by $\varphi(\sigma)=\{i\in I\mid \sigma\in X_i\}$ for each cell $\sigma$ of $X$. This is an order-reversing map of posets. It is also equivariant, since for a given cell $\sigma$, $\{g(i)\mid i\in I\text{ with }\sigma\in X_i\}=\{i\in I\mid g(\sigma)\in X_i\}$. Then $\varphi$ induces an equivariant continuous function $\abs{\Delta(\fpos{X})}\rightarrow\abs{\Delta(\fpos{\mathcal{N}})}$. Since Remark \ref{thm:SymmetricOrderComplex} gives an equivariant homeomorphism $\abs{\Delta(\fpos{\mathcal{N}})}\rightarrow \abs{\mathcal{N}}$, composing gives us an equivariant continuous function
\[
f: X\rightarrow \abs{\Delta(\fpos{X})} \rightarrow \abs{\Delta(\fpos{\mathcal{N}})} \rightarrow \abs{\mathcal{N}}
\]
The equivariance of $f$ gives the equivariance of each induced $f_{\#,j}:\pi_j(X)\rightarrow \pi_j(\abs{\mathcal{N}})$.

Since our equivariant map $f$ is the map from the original proof of the nerve theorem, part (ii) follows automatically.
\end{proof}

Note that though, if following the proof for (ii), one would obtain a map $g:\abs{\mc{N}}\rightarrow X$ which serves as a homotopy inverse to $f$, we have not guaranteed the equivariance of this or any map in the opposite direction. Thus there is more work to be done before we may term this a true equivariant version of Bj\"{o}rner's nerve theorem.

The barycentric retracting map described in Definition \ref{def:BarycentricRetraction} allows us to apply this version of the nerve theorem to coverings by sets open in our larger space. This version of the nerve theorem is used in \cite{gabrielov2009approximation}.

\begin{thm}\label{thm:NerveTheoremOpen}
Let $X$ be a regular connected CW complex. Let $\{Y_i\}_{i\in I}$ be a family of subsets of $X$. Assume each $Y_i$ may be written as a union of cells of $X$, and that each $Y_i$ is open in $X$. Let $Y=\bigcup_{i\in I} Y_i$, and let $\mc{N}_Y$ be the nerve of this family.
\begin{enumerate}
    \item Say that every finite nonempty intersection $Y_{i_1}\cap\ldots\cap Y_{i_t}$ is $(k-t+1)$-connected. Then there is a map $f:Y\rightarrow \abs{\mc{N}_Y}$ such that the induced homomorphism $f_{\#,j}:\pi_j(Y)\rightarrow \pi_j(\abs{\mc{N}_Y})$ is an isomorphism for all $j\leq k$ and an epimorphism for $j=k+1$.
    \item If every finite nonempty intersection $Y_{i_1}\cap\ldots\cap Y_{i_t}$ is contractible, then $f$ gives a homotopy equivalence $X\simeq \abs{\mc{N}_Y}$.
\end{enumerate}
\end{thm}
\begin{proof}
From $h_{\br}$ in \ref{def:BarycentricRetraction}, we obtain a map $f_1:Y\rightarrow \abs{\br(Y)}$ which induces a homotopy equivalence. Observe that $\abs{\br(Y)}$ is itself a regular CW complex, and that each $\abs{\br(Y_i)}$ is a subcomplex of $\abs{\br(Y)}$. It is clear that $\abs{\br(Y)}\supset\bigcup_{i\in I} \abs{\br(Y_i)}$. To show the opposite inclusion, say that $\Delta=\Delta(v_0,\ldots, v_d)\in \br(Y)$. This means that $\clos{\Delta}\in\widehat{Y}$, i.e. that the vertices of $\Delta$, suitably ordered, correspond to a flag $\sigma_0,\ldots, \sigma_d$ of cells of $X$, all of which are contained in $Y$. Assume $\sigma_d$ has lowest dimension amongst the cells in this chain, and say that $i\in I$ is such that $\sigma_d\subset Y_i$. Since $Y_i$ is open in $X$, this means that each of $\sigma_0,\ldots, \sigma_d$ are contained in $Y_i$, and therefore $\Delta\in\br(Y_i)$. Hence as desired, $\abs{\br(Y)}=\bigcup_{i\in I} \abs{\br(Y_i)}$.

For $Y_{i_1},\ldots,Y_{i_t}\in \{Y_i\}_{i\in I}$, we can see that
\begin{align*}
    \bigcap_{j=1}^t \br(Y_{i_j})&=\bigcap_{j=1}^t\{\sigma\in \widehat{X}\mid \clos{\sigma}\subset Y_{i_j}\}\\
    &=\{\sigma\in \widehat{X}\mid \clos{\sigma}\subset \bigcap_{j=1}^t Y_{i_j}\}\\
    &=\br(\cap_{j=1}^t Y_{i_j})
\end{align*}
If $\mc{N}_{\br(Y)}$ denotes the nerve of the covering of $\abs{\br(Y)}$ by the family $\{\abs{\br(Y_i)}\}_{i\in I}$, this means that $\mc{N}_Y=\mc{N}_{\br(Y)}$.
Furthermore, since via $h_{\br}$ we have that $\bigcap_{j=1}^t Y_{i_j}$ is homotopy equivalent to $\abs{\br(\bigcap_{j=1}^t Y_{i_j})}=\bigcap_{j=1}^t \abs{\br(Y_{i_j})}$, we may conclude that for any $k$. a finite intersection intersection $Y_{i_1}\cap\cdots\cap Y_{i_t}$ is $k$-connected iff $\abs{\br(Y_{i_1})}\cap \cdots\cap \abs{\br(Y_{i_t})}$ is.

Let $f_2:\br(Y)\rightarrow \abs{\mc{N}_{\br(Y)}}$ be as given by the standard Nerve Theorem (Theroem \ref{thm:NerveTheorem}). Then composing
\[
Y\overset{f_1}{\rightarrow}\abs{\br(Y)}\overset{f_2}{\rightarrow}\abs{\mc{N}_{\br(Y)}}=\abs{\mc{N}_Y}
\]
we obtain our map $f:Y\rightarrow \abs{\mc{N}_Y}$ with the desired properties.
\end{proof}

\begin{cor}\label{thm:EquivariantNerveTheoremOpen}
Let $X$, $Y$, and $\{Y_i\}_{i\in I}$ be as in Theorem \ref{thm:NerveTheoremOpen}. Assume further that $X$ is a subset of a real vector space and all cells of $X$ are convex polyhedra. Let the group $G$ act linearly on our space in such a way that $X$ is symmetric as a CW complex and $\{Y_i\}$ has the property that for each $g\in G$ and $i\in I$, $g(Y_i)\in\{Y_i\}_{i\in I}$. Let $\mc{N}_Y$ denote the nerve of the covering of $Y$ by $\{Y_i\}_{i\in I}$. Then the map $f:Y\rightarrow \abs{\mc{N}_Y}$ of Theorem \ref{thm:NerveTheoremOpen} is equivariant.
\end{cor}
\begin{proof}
This follows from Proposition \ref{thm:BREquivariance} and Theorem \ref{thm:EquivariantNerveTheorem}.
\end{proof}

We will want one more equivariant Nerve Theorem. The statement we need was proved by Hess and Hirsch in \cite{hess2013topology}. This version is stated for $X$ a simplicial complex covered by a family for which nonempty intersections are contractible, but from it we may obtain equivariant maps in both directions.

For $G$ a group and $\sigma$ a simplex in some simplicial complex $X$ which is symmetric relative to $G$, let $G_\sigma$ be the subgroup of $G$ given by $G_\sigma=\{g\in G\mid g(\sigma)=\sigma\}$.

\begin{thm}[\cite{hess2013topology} Lemma 2.5]\label{thm:Hess2.5}
Let $X$ be a simplicial complex symmetric relative to the action of $G$, and let $\{X_i\}_{i\in I}$ be a family of subcomplexes with $\abs{X}=\bigcup_{i\in I} \abs{X_i}$ and with $g(X_i)\in \{X_i\}_{i\in I}$ for each $g\in G$ and $i\in I$. Say that every nonempty finite intersection $\bigcap_{i\in\sigma} \abs{X_i}$, for $\sigma\subset I$, is $G_\sigma$-contractible. Then if $\mc{N}$ is the nerve of $\{X_i\}_{i\in I}$, we have that $\abs{X}$ and $\abs{\mc{N}}$ are $G$-homotopy equivalent.
\end{thm}

\section{Equivariance in the Gabrielov-Vorobjov Construction}\label{sect:ProofsOfEquivariance}

In what follows, we assume that $G$ is a finite reflection group acting on $\bb{R}^n$. Let $S$ be the definable set we intend to approximate. Assume we have also chosen a compact definable set $A\subset \bb{R}^n$ with $S\subset A$ and families $\{S_\delta\}_{\delta>0}$ and $\{S_{\delta, \epslon}\}_{\delta, \epslon>0}$ representing $S$ in $A$ as described in Section \ref{sect:GVConstruction}. The conclusions of Subsections \ref{sect:ConstructionV} and \ref{sect:Tau} (which concern relations between $S$ and an intermediate set $V$) apply both to the definable and constructible cases. The distinctions between these two cases are addressed in Subsections \ref{sect:MainTheoremDefinableCase} and \ref{sect:MainTheoremSeparableCase}, which describe relations between $V$ and our approximating set $T$.

\subsection{Symmetric construction for $V$}\label{sect:ConstructionV}

Gabrielov and Vorobjov' argument in \cite{gabrielov2009approximation} first uses a triangulation of $A$ adapted to $S$ to construct, for a given integer $m>0$ and sequence $0<\epslon_0,\delta_0,\ldots, \epslon_m,\delta_m<1$, an intermediate set $V=V(\epslon_0,\delta_0,\ldots,\epslon_m,\delta_m)$ and homomorphisms $\tau_{\#,k}: \pi_k(V)\rightarrow \pi_k(S)$ and $\tau_{*,k}:H_k(V)\rightarrow H_k(S)$. The set $V$ echoes the behavior of $T$ relative to the parameters $\epslon_i$ and $\delta_i$. However, it also has a cover that allows one to liken it to the triangulation of $S$ well enough to establish the aforementioned maps. We show that, if we start with a symmetric triangulation, the set $V$ is also symmetric and that there exists an equivariant map $\tau:V\rightarrow S$ which induces $\tau_{\#}$ and $\tau_{*}$.

Let $(\Lambda,\Phi)$ be a symmetric triangulation of $A$ adapted to $S$. We replace $S$ by $\Phi^{-1}(S)$ to assume that $S$ is a union of simplices of $\Lambda$. We will primarily work within the first barycentric subdivision of $\Lambda$, denoted $\widehat{\Lambda}$. If $\Delta(b_0,\ldots,b_p)$ corresponds to a simplex in $\widehat{\Lambda}$, we have that for each $b_i$ there is a unique simplex $\Delta_{b_i}$ in $\Lambda$ such that $b_i$ is the barycenter of $\Delta_{b_i}$. We will assume when we write $\Delta(b_0,\ldots,b_p)$ that $b_0,\ldots,b_p$ are ordered so that $\dim(\Delta_{b_0})>\cdots>\dim(\Delta_{b_p})$. By $\widehat{S}$ we mean the set of all simplices of $\widehat{\Lambda}$ which belong to $S$.

Gabrielov and Vorobjov in \cite{gabrielov2009approximation} construct $V$ as a union of sets $K_{B}(\delta_i,\epslon_i)$ defined for pairs of simplices $K\in \widehat{\Lambda}$ and $B\in \widehat{S}$. The set $K_B(\delta_i,\epslon_i)$ depends on something that Gabrielov and Vorobjov call the core of the simplex $B=B(b_0,\ldots, b_p)$. This in some sense is meant to capture which faces of $B$ continue to intersect the preimages of the sets $S_\delta$ as $\delta$ shrinks to $0$. To properly define this notion, though, we must reiterate some technical terminology from \cite{gabrielov2009approximation}.

\begin{definition}[\cite{gabrielov2009approximation} Definition 3.2]\label{def:Gabrielov3.2}
We say $S$ is \emph{marked} if for each pair of simplices $(\Delta',\Delta)$ in $S$ with $\Delta'$ a subsimplex of $\Delta$, we have designated $\Delta'$ as either a \emph{hard} or \emph{soft} subsimplex of $\Delta$. For a pair $(\Delta',\Delta)$ with $\Delta'$ not in $S$, we always designate $\Delta'$ as a soft subsimplex of $\Delta$.
\end{definition}

\begin{definition}
If $A$ and $S$ are symmetric and $(\Lambda,\Phi)$ is an equivariant triangulation of $A$ adapted to $S$, we say $S$ is \emph{symmetrically marked} if $S$ is marked in such a way that for each $g\in G$, $(g(\Delta'),g(\Delta))$ has the same hard/soft designation as $(\Delta',\Delta)$. Since $\Delta'$ is in $S$ iff $g(\Delta')$ is also in $S$ by our choice of triangulation, this stipulation does not interfere with the property that simplices $\Delta'$ not contained in $S$ are always designated as soft subsimplices.
\end{definition}

We will specify separate hard-soft relations for our triangulation depending on whether we are in the separable case or not (see Subsections \ref{sect:MainTheoremDefinableCase} and \ref{sect:MainTheoremSeparableCase}). For the moment, it is enough to assume that $S$ is symmetrically marked.

\begin{definition}[\cite{gabrielov2009approximation} Definition 3.3]\label{def:Gabrielov3.3}
For a simplex $B=B(b_0,\ldots,b_p)$ of $\widehat{\Lambda}$ contained in $\widehat{S}$, the \emph{core} of $B$, denoted $C(B)$, is the maximal subset $\{b_0,\ldots,b_{p'}\}$ of $\{b_0,\ldots,b_p\}$ so that for $0\leq \nu\leq p'$, we have $\Delta_{b_\nu}$ is a hard subsimplex of $\Delta_{b_\mu}$ for every $0\leq\mu<\nu$.
\end{definition}

Note that we always have $b_0\in C(B)$. We will establish the convention that if $B$ is not in $\widehat{S}$, $C(B)=\emptyset$.

\begin{definition}[\cite{gabrielov2009approximation} Definition 3.6]\label{def:Gabrielov3.6}
Let $B=B(b_0,\ldots,b_p)$ be a simplex in $\widehat{S}$ and $K=K(c_0,\ldots,c_q)$ a simplex in $\widehat{\Lambda}$ with $B\subset \clos{K}$. Let $I=\{b_0,\ldots,b_p\}$ and $J=\{c_0,\ldots,c_q\}$. Then for $0<\delta<1$ and $0<\epslon<1$, define
\begin{multline*}
    K_B(\delta,\epslon):=\left\{\sum_{c_\nu\in J}t_{c_\nu}c_\nu\in K(c_1,\ldots,c_q)\mid \sum_{b_\nu\in C(B)}t_{b_\nu} >\delta,\right.\\
    \left.\sum_{b_\nu\in I}t_{b_\nu}>1-\epslon, \text{ and }\forall b_\nu\in I\;\forall c_\mu\in (J\setminus I), t_{b_\nu}>t_{c_\mu}\right\}
\end{multline*}
\end{definition}

Given $B$ a simplex in $\widehat{S}$, let $S_B$ denote the set of all simplices $B'$ of $\widehat{\Lambda}$ with $B'\subset \clos{B}\cap \widehat{S}$.

\begin{definition}[\cite{gabrielov2009approximation} Definition 3.7]\label{def:Gabrielov3.7}
Fix $m>0$ and a sequence $0<\epslon_0,\delta_0,\epslon_1,\delta_1,\ldots, \epslon_m,\delta_m<1$. Then for $B$ a simplex in $\widehat{S}$, let
\[
V_B=\bigcup_{B'\in S_B}\;\bigcup_{\clos{K} \supset B'}
\;\bigcup_{i=0}^m K_{B'}(\delta_i,\epslon_i)
\]
(that is, for each $B'\in S_B$ the union is taken over all simplices $K$ in $\widehat{\Lambda}$ with $B'\subset\clos{K}$). Let
\[
V=\bigcup_{B\in \widehat{S}} V_B
\]
\end{definition}

\begin{prop}\label{thm:SymmetryV}
Say that $S$ and $A$ are symmetric under the action of $G$, $(\Lambda,\Phi)$ is an equivariant triangulation of $A$ adapted to $S$, and $S$ is symmetrically marked. Then the family $\{V_B\}_{B\in \widehat{S}}$ is symmetric under the induced action of $G$, and hence the set $V$ is symmetric.
\end{prop}
\begin{proof}
This follows more or less immediately from construction. By symmetry of $\widehat{\Lambda}, \widehat{S}$, and our marking and by linearity of the action of $G$, we can see that $g(K_{B'}(\delta_i,\epslon_i))=g(K)_{g(B')}(\delta_i,\epslon_i)$, and hence $g(V_B)=V_{g(B)}$.
\end{proof}

Some properties of the sets $K_B(\delta,\epslon)$, $V_B$ and $V$ are worth noting here.

\begin{lemma}\label{thm:UnionK_B}
Let $B=B(b_0,\ldots, b_p)$ be a simplex in $\widehat{S}$ and $K=K(c_0,\ldots, c_q)$ a simplex in $\widehat{\Lambda}$ with $B\subset \clos{K}$. Then for $0<\delta,\epslon,\delta',\epslon'<1$ we have
\begin{align*}
    K_B(\delta,\epslon)\cup K_B(\delta',\epslon')&=K_B(\min\{\delta,\delta'\}, \max\{\epslon, \epslon'\})\\
    K_B(\delta,\epslon)\cap K_B(\delta',\epslon')&=K_B(\max\{\delta,\delta'\}, \min\{\epslon, \epslon'\})
\end{align*}
\end{lemma}
\begin{proof}
This follows immediately from the definition. The second line appears as \cite{gabrielov2009approximation} Lemma 4.1.
\end{proof}

\begin{prop}[see the proof of \cite{gabrielov2009approximation} Theorem 4.8]\label{thm:IntersectionVB}
For any pair of simplices $B_1$ and $B_2$ in $\widehat{S}$, one of the following holds
\begin{enumerate}[(i)]
    \item $V_{B_1}\cap V_{B_2}=\emptyset$ ($\Leftrightarrow \clos{B}_1\cap \clos{B}_2\cap \abs{\widehat{S}}= \emptyset$)
    \item $V_{B_1}\cap V_{B_2}=V_{B_0}$ where $V_{B_0}$ is the unique simplex in $\widehat{S}$ with $\clos{B}_1\cap\clos{B}_2\cap\abs{\widehat{S}}=\clos{B}_0\cap \abs{\widehat{S}}$
\end{enumerate}
\end{prop}
\begin{proof}
The bulk of this statement is taken from \cite{gabrielov2009approximation}. Because we assert something slightly more detailed, we include a proof.

From the definition, we have that for $B, B'\in \widehat{S}$ and $K,K'\in \widehat{\Lambda}$ with $B\subset \clos{K}$ and $B'\subset\clos{K}'$, and any $0< \epslon,\delta<1$, then $K_B(\delta,\epslon)\cap K'_{B'}(\delta,\epslon)\neq \emptyset$ implies that $K=K'$ and either $B\subset\clos{B}'$ or $B'\subset\clos{B}$ (this is \cite{gabrielov2009approximation} Lemma 4.1). This means that
\[
V_{B_1}\cap V_{B_2}=\bigcup_{B'\in S_{B_1}\cap S_{B_2}}\;\bigcup_{\clos{K}\supset B'}\;\bigcup_{i=0}^m K_{B'}(\delta_i,\epslon_i)
\]
Since $S_{B_1}\cap S_{B_2}=\emptyset$ $\Leftrightarrow$ $\clos{B}_1\cap \clos{B}_2\cap \widehat{S}=\emptyset$ and otherwise $S_{B_1}\cap S_{B_2}=S_{B_0}$ where $B_0$ is the simplex such that $\clos{B}_0\cap \widehat{S}=\clos{B}_1\cap \clos{B}_2\cap \widehat{S}$, the statements follow.
\end{proof}

\begin{lemma}[\cite{gabrielov2009approximation} Lemma 4.5]\label{thm:Gabrielov4.5}
For $B$ in $\widehat{S}$, $m\geq 1$, and $0<\epslon_0\ll \delta_0\ll\cdots\ll \epslon_m\ll\delta_m\ll 1$, the set $V_B$ is open in $\abs{\Lambda}$ and $(m-1)$-connected.
\end{lemma}

Note that Gabrielov and Vorobjov give $V_B$ as
\[
V_B=\bigcup_{B'\in S_B}\;\bigcup_{\clos{K} \supset B}
\;\bigcup_{i=0}^m K_{B'}(\delta_i,\epslon_i)
\]
where the union is only taken over those $K\in \widehat{\Lambda}$ with $B\subset\clos{K}$. However, we need to define $V_B$ as it appears here in Definition \ref{def:Gabrielov3.7} in order for Proposition \ref{thm:IntersectionVB} to be stated and used as it is in \cite{gabrielov2009approximation}. Lemma \ref{thm:Gabrielov4.5} holds with the updated definition; in the proof of \cite{gabrielov2009approximation} Lemma 4.5, one needs only to define the sets $U_{B',i}$ for $B'\in S_B$ as
\[
U_{B',i}=\bigcup_{\clos{K}\supset B'} K_{B'}(\delta_i,\epslon_i)
\]
rather than
\[
U_{B',i}=\bigcup_{\clos{K}\supset B}K_{B'}(\delta_i.\epslon_i)
\]
(i.e., one must take the union over all simplices $K$ of $\widehat{\Lambda}$ with $B'\subset \clos{K}$, and not just those with $B\subset\clos{K}$). Then the updated sets $U_{B',i}$ for $B'\in S_B$ and $1\leq i\leq m$ cover the updated $V_B$, but the intersection condition and hence the nerve of this family remains unchanged, and so the argument in \cite{gabrielov2009approximation} continues to hold.

\subsection{Equivariance of the map $\tau$}\label{sect:Tau}

The existence of homomorphisms $\tau_{\#,k}: \pi_k(V)\rightarrow \pi_k(S)$ and $\tau_{*,k}: H_k(V)\rightarrow H_k(S)$ for $0\leq k\leq m$ is given in \cite{gabrielov2009approximation} Theorem 4.8. In order to show that we may construct these functions equivariantly, we must dissect the proofs there and describe more explicitly some of the details involved in applying the nerve theorem to $V$. We will in fact obtain an equivariant map $\tau:V\rightarrow S$ on the level of sets, inducing equivariant maps of (pointed) homotopy and homology groups which are isomorphisms or epimorphisms for the promised indices.

Our first goal is to endow $\abs{\Lambda}$ with a regular CW complex structure fine enough to allow us to write $V$ as a union of cells.

\begin{notation}\label{def:DecompsK}
Let $K$ be a simplex of some simplicial complex $\Lambda$, and say $K$ has vertex set $\{v_0,\ldots,v_d\}$.
\begin{itemize}
    \item Let $\mc{C}_{K,\text{vert}}=\{L\in \widehat{\Lambda}\mid L\subset K\}$.
    \item For $I\subset \{v_0,\ldots, v_d\}$ and a given $0<\epslon<1$, let
    \begin{itemize}
        \item[$\sim$] $C_{K,I,>,\epslon}=\{\sum_{i=1}^d t_iv_i\in K\mid \sum_{i\in I} t_i >\epslon\}$
        \item[$\sim$] $C_{K,I,<,\epslon}=\{\sum_{i=1}^d t_iv_i\in K\mid \sum_{i\in I} t_i <\epslon\}$
        \item[$\sim$] $C_{K,I,=,\epslon}=\{\sum_{i=1}^d t_iv_i\in K\mid \sum_{i\in I} t_i =\epslon\}$
    \end{itemize}
    We denote $\mc{C}_{K,I,\epslon}=\{C_{K,I,>,\epslon}, C_{K,I,<,\epslon},  C_{K,I,=,\epslon}\}$.
\end{itemize}
\end{notation}

\begin{lemma}\label{thm:CWDecompK}
Fix a simplex $K$ of some simplicial complex $\Lambda$, let $I$ be a subset of the vertex set of $K$, and let $0<\epslon<1$. For a given subsimplex $K'$ of $K$, let $I_{K'}$ be the intersection of $I$ with the vertex set of $K'$. Then
\[
\bigcup_{K'\subset\clos{K}}\mc{C}_{K', I_{K'}, \epslon}
\]
gives a regular CW decomposition of $\clos{K}$ into convex polyhedra, where the union is taken over all simplices $K'\subset\clos{K}$.
\end{lemma}
\begin{proof}
Each cell of our collection is a convex polyhedron more or less by definition, and hence a regular cell in a fairly immediate manner. We have that for each simplex $K'\subset\clos{K}$, $\mc{C}_{K',I_{K'},\epslon}$ gives a partition of $K'$, and so since the simplices $K'$ partition $\clos{K}$, our collection indeed forms a partition of $\clos{K}$.

It remains to show that for a cell $C$ in our collection, $\partial C$ is contained in a union of cells of lower dimension. Say $C=C_{K',I_{K'},=,\epslon}$. Then the boundary of $C$ is the union of those cells $C_{K'',I_{K''},=,\epslon}$, $K''\subsetneq \clos{K'}$, which are nonempty (so, those for which the vertex set of $K''$ is neither contained in nor disjoint from $I$). If $C=C_{K', I_{K'}, >,\epslon}$, then $\partial C$ consists both of $C_{K',I_{K'},=,\epslon}$ together with the cells in its boundary and those nonempty cells $C_{K'',I_{K''},>,\epslon}$ corresponding to simplices $K''\subsetneq K'$ (i.e. where the vertex set of $K''$ is not disjoint from I). The case of $C=C_{K', I_{K'}, <,\epslon}$ is identical save that in the condition for nonemptiness we must replace $I$ with its complement. Hence this collection indeed gives a regular CW decomposition of $\clos{K}$.
\end{proof}

We return to our primary setting, in which $\Lambda$ is a symmetric triangulation of $A$ adapted to $S$. For $K=K(c_1,\ldots, c_q)$ a simplex of $\widehat{\Lambda}$ and $0< \delta, \epslon< 1$, we let
\begin{multline*}\label{eq:DecompC}
    \mc{C}_{K,\delta,\epslon}=\left\{C=L\cap \bigcap_{I\subset\{c_0,\ldots,c_q\}}C_{K,I,\delta} \cap \bigcap_{I\subset \{c_0,\ldots,c_q\}}C_{K,I,\epslon} \mid \right.\\
    L\in \mc{C}_{K, \text{vert}}, C_{K,I,\delta}\in \mc{C}_{K, I, \delta} \text{ and } C_{K,I,\epslon}\in \mc{C}_{K, I, \epslon}\\
    \left.\text{ for each } I, \text{ and } C\neq\emptyset\right\}
\end{multline*}
and let $\mc{C}_{\delta,\epslon}=\bigcup_{K\in \widehat{\Lambda}} \mc{C}_{K,\epslon,\delta}$.

\begin{figure}
    \centering
    \includegraphics{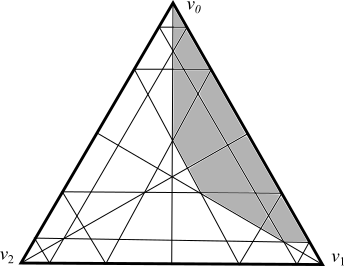}
    \caption{Decomposition $\mc{C}_{K,\delta,\epslon}$ for $K(v_0,v_1,v_2)$, assuming $\delta<\epslon$. $K_B(\delta,\epslon)$ is shaded for $B=B(v_0,v_1)$ and $C(B)=\{v_0\}$}
    \label{fig:DecompositionK}
\end{figure}

\begin{lemma}\label{thm:CWDecompA}
For any $0<\delta,\epslon<1$, $\mathcal{C}_{\delta,\epslon}$ gives a regular CW decomposition of $\abs{\widehat{\Lambda}}$ which is symmetric relative to $G$ and in which each cell is a convex polyhedron.
\end{lemma}
\begin{proof}
We have that for each $K\in \widehat{\Lambda}$, the collection of cells of the second barycentric subdivision of $\Lambda$ contained in $\clos{K}$ gives a regular CW decomposition of $K$ into convex polyhedra. Lemma \ref{thm:CWDecompK} shows that for each subset $I$ of the vertex set of $K$, $\bigcup_{K'\subset\clos{K}}\mathcal{C}_{K', I_{K'}, \epslon}$ and $\bigcup_{K'\subset\clos{K}}\mathcal{C}_{K', I_{K'}, \delta}$ each gives a regular CW decomposition of $\clos{K}$. Since each set in $\mathcal{C}_{\delta,\epslon}$ is an intersection of one set from each of these decompositions across the various simplices $K$ of $\widehat{\Lambda}$, the cells of $\mathcal{C}_{\delta,\epslon}$ remain convex polyhedra. Also for each $K$, $\mathcal{C}_{K.\delta,\epslon}$ remains a partition of $K$, and so $\mathcal{C}_{\delta,\epslon}$ gives a partition of $\abs{\widehat{\Lambda}}$.

Let $C\subset K$ for some $C\in\mathcal{C}_{\delta,\epslon}$. We simplify the notation above to write $C=C_1\cap\cdots\cap C_l$ for $C_i$ cells of our various decompositions. Then we have
\begin{align*}
    \partial C &\subset \bigcup_{i=1}^l \clos{C}_1\cap \ldots \cap\clos{C}_{i-1}\cap\partial C_i\cap \clos{C}_{i+1}\cap \ldots\cap \clos{C}_l\\
    &=\bigcup_{J\subsetneq\{1,\ldots,l\}} \paren{\bigcap_{i\in J} C_i \cap \bigcap_{i\in \{1,\ldots, l\}\setminus J}\partial C_i }
\end{align*}
after decomposing each $\clos{C}_i=\partial C_i \cup C_i$ and distributing. We know $\delta C_i$ is a union of cells of dimension lower than that of $C_i$ from the decomposition of $\clos{K}$ corresponding to the decomposition of $K$ from which $C_i$ comes. Replacing each instance of $\partial C_i$ with this decomposition for each $1\leq i\leq l$ in the expression above and again distributing, rewriting, and discarding empty intersections, we obtain that $\partial C$ is contained in a union of cells of $\mathcal{C}_{\delta,\epslon}$ all of which have dimension lower than that of $C$. Hence $\mathcal{C}_{\delta,\epslon}$ gives the desired CW decomposition of $\abs{\Lambda}$.

To show symmetry, let $C\in \mc{C}_{\delta,\epslon}$ with $C\subset K$ for a simplex $K$ in $\widehat{\Lambda}$. Then $C$ is the set of points of the form $\sum t_{c_i} c_i$, with the $c_i$ being the vertices of $K$, such that the coefficients $t_{c_i}$ satisfy certain conditions. The linearity of the action of $G$ implies that $g(C)$ then consists of points $\sum t_{c_i} g(c_i)$ such that the coefficients $t_{c_i}$ satisfy those same conditions. From our definition of $\mc{C}_{\delta,\epslon}$, this means $C$ is a cell contained in $\mc{C}_{g(K),\delta,\epslon}\subset \mc{C}_{\delta,\epslon}$, as desired.
\end{proof}

Say $V=V(\epslon_0, \delta_0,\ldots, \epslon_m,\delta_m)$, and let $\delta=\min\{\delta_0,\ldots,\delta_m\}$ and $\epslon=\max\{\epslon_0,\ldots,\epslon_m\}$.

\begin{lemma}\label{thm:DecompV}
For each $B$ a simplex of $\widehat{S}$, $V_B$ can be written as a union of cells of $\mathcal{C}_{\delta,\epslon}$, and hence so can $V$. Furthermore, the set $\{C\in \mathcal{C}_{\delta,\epslon}\mid C\subset V\}$ is symmetric under the action of $G$.
\end{lemma}
\begin{proof}
First, observe that if a simplex $K$ has vertex set $\{c_0,\ldots,c_q\}$, then the simplices of the barycentric subdivision of $\clos{K}$ belonging to $K$ correspond to the subsets of $\clos{K}$ given by
\begin{multline*}
0<t_{c_{i_1}}=\cdots=t_{c_{i_{\lambda_1}}}<t_{c_{i_{\lambda_1+1}}}=\cdots=t_{c_{i_{\lambda_2}}}<\\
\ldots<t_{c_{i_{\lambda_l}}}=\cdots=t_{c_{i_q}}<1
\end{multline*}
for various $1\leq\lambda_1\leq \cdots\leq \lambda_l\leq q$ and permutations $i_1,\ldots, i_q$ of $1,\ldots, q$. Using this observation to translate the condition ``$t_{b_\nu}>t_{c_\nu}$ for all $b_\nu\in I$ and $c_\mu\in (J\setminus I)$" from the definition of $K_B(\delta, \epslon)$, we see we have defined $\mc{C}_{\delta,\epslon}$ in such a way that for each pair $K$ and $B$ with $K$ a simplex of $\widehat{\Lambda}$, $B$ a simplex of $\widehat{S}$, and $B\subset\clos{K}$, each cell of $\mc{C}_{\delta,\epslon}$ is either contained in or disjoint from the set $K_B(\delta, \epslon)$ (compare Notation \ref{def:DecompsK} and Definition \ref{def:Gabrielov3.6}). However, by Lemma \ref{thm:UnionK_B}, we have that $K_B(\delta, \epslon)=K_B(\delta_0, \epslon_0)\cup\cdots \cup K_B(\delta_m, \epslon_m)$. Now, for $B$ any simplex contained in $\widehat{S}$, we have that $V_B$ is the union of sets $K_{B'}(\delta, \epslon)$ for $K$ in $\widehat{\Lambda}$ with $B\subset \clos{K}$ and $B'\in S_B$, and hence is a union of cells of $\mc{C}_{\delta,\epslon}$.

Let $C\in \mc{C}_{\delta,\epslon}$ with $C\subset K$ for $K$ a simplex in $\widehat{\Lambda}$. If $C\subset K_B(\delta, \epslon)$ for some appropriate $K$ and $B$, then $g(C)\subset g(K)_{g(B)}(\delta,\epslon)$. Hence the set of cells of $\mc{C}_{\delta,\epslon}$ contained in $V$ is symmetric under the action of $G$.
\end{proof}

Now, we may begin assembling our equivariant map $\tau: V\rightarrow S$.

\begin{lemma}\label{thm:VandNV}
Let $\mc{N}_V$ denote the nerve of the covering of $V$ by the family $\{V_B\}_{B\in \widehat{S}}$ (see Definition \ref{def:Gabrielov5.7}). Then for $m\geq 1$ and $0<\epslon_0\ll\delta_0\ll\cdots\ll\epslon_m\ll\delta_m\ll 1$, there exists an equivariant map $\psi_V:V\rightarrow \mc{N}_V$ such that the induced homomorphism $(\psi_V)_{\#,k}:\pi_{k}(V,*)\rightarrow \pi_{k}(\abs{\mc{N}_V},\psi_V(*))$ is an isomorphism for $k\leq m-1$ and an epimorphism for $k=m$.
\end{lemma}
\begin{proof}
By Lemma \ref{thm:CWDecompA}, $\abs{\widehat{\Lambda}}$ together with the collection $\mc{C}$ is a symmetric regular CW complex whose cells are all convex polyhedra. Each set in $\{V_B\}_{B\in \widehat{S}}$ is open in $\abs{\Lambda}$ (Lemma \ref{thm:Gabrielov4.5}) and is a union of cells of $\mc{C}$ (Lemma \ref{thm:DecompV}), and the family $\{V_B\}_{B\in\widehat{S}}$ has the property that $g(V_B)\in \{V_B\}$ for each $g\in G$ and $B\in \widehat{S}$ (Proposition \ref{thm:DecompV}). Furthermore, for each finite nonempty intersection we have $V_{B_1}\cap \ldots \cap V_{B_t} = V_{B_0}$ for some $B_0\in \widehat{S}$ (Proposition \ref{thm:IntersectionVB}), and so this is $(m-1)$-connected (Lemma \ref{thm:Gabrielov4.5}). The map $\psi_V:V\rightarrow \abs{\mc{N}_V}$ given as in Corollary \ref{thm:EquivariantNerveTheoremOpen} is then equivariant. We also see that, when restricted to each connected component of $V$, $\psi_V$ induces a homomorphism of homotopy groups which is an isomorphism for $k\leq m-1$ and an epimorphism for $k=m$. Hence $(\psi_V)_{\#,k}:\pi_{k}(V,*)\rightarrow \pi_{k}(\abs{\mc{N}_V},\psi_V(*))$ is also an isomorphism for $k\leq m-1$ and an epimorphism for $k=m$.
\end{proof}

Let $\br(\widehat{S})$ be the set of simplices in the barycentric subdivision of $\widehat{\Lambda}$ which are contained with their closure in $\abs{\widehat{S}}$ (see Definition \ref{def:BarycentricRetraction}). For a simplex $B\in \widehat{S}$, let $\widetilde{B}=\clos{B}\cap \abs{\widehat{S}}$.

\begin{lemma}\label{thm:NVandNS}
Let $\mc{N}_V$ be the nerve of the covering of $V$ by the family $\{V_B\}_{B\in \widehat{S}}$, and let $\mc{N}_{
S}$ denote the nerve of the covering of $\abs{\br(\widehat{S})}$ by the family $\{\br(\widetilde{B})\}_{B\in\widehat{S}}$ Then there exists an equivariant homeomorphism $\xi: \abs{\mc{N}_V}\rightarrow \abs{\mc{N}_{\br(\widehat{S})}}$.
\end{lemma}
\begin{proof}
The homeomorphism $\xi: \abs{N_V}\rightarrow \abs{N_{\widehat{S}}}$ is described in the proof of \cite{gabrielov2009approximation} Theorem 4.8. Both
families $\{V_B\}$ and $\{\br(\widetilde{B})\}$ are indexed over the simplices of $\widehat{S}$. Furthermore, by Proposition \ref{thm:IntersectionVB} we see that $\sigma\subset \widehat{S}$ is a simplex of $\mc{N}_V$ iff $\sigma$ is a simplex of $\mc{N}_{\br(\widehat{S})}$. Thus, since $\xi$ is induced from the identity map on the vertex sets of $\mc{N}_V$ and $\mc{N}_{\br(\widehat{S})}$, $\xi$ is equivariant.
\end{proof}

\begin{lemma}\label{thm:NSandS}
Let $\mc{N}_{\br(\widehat{S})}$ denote the nerve of the covering of $\br(\widehat{S})$ by the family $\{\br(\widetilde{B})\}_{B\in\widehat{S}}$. Then there exists an equivariant map $\psi_{\widehat{S}}: \abs{\mc{N}_{\br(\widehat{S})}}\rightarrow S$ which induces a homotopy equivalence.
\end{lemma}
\begin{proof}
Observe that $\br(\widehat{S})$ is a full simplicial complex, each $\br(\widetilde{B})$ is a subcomplex of $\br(\widehat{S})$, and that the family $\{\br(\widetilde{B})\}_{B\in\widehat{S}}$ is invariant under the action of $G$. Furthermore, any nonempty intersection $\br(\widetilde{B}_1)\cap\ldots\cap \br(\widetilde{B}_l)$ is equal to $\br(\widetilde{B}_0)$ for some $B_0\in \widehat{S}$.

Let $v_{B_0}$ be the vertex of the second barycentric subdivision of $\Lambda$ corresponding to the simplex $B_0\in \widehat{S}$. For $\Delta$ a simplex of the second barycentric subdivision of $\Lambda$ having $v_{B_0}$ as one of its vertices, we may define a map $h_{\restriction\clos{\Delta}}:[0,1]\times\clos{\Delta}\rightarrow \clos{\Delta}$ by
\[
h_{\restriction\clos{\Delta}}(t,x)=t\sum_{v_i\neq v_{B_0}} t_{v_i}v_i+\paren{1-t\sum_{v_i\neq v_{B_0}} t_{v_i}}v_{B_0}
\]
where $x=\sum t_{v_i}v_i\in \clos{\Delta}$. Since we have agreement on the boundaries, and since $\br(\widetilde{B}_0)$ is the union of all such $\clos{\Delta}$ which are contained entirely in $\abs{\widehat{S}}$, we may extend to obtain a map $h:[0,1]\times \abs{\br(\widetilde{B}_0)}\rightarrow \abs{\br(\widetilde{B}_0)}$. The map $h$ has the property that $h(1,-)=\id$ and $h(0,-)=v_{B_0}$, and by linearity together with the fact that for such $g$, $g(B_0)=B_0$, we see that $g(h(t,x))=h(t,g(x))$ for any $g$ with $\{g(\br(\widetilde{B}_1)),\ldots,g(\br(\widetilde{B}_l))\}=\{\br(\widetilde{B}_1),\ldots,\br(\widetilde{B}_l)\}$.

Thus by Theorem \ref{thm:Hess2.5}, we obtain an equivariant map $\abs{\mc{N}_{\br(\widehat{S})}}\rightarrow \abs{\br(\widehat{S})}$ which induces a homotopy equivalence. However, as described in Definition \ref{def:BarycentricRetraction}, the embedding $\abs{\br(\widehat{S})}\hookrightarrow \abs{\widehat{S}}=\Phi^{-1}(S)$ induces a homotopy equivalence (and is equivariant). Composing gives us the desired equivariant map $\varphi_{\br(\widehat{S})}:\abs{\mc{N}_{\br(\widehat{S})}}\rightarrow S$.
\end{proof}

Now, to establish the equivariance of $\\tau_{\#,k}$ and $\tau_{*,k}$ we need only compose these three maps.

\begin{thm}\label{thm:EquivariantGabrielov4.8}
(ref \cite{gabrielov2009approximation} Theorem 4.8) For $m>0$, $0<\epslon_0\ll\delta_0\ll\cdots\ll \epslon_m\ll\delta_m\ll 1$ and $V=V(\epslon_0,\delta_0,\ldots, \epslon_m, \delta_m)$, there is an equivariant map $\tau:V\rightarrow S$ inducing equivariant homomorphisms $\tau_{\#,k}:\pi_k(V,*'')\rightarrow \pi_k(S,*')$ and $\tau_{*,k}:H_k(V)\rightarrow H_k(S)$ with $\tau_{\#,k}$, $\tau_{*,k}$ isomorphisms for every $k\leq m-1$ and $\tau_{\#,m}$, $\tau_{*,m}$ epimorphisms. Moreover if $m\geq \dim(S)$, then $\tau$ induces a homotopy equivalence $V\simeq S$.
\end{thm}
\begin{proof}
Lemmas \ref{thm:VandNV}, \ref{thm:NVandNS}, and \ref{thm:NSandS} demonstrate the equivariance of the maps
\[
V\overset{\psi_V}{\rightarrow} \abs{\mc{N}_V}\overset{\xi}{\rightarrow}\abs{\mc{N}_{\br(\widehat{S})}}\overset{\psi_{\widehat{S}}}{\rightarrow} S
\]
which are used in \cite{gabrielov2009approximation} Theorem 4.8 to construct the homomorphisms $\tau_k$. Accordingly, let $\tau=\psi_{\br(\widehat{S})}\circ\xi\circ\psi_V:V\rightarrow S$, and the induced (equivariant) homomorphisms of homotopy groups
\[
\tau_{\#,k}=(\psi_{\br(\widehat{S})})_{\#,k}\circ \xi_{\#}\circ (\psi_V)_{\#,k}:\pi_k(V,*'')\rightarrow \pi_k(S,*')
\]
Since $\psi_{\br(\widehat{S})}$ induces a homotopy equivalence and $\xi$ is a homeomorphism, the Whitehead Theorem on weak homotopy equivalence (cited as Theorem \ref{thm:Spanier7.6.24} here) states that on the level of homotopy groups both of these maps induce isomorphisms. Hence $\tau_{\#,k}$, like $(\psi_V)_{\#,k}$, is a isomorphism for $k\leq m-1$ and an epimorphism for $k=m$.

Gabrielov and Vorobjov in \cite{gabrielov2009approximation} next apply the Whitehead Theorem on homotopy and homology (cited as Theorem \ref{thm:Spanier7.5.9} here). The equivariant map $\tau:V\rightarrow S$ induces an equivariant homomorphism
\[
\tau_{*,k}=*\psi_{\br(\widehat{S})})_{*,k}\circ \xi_*\circ (\psi_V)_{*,k}: H_k(V)\rightarrow H_k(S)
\]
On each connected component of $V$ and $S$, $\tau_{*,k}$ is an isomorphism of homology groups for $1\leq k\leq m-1$ and an epimorphism when $k=m$. Hence $\tau_{*,k}$ itself is an isomorphism for $k\leq m-1$ and an epimorphism when $k=m$.

As observed in the proof of Theorem 4.8 in \cite{gabrielov2009approximation}, when $m\geq \dim(S)$ we may apply part (ii) of Theorem \ref{thm:EquivariantNerveTheoremOpen} to see that the map $\psi_V$ induces a homotopy equivalence. Thus, $\tau:V\rightarrow S$ also induces a homotopy equivalence.
\end{proof}

We have shown that there exists an equivariant map $V\rightarrow S$ which induces a homotopy equivalence between the two spaces. Unfortunately, this does not itself guarantee that there is an \emph{equivariant} homotopy inverse $S\rightarrow V$. Theorem \cite{hess2013topology} together with the contraction discussed in Definition \ref{def:BarycentricRetraction} give us an equivariant homotopy inverse of $\psi_{\br(\widehat{S})}$, and $\xi$ is a homeomorphism (and so $\xi^{-1}$ must also be equivariant). All that lacks is a fully equivariant version of Bj\"{o}rner's nerve theorem. It seems plausible that one could prove this using the methods of Hess and Hirsch in \cite{hess2013topology}, but for the purposes of this paper we do not need an equivariant map $S\rightarrow V$.

For the remainder of this section, we follow \cite{gabrielov2009approximation} and choose our triangulation in such a way as to account for our families of sets representing $A$. Consider the projection $\rho: A\times[0,1]\rightarrow [0,1]$. Since $A\times [0,1]$ is closed, bounded, definable, and symmetric under the induced action given by $g(\bs{x},t)=(g(\bs{x}),t)$, the projection $\rho$ is continuous, definable, and symmetric relative to this action of $G$ on $A\times [0,1]$. Since each $S_\delta$ is assumed to be a symmetric set, the set $S'=\bigcup_{\delta\in (0,1)} (S_\delta, \delta)\subset A\times [0,1]$ is also symmetric relative to our action on $A\times [0,1]$. So, via Theorem \ref{thm:SymmetricTriangulationFn}, let $(\Lambda', \Phi')$ be an equivariant triangulation of $\rho$ which is compatible with $S'$. Then we take $(\Lambda,\Phi)$ to be the triangulation induced by $\Lambda'$ on $\rho^{-1}(0)$. This is a symmetric triangulation of $A$ adapted to $S$, and so all conclusions stated thus far in Section \ref{sect:ProofsOfEquivariance} hold for this particular triangulation.

\subsection{The Main Theorem: Definable Case}\label{sect:MainTheoremDefinableCase}

Recall that in Subsection \ref{sect:ConstructionV}, we assumed that $S$ was symmetrically marked without specifying the relation. In the general definable case, we take $\Delta_1$ to be soft in $\Delta_2$ for each pair $(\Delta_1,\Delta_2)$ with $\Delta_1$ a (proper) subsimplex of $\Delta_2$. This means that for $B=B(b_0,\ldots, b_p)$ a simplex in $\widehat{S}$, we always have $C(B)=\{b_0\}$. This is trivially a symmetric marking.

In the definable case, while the proofs in \cite{gabrielov2009approximation} also utilize sets of the form $V(\epslon_0,\delta_0,\ldots,\epslon_m,\delta_m)$, they ultimately relate $S$ to a similarly defined set $V''=V''(\epslon'')$. We now outline the construction of $V''$.

\begin{definition}[\cite{gabrielov2009approximation} Definition 5.2]\label{def:Gabrielov5.2}
Say that $B=B(b_0,\ldots, b_p)$ is a simplex in $\widehat{S}$ and $K=K(c_0,\ldots, c_q)$ is a simplex in $\widehat{\Lambda}$ with $B\subset \clos{K}$. Let $I=\{b_0,\ldots, b_p\}$ and $J=\{c_0,\ldots, c_q\}$, and let $0<\epslon<1$. We define
\begin{multline*}
    K_B(\epslon):=\left\{\sum_{c_\nu\in J}t_{c_\nu}c_\nu\in K(c_0,\ldots,c_q) \mid \sum_{b_\nu\in I}t_{b_\nu}>1-\epslon,
    \text{ and }\right.\\
    \left.\forall b_\nu\in I\;\forall c_\mu\in (J\setminus I), t_{b_\nu}>t_{c_\mu}\right\}
\end{multline*}
\end{definition}

For a given parameter $0<\epslon''<1$, we let $V''(\epslon'')$ be the union of all $K_B(\epslon'')$ for $B$ a simplex of $\widehat{S}$ and $K$ a simplex of $\widehat{\Lambda}$ with $B\subset\clos{K}$.

We endow $\abs{\widehat{\Lambda}}$ with a cell structure adapted to $V''$ in a manner similar to that for $V$.

\begin{definition}\label{def:DecompV''}
For $K$ a simplex of $\widehat{\Lambda}$ with vertex set $\{v_0,\ldots, v_d\}$ and $0<\epslon<1$, let
\begin{multline*}
    \mc{C}_{K,\epslon}=\left\{C=L \cap \bigcap_{I\subset \{v_0,\ldots,v_d\}}C_{K,I,\epslon} \mid L\in \mc{C}_{K, \text{vert}},\right.\\
    \left.C_{K,I,\epslon}\in \mc{C}_{K, I, \epslon} \text{ for each } I, \text{ and } C\neq\emptyset\right\}
\end{multline*}
Let $\mc{C}_{\epslon}=\bigcup_{K\in \widehat{\Lambda}} \mc{C}_{K,\epslon}$.
\end{definition}

\begin{lemma}\label{thm:CWDecompAV''}
For any $0<\epslon<1$, $\mc{C}_{\epslon}$ gives a regular CW decomposition of $\abs{\widehat{\Lambda}}$ in which each cell is a convex polyhedron.
\end{lemma}
\begin{proof}
This proof is analogous to that of Lemma \ref{thm:CWDecompA}.
\end{proof}

\begin{lemma}\label{thm:SymmetryV''}
Let $0<\epslon''<1$. Given $B$ a simplex in $\widehat{S}$, let $U_B$ be the union of all sets $K_B(\epslon'')$ for $K$ a simplex of $\widehat{\Lambda}$ with $B\subset \clos{K}$. Then the family $\{U_B\mid B\in \widehat{\Lambda}\text{ is contained in }\widehat{S}\}$ is symmetric under the induced action of $G$, and hence the set $V''(\epslon'')=\bigcup_{B\in \widehat{S}} U_B$ is symmetric.
\end{lemma}
\begin{proof}
This is analogous to (and slightly simpler than) the proof of Proposition \ref{thm:SymmetryV}
\end{proof}

\begin{lemma}\label{thm:DecompV''}
Fix an $0<\epslon''<1$ and let $V''=V''(\epslon'')$. Then $U_B$ can be written as a union of cells of $\mathcal{C}_{\epslon''}$, and hence so can $V''$. Furthermore, the set $\{C\in \mathcal{C}_{\epslon''}\mid C\subset V''\}$ is symmetric under the action of $G$.
\end{lemma}
\begin{proof}
Again, this is analogous to the corresponding statement, Lemma \ref{thm:DecompV}.
\end{proof}

\begin{lemma}[ref \cite{gabrielov2009approximation} Lemma 5.3]\label{thm:EquivariantGabrielov5.3}
For $0<\epslon''<1$, there is an equivariant homotopy equivalence $\tau'': V''\rightarrow S$ inducing equivariant isomorphisms of homotopy and homology groups $\tau_{\#,k}'':\pi_k(V'',*'')\rightarrow \pi_k(S,*')$ and $\tau_{*,k}'':H_k(V'')\rightarrow H_k(S)$.
\end{lemma}
\begin{proof}
We apply Theorem \ref{thm:EquivariantNerveTheoremOpen}  to the covering of $V''$ by the family $\{U_B\}_{B\in \widehat{S}}$. Gabrielov and Vorobjov in the proof of Lemma 5.3 in \cite{gabrielov2009approximation} argue that any nonempty intersection $U_{B_1}\cap\ldots\cap U_{B_l}$ is contractible, and so part (ii) of Theorem \ref{thm:EquivariantNerveTheoremOpen} supplies an equivariant map $\psi_{V''}:V\rightarrow \abs{\mc{N}_{V''}}$ inducing a homotopy equivalence. Gabrielov and Vorobjov also assert that such an intersection $U_{B_1}\cap\ldots\cap U_{B_l}$ is nonempty iff $B_0,\ldots,B_l$, suitably reordered, form a flag of simplices of $\widehat{S}$. This means that there is an (equivariant) homeomorphism $\xi'':\abs{\mc{N}_{V''}}\rightarrow \abs{\Delta(\fpos{\widehat{S}})}$. By Proposition \ref{thm:BREquivariance} and the remark immediately following, we have an equivariant homeomorphism $\abs{\Delta(\fpos{\widehat{S}})}\rightarrow \abs{\br(\widehat{S})}$, and the (equivariant) inclusion $\abs{\br(\widehat{S})}\hookrightarrow \abs{\widehat{S}}=\Phi^{-1}(S)$ induces a homotopy equivalence. Hence, composing, we obtain the desired equivariant map $\tau'':V''\rightarrow S$, which induces equivariant isomorphisms $\tau_{\#,k}'':\pi_k(V'',*'')\rightarrow \pi_k(S,*')$ and $\tau_{*,k}'':H_k(V'')\rightarrow H_k(S)$ for all $k\geq 0$.
\end{proof}

We include the next statement from \cite{gabrielov2009approximation} for reference, though no claims of equivariance need be added here.

\begin{lemma}[\cite{gabrielov2009approximation} Lemma 5.4]\label{thm:Gabrielov5.4}
For $0<\epslon_0'\ll\cdots\ll \epslon_i'\ll\epslon_i\ll\delta_i\ll\delta_i'\ll\cdots\ll\delta_m'\ll\epslon''$, if  $V'=V(\epslon_0',\delta_0',\ldots, \epslon_m',\delta_m')$, $T=T(\epslon_0,\delta_0,\ldots, \epslon_m,\delta_m)$, and $V''=V''(\epslon'')$, then we have
\[
V'\subset \Phi^{-1}(T)\subset V''
\]
\end{lemma}

\begin{lemma}[\cite{gabrielov2009approximation} Lemma 5.6]\label{thm:Gabrielov5.6}
For $0<\epslon_0\ll\delta_0\ll\cdots\ll\epslon_m\ll\delta_m\ll\epslon''\ll 1$ and for every $k\leq m$, the inclusion $\zeta:\Phi^{-1}(T)\hookrightarrow V''$ induces equivariant epimorphisms
\[
\zeta_{\#,k}:\pi_k(T,*)\rightarrow \pi_k(V'',*'') \text{ and } \zeta_{*,k}:H_k(T)\rightarrow H_k(V'')
\]
\end{lemma}
\begin{proof}
That $\zeta_{\#,k}$ and $\zeta_{*,k}$ are epimorphisms is justified in \cite{gabrielov2009approximation}. That they are equivariant is clear.
\end{proof}

This brings us to the equivariant version of the main theorem of \cite{gabrielov2009approximation} for the definable case.

\begin{thm}[ref \cite{gabrielov2009approximation} Theorem 1.10(i)]\label{thm:EquivariantGabrielov1.10i}
For $0<\epslon_0\ll\delta_0\ll\cdots\ll \epslon_m\ll\delta_m\ll 1$ and every $0\leq k\leq m$, there is an equivariant map $\psi:T\rightarrow S$ inducing equivariant epimorphisms
    \begin{align*}
        \psi_{\#,k}&:\pi_k(T,*)\rightarrow\pi_k(S,*')\\
        \psi_{*,k}&:H_k(T)\rightarrow H_k(S)
    \end{align*}
    and in particular, $\rank(H_k(S))\leq \rank(H_k(T))$.
\end{thm}
\begin{proof}
This follows from Lemmas \ref{thm:EquivariantGabrielov5.3} and \ref{thm:Gabrielov5.6}, taking
\[
\psi=\tau''\circ \zeta:T\rightarrow S
\]
\end{proof}

\subsection{The Main Theorem: Separable Case}\label{sect:MainTheoremSeparableCase}

The proof of part (ii) of Theorem 1.10 in \cite{gabrielov2009approximation} requires that our family $\{S_\delta\}_{\delta>0}$ have an additional property referred to as separability. In particular, this property holds in the constructible case.

\begin{definition}[\cite{gabrielov2009approximation} Definition 5.7]\label{def:Gabrielov5.7}
For a family $\{S_\delta\}_{\delta>0}$ representing $S$ in $A$ and triangulation $(\Lambda,\Phi)$ of $A$, the pair $(\Lambda,\{S_\delta\}_{\delta>0})$ is called \emph{separable} if we have, for any pair $(\Delta_1,\Delta_2)$ of simplices of $S$ with $\Delta_1$ a subsimplex of $\Delta_2$,
\[
\paren{\clos{\Delta_2\cap \Phi^{-1}(S_\delta)}\cap \Delta_1=\emptyset}\Leftrightarrow \paren{\Delta_1\subset \clos{\Delta_2\setminus \Phi^{-1}(S_\delta)}}
\]
for all sufficiently small $\delta>0$.
\end{definition}

\begin{lemma}[\cite{gabrielov2009approximation} Lemma 5.8]\label{thm:Gabrielov5.8}
In the constructible case, $(\Lambda,\{S_\delta\}_{\delta>0})$ is separable.
\end{lemma}

So, for the remainder of the section, we will assume we are in the definable case but that $(\Lambda,\{S_\delta\}_{\delta>0})$ is separable, and all claims made will in particular apply to the constructible case. Provided with separability, we may define the hard/soft relation for a pair $(\Delta_1,\Delta_2)$ of simplices of $S$ with $\Delta_1$ a subsimplex of $\Delta_2$ by saying that $\Delta_1$ is a soft subsimplex of $\Delta_2$ if $\clos{\Delta_2\cap \Phi^{-1}(S_\delta)}\cap \Delta_1=\emptyset$ for all sufficiently small $\delta$, and $\Delta_1$ is a hard subsimplex of $\Delta_2$ otherwise.

Given sequences of parameters $(\epslon^{(j)},\delta^{(j)})=\epslon^{(j)}_0,\delta^{(j)}_0,\ldots, \epslon^{(j)}_m,\delta^{(j)}_m$, we will denote the sets defined relative to these parameters by $V^{(j)}=V(\epslon^{(j)},\delta^{(j)})$ and $T^{(j)}=T(\epslon^{(j)},\delta^{(j)})$. We cite and where necessary adjust a few results from \cite{gabrielov2009approximation} about relative containments of such sets.

\begin{lemma}[\cite{gabrielov2009approximation} Lemma 5.10]\label{thm:Gabrielov5.10}
For
\[
0<\epslon^{(1)}_0\ll\cdots\ll \epslon^{(1)}_i\ll\epslon^{(2)}_i\ll\delta^{(2)}_i\ll\delta^{(1)}_i\ll\cdots\ll\delta^{(1)}_m\ll 1
\]
$\Phi^{-1}(T^{(1)})\subset V^{(2)}$ and $V^{(1)}\subset \Phi^{-1}(T^{(2)})$.
\end{lemma}

\begin{lemma}[\cite{gabrielov2009approximation} Lemma 5.11]\label{thm:Gabrielov5.11}
For
\[
0<\epslon^{(1)}_0\ll\cdots\ll \epslon^{(1)}_i\ll\epslon^{(2)}_i\ll\delta^{(2)}_i\ll\delta^{(1)}_i\ll\cdots\ll\delta^{(1)}_m\ll 1
\]
the inclusion maps $T^{(1)}\hookrightarrow T^{(2)}$ and $V^{(1)}\hookrightarrow V^{(2)}$ are homotopy equivalences.
\end{lemma}

These inclusions therefore induce equivariant isomorphisms of homotopy and homology groups. At least for $T^{(1)}$ and $T^{(2)}$, we will want an equivariant map in the opposite direction.

\begin{lemma}[see \cite{gabrielov2009approximation} Lemma 5.11] \label{thm:EquivariantTtoT'}
For
\[
0<\epslon^{(1)}_0\ll\cdots\ll \epslon^{(1)}_i\ll\epslon^{(2)}_i\ll\delta^{(2)}_i\ll\delta^{(1)}_i\ll\cdots\ll\delta^{(1)}_m\ll 1
\]
there is an equivariant map $T^{(2)}\rightarrow T^{(1)}$ which is a homotopy inverse of the inclusion $T^{(1)}\hookrightarrow T^{(2)}$.
\end{lemma}
\begin{proof}
Gabrielov and Vorobjov in the proof of their Lemma 5.11 use Hardt Triviality to show that $T^{(1)}$ is a strong deformation retract of $T^{(2)}$. Equipped with our equivariant version of Hardt Triviality, we demonstrate that the homotopy they construct is equivariant.

Let $\bs{T}\subset \bb{R}^n\times \bb{R}^{2m+2}$ be the union of sets $T(\epslon_0,\delta_0,\ldots,\epslon_m,\delta_m)$ over $0<\epslon_i,\delta_i<1$, and let $\rho:\bs{T}\rightarrow \bb{R}^{2m+2}$ be the projection on the second coordinate. Then $\rho$ is a symmetric function, to which we may apply Theorem \ref{thm:EquivariantHardtTriviality}. We obtain a partition of $\bb{R}^{2m+2}$ into a finite number of definable sets $A_i$ over which $\rho$ is definably trivial. Subdividing further, we may assume each $A_i$ is connected (using the fact that sets definable over $\bb{R}$ have a finite number of connected components). Let $A_0$ be the element of this partition which contains both $(\epslon^{(1)},\delta^{(1)})$ and $(\epslon^{(2)},\delta^{(2)})$ for $0<\epslon^{(1)}_0\ll\cdots\ll \epslon^{(1)}_i\ll\epslon^{(2)}_i\ll\delta^{(2)}_i\ll\delta^{(1)}_i\ll\cdots\ll\delta^{(1)}_m\ll 1$. Then in particular, there is an equivariant map
\[
h:\rho^{-1}(A_0)\rightarrow \rho^{-1}(\epslon^{(2)},\delta^{(2)})\times A_0=T^{(2)}\times A_0
\]

Choose a definable simple curve $\gamma:[0,1]\rightarrow A_0$ with $\gamma(0)=(\epslon^{(2)},\delta^{(2)})$ and $\gamma(1)=(\epslon^{(1)},\delta^{(1)})$. This means $\rho^{-1}(\gamma([0,1]))$ is homeomorphic to $T^{(2)}\times \gamma([0,1])$ via an equivariant map. For any $0\leq t\leq t'\leq 1$, we can use this to define an equivariant homeomorphism $\Phi_{t,t'}:\rho^{-1}(\gamma(t'))\rightarrow \rho^{-1}(\gamma(t))$. After possibly adjusting the point $(\epslon^{(2)},\delta^{(2)})$ to assume that $\rho^{-1}(\gamma(t'))\subset \rho^{-1}(\gamma(t))$ for all $0\leq t\leq t'\leq 1$, Gabrielov and Vorobjov construct the following homotopy $F:T^{(2)}\times [0,1]\rightarrow T^{(2)}$.

Let $(\bs{x},t)\in T^{(2)}\times [0,1]$. If there is a $t'\leq t$ such that $\bs{x}\in \rho^{-1}(\gamma(t'))$ but $\bs{x}\not\in \rho^{-1}(\gamma(t''))$ for any $t''>t'$, let $F(\bs{x},t)=\Phi_{t',t}(\bs{x})$. Otherwise, let $F(\bs{x},t)=\bs{x}$. Since for any $0\leq t'\leq 1$, $\rho^{-1}(\gamma(t'))$ is symmetric, it follows that $F$ is equivariant. Then $F(-,1):T^{(2)}\rightarrow T^{(1)}$ is the desired equivariant homotopy inverse.
\end{proof}

If desired, we could prove that the inclusion $V^{(1)}\hookrightarrow V^{(2)}$ has an equivariant homotopy inverse in a similar manner.

Now, we are ready to relate $T$ and $V$.

\begin{thm}[ref \cite{gabrielov2009approximation} Theorem 5.12]\label{thm:EquivariantGabrielov5.12}
In the separable case, for $0<\epslon_0\ll\delta_0\ll\cdots\ll \epslon_m\ll\delta_m\ll 1$ and every $k\geq 0$, there is an equivariant map $\zeta:T\rightarrow V$ inducing equivariant isomorphisms $\zeta_{\#,k}:\pi_k(T)\rightarrow \pi_k(V)$ and $\zeta_{*,k}:H_k(T)\rightarrow H_k(V)$ (and hence a homotopy equivalence $T\simeq V$).
\end{thm}
\begin{proof}
To show the existence of isomorphisms $\zeta_{\#,k}$ and $\zeta_{*,k}$, \cite{gabrielov2009approximation} employs four different sequences $(\epslon^{(j)},\delta^{(j)})=\epslon_0^{(j)}, \delta_0^{(j)},\ldots, \epslon_m^{(j)}, \delta_m^{(j)}$. Lemma \ref{thm:Gabrielov5.10} gives us the chain of inclusions
\[
V^{(1)}\overset{\iota'}{\hookrightarrow} \Phi^{-1}(T^{(2)})\overset{\iota}{\hookrightarrow} V^{(3)}\overset{\iota''}{\hookrightarrow} \Phi^{-1}(T^{(4)})
\]
for $0<\epslon_0^{(j-1)}\ll\cdots\ll \epslon_i^{(j-1)}\ll\epslon_i^{(j)}\ll\delta_i^{(j)}\ll\delta_i^{(j-1)}\ll\cdots\ll\delta_m^{(j-1)}\ll 1$ (for $j=2,3,4$). The argument in \cite{gabrielov2009approximation} uses the fact that $\iota\circ \iota'$ and $\iota''\circ \iota$ are homotopy equivalences (which follows from Lemma \ref{thm:Gabrielov5.11}) to obtain that $\iota$ induces isomorphisms $\iota_{\#,k}:\pi_k(T^{(2)})\rightarrow \pi_k(V^{(3)})$ and hence also isomorphisms $\iota_{*,k}:H_k(T^{(2)})\rightarrow H_k(V^{(3)})$. Since $\iota$ is the inclusion map, the maps $\iota_
{\#,k}$ and $\iota_{*,k}$ are equivariant. Lemma \ref{thm:EquivariantTtoT'} gives an equivariant map $\eta:T^{(3)}\rightarrow T^{(2)}$ that induces a homotopy equivalence. Let $\zeta=\iota\circ \eta:T^{(3)}\rightarrow V^{(3)}$. Then $\zeta$ is equivariant and if we let
\begin{align*}
    \zeta_{\#,k}:&\pi_k(T^{(3)})\rightarrow \pi_k(V^{(3)})\\
    \zeta_{*,k}:&H_k(T^{(3)})\rightarrow H_k(V^{(3)})
\end{align*}
we have that $\zeta_{\#,k}$ and $\zeta_{*,k}$ are equivariant isomorphisms for all $k\geq 0$.
\end{proof}

Again, though $\zeta$ induces a homotopy equivalence $T^{(3)}\rightarrow V^{(3)}$ and a similar argument exchanging the roles of $T$ and $V$ would produce an equivariant map $V^{(3)}\rightarrow T^{(3)}$ that also induces a homotopy equivalence between these spaces, these maps are not necessarily homotopy inverses of one another.

\begin{cor}[ref \cite{gabrielov2009approximation} Thm 1.10(ii)]\label{thm:EquivariantGabrielov1.10ii}
In the separable (and so, in the constructible) case, for $0<\epslon_0\ll\delta_0\ll\cdots\ll \epslon_m\ll\delta_m\ll 1$, there is an equivariant map $\psi:T\rightarrow S$ inducing equivariant homomorphisms
\begin{align*}
    \psi_{\#,k}&: \pi_k(T,*)\rightarrow \pi_k(S,*')\\
    \psi_{*,k}&: H_k(T)\rightarrow H_k(S)
\end{align*}
which are isomorphisms for $1\leq k\leq m-1$ and epimorphisms for $k=m$. In particular, $\rank H_k(T)=\rank H_k(S)$ and if $m\geq \dim(S)$, $\psi$ induces a homotopy equivalence $T\simeq S$.
\end{cor}
\begin{proof}
Take $\psi=\tau\circ \zeta: T\rightarrow S$. The result follows from Theorem \ref{thm:EquivariantGabrielov4.8} and Theorem \ref{thm:EquivariantGabrielov5.12}.
\end{proof}

\subsection{Summary}\label{sect:Summary}
We set out all of the maps involved in constructing $\psi$ in both the definable and separable (so in particular, constructible) cases. In the diagrams, double headed arrows indicate homeomorphisms while pairs of arrows indicate homotopy equivalences. Dashed arrows indicate that we have not demonstrated the existence of an equivariant map in this direction. Hooked arrows indicate inclusion maps.

In the diagrams below, we let $m>0$ and $T=T(\epslon_0,\delta_0,\ldots, \epslon_m,\delta_m)$ for $0<\epslon_0\ll\delta_0\ll\cdots\ll\epslon_m\ll\delta_m\ll 1$. $\Phi$ denotes the homeomorphism of the triangulation $\Phi:\abs{\Lambda}\rightarrow A$ described at the beginning of Subsection \ref{sect:ConstructionV} and refined immediately before Subsection \ref{sect:MainTheoremDefinableCase}.

\noindent\textbf{Definable Case}
Let $V''=V''(\epslon)$ for any $\delta_m\ll \epslon''\ll 1$.
\[
\begin{tikzcd}
T\arrow[r, leftrightarrow, "\Phi^{-1}_{\restriction{T}}"] & \Phi^{-1}(T)\arrow[r, hook, "\zeta","\ref{thm:Gabrielov5.6}"'] &V''\arrow[r, bend left=50, "\psi_{V''}"]\arrow[r, phantom, "\ref{thm:EquivariantNerveTheoremOpen}(ii)" font=\tiny] &\abs{\mc{N}_{V''}}\arrow[l, dashed, bend left=50]
\arrow[r, leftrightarrow, "\xi''","\ref{thm:EquivariantGabrielov5.3}"']& \;
\end{tikzcd}
\]
\[
\begin{tikzcd}
\abs{\Delta(\fpos{\widehat{S}})}\arrow[r, leftrightarrow, "\ref{def:BarycentricRetraction}"'] &
\abs{\br(\widehat{S})}\arrow[r, hook, bend left=50] &\abs{\widehat{S}}=\Phi^{-1}(S)\arrow[l, bend left=50, "\ref{thm:BREquivariance}"']\arrow[r, leftrightarrow, "\Phi_{\restriction{S}}"] & S
\end{tikzcd}
\]

The map $\zeta$ induces epimorphisms of homotopy and homology groups for $k\leq m$.

\noindent\textbf{Separable/Constructible Case}
Let $S_r=S\cap\clos{B(0,r)}$ for some sufficiently large $r$, $V=V(\epslon_0,\delta_0,\ldots, \epslon_m,\delta_m)$, and for a sequence
\[
0<\epslon_0'\ll\cdots\ll \epslon_i'\ll\epslon_i\ll\delta_i\ll\delta_i'\ll\cdots\ll\delta_m'\ll 1
\]
let $T'=T(\epslon_0',\delta_0',\ldots, \epslon_m',\delta_m')$.

\[
\begin{tikzcd}
T\arrow[r, bend left=50, "r","\ref{thm:EquivariantTtoT'}"'] &T'\arrow[l, hook, bend left=50, "\ref{thm:Gabrielov5.11}"']\arrow[r, leftrightarrow, "\Phi^{-1}_{\restriction{T'}}"] & \Phi^{-1}(T')\arrow[r, hook, bend left=50, "\iota", "\ref{thm:Gabrielov5.10}"'] &V\arrow[l, dashed, bend left=50,"\ref{thm:EquivariantGabrielov5.12}"' near start]\arrow[r, "\psi_{V}", "\ref{thm:VandNV}"'] &\abs{\mc{N}_{V}}
\arrow[r, leftrightarrow, "\xi","\ref{thm:NVandNS}"']& \;
\end{tikzcd}
\]
\[
\begin{tikzcd}
\abs{\mc{N}_{\br(\widehat{S}_r)}}\arrow[r, bend left=50]\arrow[r,phantom, "\ref{thm:Hess2.5}" font=\tiny] &
\abs{\br(\widehat{S}_r)}\arrow[l, bend left=50]\arrow[r, hook, bend left=50] &\abs{\widehat{S}_r}=\Phi^{-1}(S_r)\arrow[l, bend left=50, "\ref{thm:BREquivariance}"']\arrow[r, leftrightarrow, "\Phi_{\restriction{S_r}}"] & S_r\arrow[r,hook, bend left=50]&S\arrow[l,dashed, bend left=50, "\ref{thm:ConicStructureAtInfty}"'near start]
\end{tikzcd}
\]

The map $\psi_V$ induces isomorphisms of the homotopy and homology groups of $T$ and $S$ for $0\leq k\leq m-1$ and epimorphisms for $k=m$. If $m\geq \dim(S)$, $\psi_V$ induces a homotopy equivalence.

Making use of the equivariant version of Hardt Triviality and with a bit of care, one may see that the map $S\rightarrow S_r$ is also equivariant. It seems likely that one could construct an equivariant homotopy inverse for $\psi_{V''}$ in the definable case or $\phi_{V}$ for large enough $m$ in the separable case. It might be interesting to investigate the existence of an equivariant homotopy inverse of $\iota$, but that goes far beyond the requirements of our application.

\section{Application to Cohomology of Symmetric Semialgebraic sets}\label{sect:Applications}

Basu and Riener in \cite{basu2021vandermonde} investigate the cohomology of semialgebraic sets defined by symmetric polynomials of bounded degree. They develop an algorithm for computing the first $l$ Betti numbers of such a set $S$, with complexity polynomially bounded in dimension, $n$, and number of symmetric polynomials, $s$. To accomplish this, Basu and Riener develop results on the structure of the cohomology groups of $\mc{P}$-closed semialgebraic sets, which allow for a significant reduction in the number of computations which must be performed. To extend these results to an arbitrary semialgebraic set, they apply the Gabrielov-Vorobjov construction. From the isomorphisms of homology groups, it follows that $S$ and its approximating set have the same Betti numbers (at least up to a chosen degree). However, the structural results rely on decomposing the cohomology spaces as $\s{n}$-modules. The Gabrielov-Vorobjov construction alone does not guarantee that the cohomology spaces of $S$ and its compact approximation have the same $\s{n}$-module structure, but our equivariant version does. Utilizing our new equivariance results, we strengthen a few results of Basu and Riener in \cite{basu2021vandermonde}.

Let $\bb{R}[X_1,\ldots, X_n]^{\s{n}}_{\leq d}$ denote the space of symmetric polynomials over $\bb{R}$ of degree at most $d$ (see \cite{basu2021vandermonde} Notation 4). By $H^k(S)$ we mean the $k$th cohomology space (with rational coefficients) of $S$. If $S$ is a symmetric semialgebraic set, then the action of $\s{n}$ on $S$ induces an action of $\s{n}$ on $H^*(S)$, giving $H^*(S)$ the structure of a finite dimensional $\s{n}$-module. As such, each $H^k(S)$ admits a decomposition into a direct sum of irreducible $\s{n}$-modules
\[
H^k(S)\cong_{\s{n}}\bigoplus_{\lambda\vdash n} m_{k,\lambda}(S)\; \bb{S}^{\lambda}
\]
where the sum is taken over all partitions $\lambda$ of the integer $n$ and $\bb{S}^\lambda$ is the particular irreducible $\s{n}$-module (the Specht module) corresponding to $\lambda$. The integer $m_{k,\lambda}(S)\in \bb{Z}_{\geq 0}$ is called the mulitplicity of $\bb{S}^\lambda$ in $H_k(S)$. The dimension of each $\bb{S}^\lambda$ may be computed (via what is known as the hook length formula), and so by this isotypic decomposition, we have reduced the task of computing the $k$th Betti number $b_k(S)=\dim(H^k(S))$ to that of computing the various mulitpicities $m_{k,\lambda}(S)$. Appendix 6 of \cite{basu2021vandermonde} summarizes the classical results from the represention theory of finite groups pertaining to the above decomposition. For $S$ a $\mc{P}$-closed set, Basu and Riener in \cite{basu2021vandermonde} Theorem 4 prove that the mulitpicities corresponding to partitions whose lengths are too long or too short must be zero. This dramatically reduces the number of multiplicities one needs to compute.

\begin{remark}
Though \cite{basu2021vandermonde} defines the Betti number $b_k(S)$ in terms of cohomology spaces rather than homology groups, in this setting it holds that $H_k(S)\cong_{\s{n}}H^k(S)$, and so the distinction may be set aside (see \cite{basu2020isotypic} Remark 1).
\end{remark}

Let $S$ be a $\mc{P}$-semialgebraic set for some finite $\mc{P}\subset \bb{R}[X_1,\ldots, X_n]^{\s{n}}_{\leq d}$, where $d\geq 2$. For algorithmic reasons, Basu and Riener reset the Gabrielov-Vorobjov results in terms of Puiseux series. Let $\bb{R}\gen{\epslon}$ be the real closed field of algebraic Puiseux series in $\epslon$ with coefficients in $\bb{R}$, and let
\[\bb{R}\gen{\epslon_m,\epslon_{m-1},\ldots,\epslon_0}=\bb{R}\gen{\epslon_m}\gen{\epslon_{m-1}}\cdots\gen{\epslon_0}
\]
(see \cite{basu2021vandermonde} Notation 15). Then in the unique ordering on $\bb{R}\gen{\epslon_m,\ldots, \epslon_0}$, we have $0<\epslon_0\ll\epslon_1\ll \ldots \ll \epslon_m$ (where here $\ll$ indeed denotes `infinitesimally smaller than'). In this phrasing, the equivariant Gabrielov-Vorobjov construction (Theorem \ref{thm:MainTheorem} part (ii)) gives us a $\mc{P'}$-closed and bounded semialgebraic set
\[
S'_m\subset \bb{R}\gen{\delta_m, \epslon_m, \delta_{m-1},\epslon_{m-1},\ldots, \delta_0,\epslon_0}^n
\]
and equivariant homomorphisms $\psi_{\#,k}:\pi_k(S'_m,*)\rightarrow \pi_k(S,*')$ and $\psi_{*,k}:H_k(S'_m)\rightarrow H_k(S)$ which are isomorphisms for $0\leq k\leq m-1$ and epimorphisms for $k=m$. By construction (see Section \ref{sect:GVConstruction}) we know that \[
\mc{P}'\subset\bb{R}\gen{\delta_m,\epslon_m,\ldots, \delta_0,\epslon_0}[X_1,\ldots, X_n]^{\s{n}}_{\leq d}
\]
and if $\mc{P}$ has cardinality $s$, $\mc{P}'$ has cardinality $4m(s+1)$. See \cite{basu2021vandermonde} Section 5.2 for more details on this rephrasing.

Because we have an isomorphism $H^k(S_m')\rightarrow H^k(S)$ for any $0\leq k\leq m-1$ which is equivariant relative to $\s{n}$, we know that these spaces have the same $\s{n}$-module structure and hence the same isotypic decomposition:
\[
\bigoplus_{\lambda\vdash n}m_{k,\lambda}(S)\bb{S}^\lambda\cong_{\s{n}}H^k(S)\cong_{\s{n}}H^k(S_m')\cong_{\s{n}}\bigoplus_{\lambda\vdash n} m_{k,\lambda}(S_m')\bb{S}^{\lambda}
\]
In particular, for each $0\leq k\leq m-1$ and $\lambda\vdash n$, $m_{k,\lambda}(S)=m_{k,\lambda}(S_m')$. This allows us to make the following strengthenings of two of Basu and Riener's theorems in \cite{basu2021vandermonde}.

Let $\Par{k}{S}=\{\lambda\vdash n\mid m_{k,\lambda}(S)\neq 0\}$ (see \cite{basu2021vandermonde} Notation 5).

\begin{cor}[Compare to \cite{basu2021vandermonde} Theorem 4]
Let $d,n\in \bb{Z}_{>0}$, $d\geq 2$, and let $S\subset \bb{R}^n$ be a $\mc{P}$-semialgebraic set with $\mc{P}\subset \bb{R}[X_1,\ldots, X_n]^{\s{n}}_{\leq d}$. Then, for all $\lambda \vdash n$,
\begin{enumerate}[(a)]
    \item
    \[
    m_{k,\lambda}=0 \text{ for } k\leq \length(\lambda)-2d+1
    \]
    or equivalently
    \[
    \max_{\lambda\in \Par{k}{S}}\length(\lambda) < k+2d -1
    \]
    \item
    \[
    m_{k,\lambda}=0 \text{ for } k\geq n-\length(^t\lambda)+d+1
    \]
    or equivalently
    \[
    \max_{\lambda\in \Par{k}{S}} \length(^t\lambda)<n-k+d+1
    \]
\end{enumerate}
\end{cor}

Note that Theorem 4 of \cite{basu2021vandermonde} required $S$ to be $\mathcal{P}$-closed. The proof of Theorem 4 in \cite{basu2021vandermonde} already uses the techniques discussed in Subsection \ref{sect:Boundedness} to replace an arbitrary $\mc{P}$-closed semialgebraic set by a bounded one that is equivariantly homotopy equivalent to the original set. We make use of the equivariant Gabrielov-Vorobjov construction to replace $S$ where $S'_m$ for $m=\dim(S)$. Since then $m_{k,\lambda}(S)=m_{k,\lambda}(S_m')$ for all $k\geq 0$, the result follows from applying \cite{basu2021vandermonde} Theorem 4 to $S_m'$.

\begin{cor}[Compare to Theorem 3 of \cite{basu2021vandermonde}]
Let $D$ be an ordered domain contained in $\bb{R}$, and let $l,d\geq 0$. There exists an algorithm which takes as input a finite set $\mathcal{P}\subset D[X_1,\ldots, X_n]^{\s{n}}_{\leq d}$ and a $\mathcal{P}$-formula $\mathcal{F}$, and computes the multiplicities $m_{k,\lambda}(S)$ for each $0\leq k\leq l$ and $\lambda\vdash n$, as well as the Betti numbers $b_k(S)$ for $0\leq k\leq l$, where $S$ is the realization of $\mathcal{F}$ in $\bb{R}^n$.

The complexity of this algorithm, measured by the number of arithmetic operations in $D$, is bounded by $(snd)^{2^{O(d+l)}}$. If $D=\bb{Z}$ and the bit-sizes of the coefficients of the input are bounded by $\tau$, then the bit-complexity of the algorithm is bounded by $(\tau s n d)^{2^{O(l+d)}}$.
\end{cor}

Note that Theorem 3 of \cite{basu2021vandermonde} only guaranteed the existence of an algorithm computing the first $l+1$ Betti numbers of $S$. The algorithm in question appears as Algorithm 3 in \cite{basu2021vandermonde}. Since we now have that $m_{k,\lambda}(S)=m_{k,\lambda}(S')$, the multiplicities computed in lines 15 and 19 of Algorithm 3 in \cite{basu2021vandermonde} are in fact the multiplicities for our original semialgebraic set. We need only assign a value of 0 to those multiplicities $m_{k,\lambda}$ with $\length(\lambda)>k+2d-1$ and then output the multiplicities as well as the Betti numbers.

\printbibliography
\end{document}